\documentclass[final]{amsart}
\usepackage[usenames,dvipsnames]{xcolor}
\definecolor{cite}{HTML}{11871E}
\definecolor{url}{HTML}{0851A6}
\definecolor{link}{HTML}{912F1B}

\usepackage[colorlinks=true, linkcolor=link, citecolor = cite, linktocpage, backref=page]{hyperref}

\usepackage[alphabetic]{amsrefs}

\usepackage{fancyhdr, amsmath, amsthm, enumitem, microtype}

\usepackage{todonotes}

\usepackage{tikz}
\usetikzlibrary{cd}
\usetikzlibrary{arrows, arrows.meta, positioning, calc}
\tikzcdset{arrow style=tikz, diagrams={>={Straight Barb[scale=0.65]}}}

\tikzstyle{arrow} = [-{Straight Barb[scale=0.8]}, line width=0.2mm]
\tikzstyle{fillOperator} = [double distance = 0.65em, line cap=round]

\newcommand{\supportptX}[2]{(2*\dimension-1)*(#1)-2*\dimension + 1+(#2)}

\tikzset{
	pics/xCoordtick/.style n args = {2}{
		code = {
			\draw (#1, -1pt) -- (#1, 1pt) #2;
		}
	}
}

\tikzset{
	pics/yCoordtick/.style n args = {2}{
		code = {
			\draw (1pt, #1) -- (-1pt, #1) #2;
		}
	}
}

\tikzset{
	pics/Coord/.style n args = {2}{
		code = {
			\draw[arrow] (0, 0) -- (#1, 0) node[below, at end]{$c$};
			\draw[arrow] (0, 0) -- (0, #2) node[left, at end]{$k$};
		}
	}
}

\tikzset{
	pics/Grid/.style n args = {2}{
		code = {
			\def\kOffset{#1}
			
			\def\firstend{\number\numexpr 2 * \dimension - 3 \relax}
			\def\newstart{\number\numexpr 2 * \dimension - 1 \relax}
			\def\actualend{\number\numexpr 2 * \dimension \relax}
			\def\supportptX##1##2{{(2*\dimension-1)*(##1)-2*\dimension + 1+(##2)}}
			
			\foreach \k in {\kOffset, ..., #2} {
				\foreach \l in {0, ..., \firstend, \newstart, \actualend} {
					\node at ({\supportptX{\k}{\l}}, \k) {$\times$};
				}
				
				\node at ({\supportptX{\k}{\firstend + 1}}, \k) {$\dots$};
			}
		}
	}
}

\tikzset{
	pics/MyLine/.style n args = {5}{
		code = {
			\def\k{#1}
			\def\l{#2}
			\def\kprime{#3}
			\def\lprime{#4}
			
			\draw ({\supportptX{\k}{\l}}, \k) -- ({\supportptX{\kprime}{\lprime}}, \kprime) #5;
		}
	}
}

\tikzset{
	pics/LineThruOrgn/.style n args = {4}{
		code = {
			\begin{scope}
				\def\k{#1}
				\def\l{#2}
				\node (O) at (0, 0) {};
				\node (T) at ({\supportptX{\k}{\l}}, \k) {};
				\draw (0, 0) -- ($(O)!#3cm!(T)$) #4;
			\end{scope}
		}
	}
}

\tikzset{
	pics/SuppLines/.style n args = {1}{
		code = {
			\draw pic {MyLine={1}{0}{#1 + 0.7}{0}{node[above, pos=0.95, opacity=1]{$L_B$}}};
			
			\draw pic {MyLine={0}{2*\dimension-1}{#1 + 0.38}{2*\dimension-1}{node[above, pos=0.97, opacity=1]{$L_1$}}};
			
			\draw pic {MyLine={0}{2*\dimension}{#1 + 0.3}{2*\dimension}{node[above, pos=0.97, opacity=1]{$L_T$}}};
		}
	}
}

\tikzset{
	pics/FillRow/.style n args = {1}{
		code = {
			\draw[fillOperator] ({\supportptX{#1}{0}}, #1) -- ({\supportptX{#1}{2 * \dimension}}, #1);
		}
	}
}

\tikzset{
	pics/FillFirstLine/.style n args = {1}{
		code = {
			\draw[fillOperator] ({\supportptX{1}{0}}, 1) -- ({\supportptX{#1}{0}}, #1);
		}
	}
}

\tikzset{
	pics/FillPoint/.style n args = {2}{
		code = {
			\draw[fillOperator] ({\supportptX{#1}{#2}}, #1) -- ({\supportptX{#1}{#2}}, #1);
		}
	}
}

\usepackage[bitstream-charter]{mathdesign}
\usepackage[T1]{fontenc}
\usepackage{stmaryrd}

\DeclareMathAlphabet{\eur}{U}{zeus}{m}{n}
\newcommand{\matheur}[1]{\eur{#1}}

\usepackage[letterpaper, textheight=8.56in, textwidth=5.95in, twoside, centering]{geometry}
\theoremstyle{plain}
\newtheorem{prop}[subsubsection]{Proposition}
\newtheorem{lem}[subsubsection]{Lemma}
\newtheorem{cor}[subsubsection]{Corollary}
\newtheorem{thm}[subsubsection]{Theorem}
\newtheorem*{thm*}{Theorem}

\theoremstyle{definition}
\newtheorem{defn}[subsubsection]{Definition}
\newtheorem{notation}[subsubsection]{Notation}

\theoremstyle{remark}
\newtheorem{rmk}[subsubsection]{Remark}
\newtheorem*{rmk*}{Remark}

\newcommand{\teq}{\addtocounter{subsubsection}{1}\tag{\thesubsubsection}}

\newcommand{\arrdisp}{0.33ex}
\newcommand{\arrdisplacementsp}{0.72ex}

\DeclareMathOperator{\addUnit}{addUnit}
\DeclareMathOperator{\Alt}{Alt}

\DeclareMathOperator{\aug}{au}

\DeclareMathOperator{\car}{char}

\DeclareMathOperator{\coChev}{coChev}

\DeclareMathOperator{\coFib}{coFib}

\DeclareMathOperator{\coFree}{coFree}

\DeclareMathOperator{\coLie}{coLie}

\DeclareMathOperator*{\colim}{colim}
\DeclareMathOperator{\cone}{cone}

\DeclareMathOperator{\coPrim}{coPrim}

\DeclareMathOperator{\Cinfty}{C_\infty}

\DeclareMathOperator{\ComAlg}{ComAlg}

\DeclareMathOperator{\ComAlgstar}{\ComAlg^\star}
\DeclareMathOperator{\ComCoAlg}{ComCoAlg}

\newcommand{\cont}{\mathrm{cont}}
\DeclareMathOperator{\Corr}{Corr}

\newcommand{\devissage}{d\'evissage}

\newcommand{\DGCat}{\mathrm{DGCat}}

\newcommand{\DGCatprescont}{\DGCat_{\pres, \cont}}

\newcommand{\disj}{\mathrm{disj}}

\DeclareMathOperator{\ev}{ev}

\DeclareMathOperator{\Fact}{Fact}
\newcommand{\Factstar}{\Fact^\star}
\DeclareMathOperator{\FB}{FB}
\newcommand{\FBplus}{\FB_+}
\DeclareMathOperator{\FI}{FI}
\DeclareMathOperator{\Fib}{Fib}

\DeclareMathOperator{\FS}{FS}
\DeclareMathOperator{\FSplus}{\FS_+}

\newcommand{\Fqbar}{\lbar{\mathbb{F}}_q}
\DeclareMathOperator{\Free}{Free}

\DeclareMathOperator{\Fun}{Fun}

\DeclareMathOperator{\GL}{GL}

\newcommand{\HHo}{\mathbf{H}}
\newcommand{\Ho}{\mathrm{H}}
\DeclareMathOperator{\Hom}{Hom}
\newcommand{\id}{\mathrm{id}}
\DeclareMathOperator{\im}{im}
\newcommand{\iso}{\mathrm{iso}}

\newcommand{\Kunneth}{K\"unneth}
\DeclareMathOperator{\Linfty}{L_\infty}
\DeclareMathOperator{\Lie}{Lie}

\newcommand{\Mod}{\mathrm{Mod}}

\newcommand{\op}{\mathrm{op}}

\newcommand{\otimesshriek}{\overset{!}{\otimes}}
\newcommand{\otimesstar}{\overset{\star}{\otimes}}
\newcommand{\otimesstard}{\overset{*}{\otimes}}

\DeclareMathOperator{\PConf}{PConf}

\newcommand{\poly}{\mathrm{poly}}
\newcommand{\pres}{\mathrm{pres}}
\newcommand{\PreStk}{\mathrm{PreStk}}

\newcommand{\pseudoproper}{\mathrm{ps.p.}}
\newcommand{\pt}{\mathrm{pt}}

\newcommand{\Ql}{\mathbb{Q}_\ell}
\newcommand{\Qlbar}{\lbar{\mathbb{Q}}_\ell}

\DeclareMathOperator{\Ran}{Ran}

\DeclareMathOperator{\Rep}{Rep}

\DeclareMathOperator{\Sym}{Sym}

\newcommand{\Sch}{\mathrm{Sch}}

\DeclareMathOperator{\Spec}{Spec}
\newcommand{\surjects}{\twoheadrightarrow}
\DeclareMathOperator{\Shv}{Shv}

\newcommand{\Spc}{\mathrm{Spc}}

\DeclareMathOperator{\cTop}{\matheur{T}op}
\DeclareMathOperator{\TopTriv}{T}

\DeclareMathOperator{\tr}{tr}

\newcommand{\un}{\mathrm{un}}
\newcommand{\union}{\mathrm{union}}

\newcommand{\Vect}{\mathrm{Vect}}

\newcommand{\ardis}{\ar@<\arrdisp>}
\newcommand{\ardissp}{\ar@<\arrdisplacementsp>}

\newcommand{\lbar}[1]{\overline{#1}}

\newcommand{\oversetsupscript}[3]{\overset{#2}{#1}{}^{#3}}

\newenvironment{myenum}[1]{%
  \begin{enumerate}[label=#1,topsep=1pt,itemsep=0pt,partopsep=1pt,parsep=1pt]%
    }%
    {
  \end{enumerate}%
}

\newcommand{\alg}[1]{\matheur{A}_{#1}}
\newcommand{\algOp}[1]{\matheur{O}_{#1}}
\newcommand{\algGen}[1]{\matheur{G}_{#1}}
\newcommand{\liealg}[1]{\mathfrak{a}_{#1}}
\newcommand{\liealgOp}[1]{\mathfrak{o}_{#1}}
\newcommand{\liealgGen}[1]{\mathfrak{g}_{#1}}

\newcommand{\algOr}[1]{\mathsf{A}_{#1}}
\newcommand{\freeAlgOr}{\mathsf{A}}
\newcommand{\liealgOr}[1]{\mathsf{a}_{#1}}
\newcommand{\trivLieAlgOr}{\mathsf{a}}

\newcommand{\sOmega}{\vec{\omega}}

\title[Higher representation stability for ordered configuration spaces]{Higher representation stability for ordered configuration spaces}
\author{Quoc P. Ho}
\address{Department of Mathematics, The Hong Kong University of Science and Technology (HKUST), Clear Water Bay, Hong Kong}
\email{phuquocvn@gmail.com}
\date{\today}


\begin{document}

\begin{abstract}
  Using factorization homology with coefficients in twisted commutative algebras (TCAs), we prove two flavors of higher representation stability for the cohomology of (generalized) configuration spaces of a scheme/topological space $X$. First, we provide an iterative procedure to study higher representation stability using actions coming from the cohomology of $X$ and prove that all the modules involved are finitely generated over the corresponding TCAs. More quantitatively, we compute explicit bounds for the derived indecomposables in the sense of Galatius--Kupers--Randal-Williams. Secondly, when certain $C_\infty$-operations on the cohomology of $X$ vanish, we prove that the cohomology of its configuration spaces forms a free module over a TCA built out of the configuration spaces of the affine space. This generalizes a result of Church--Ellenberg--Farb on the freeness of $\FI$-modules arising from the cohomology of configuration spaces of open manifolds and, moreover, resolves the various conjectures of Miller--Wilson under these conditions.
\end{abstract}

\maketitle
\tableofcontents

\section{Introduction}

\subsection{Motivation}
Let $\{X_n\}_{n\in \mathbb{N}}$ be a sequence of spaces/schemes that arises, in some sense, naturally. Such a sequence is said to exhibit \emph{homological stability} if there are naturally constructed maps $\Ho_i(X_n) \to \Ho_i(X_{n+1})$ which are isomorphisms when $n\gg i$. Homological stability is a ubiquitous phenomenon and plays an important role in many fields of mathematics, going back to Quillen~\cites{quillen_notes_1971}, McDuff, and Segal~\cites{mcduff_homology_1976,segal_topology_1979}, to name a few. The last ten years have witnessed an explosion of activities in the subject. This is in part due to the newly discovered connection between homological stability phenomena and problems in arithmetic statistics, pioneered by Ellenberg--Venkatesh--Westerland, as well as the influx of many new techniques from algebraic topology and homotopy theory, see, for example,~\cites{ellenberg_homological_2016,vakil_discriminants_2015, ellenberg_fox-neuwirth-fuks_2017,farb_coincidences_2019,ho_homological_2021,kupers_$e_n$-cell_2018, galatius_cellular_2018, galatius_$e_2$-cells_2019,randal-williams_homology_2019}.

It is realized by Church--Farb in~\cite{church_representation_2013}  that while some sequences of spaces, such as the sequence of ordered configuration spaces of a manifold, do not exhibit homological stability, they do display a stability phenomenon up to some group action, which is the symmetric group action in the case of ordered configuration spaces. This so-called \emph{representation stability}, too, turns out to be a ubiquitous phenomenon, appearing in a plethora of settings, from topology, group cohomology, to arithmetic statistics, for instance~\cite{church_homological_2012,church_representation_2014,church_fi-modules_2015,putman_stability_2015,putman_representation_2017}.

Intriguingly, within the range where homological stability fails, the failure itself exhibits stability and we obtain, so to speak, \emph{secondary homological stability}, a phenomenon discovered in a breakthrough paper~\cite{galatius_$e_2$-cells_2019}. Roughly speaking, if we let $\Ho'_*(X_n)$ be the difference between $\Ho_*(X_n)$ and $\Ho_*(X_{n-1})$, then we can construct natural maps $\Ho'_*(X_n) \to \Ho'_{*+h}(X_{n+k})$ for some fixed $h$ and $k$ such that for any $i$, $\Ho'_{i+th}(X_{n+tk}) \to \Ho'_{i+(t+1)h}(X_{n+(t+1)k})$ is an isomorphism when $t\gg 0$. Observe that compared to the case of usual homological stability, secondary stability can relate homology groups of different degrees (i.e. $h>0$) and between spaces that are further apart in the sequence (i.e. $k>1$).

It is a natural question to ask if~\emph{secondary representation stability} also appears~\emph{in nature}. In~\cite{miller_higher_2016}, it is discovered that this is indeed the case for ordered configuration spaces of surfaces with a boundary. It is thus tempting to ask the following questions:
\begin{myenum}{(\roman*)}
\item Do ordered configuration spaces of manifolds/schemes of higher dimensions satisfy secondary representation stability?
\item Within the range where secondary representation stability fails, does the failure satisfy some sort of stability? In other words, do ordered configuration spaces exhibit \emph{tertiary} and even \emph{higher order representation stability}?
\item Can we provide explicit ranges for higher representation stability?
\item What about generalized ordered configuration spaces?
\end{myenum}
Of these questions, the first three were stated as conjectures in~\cite{miller_higher_2016}*{\S3.6}.

\begin{rmk*}
  Progresses in these directions have also recently been made in~\cite{bibby_generating_2019}. One important feature which distinguishes our approach from theirs is that for rational coefficients, we are able to deduce higher representation stability results even in cases where one cannot completely compute the cohomology. See~\S\ref{subsubsec:compare_to_Bibby_Gadish} for a more complete comparison.
\end{rmk*}

\subsection{A sample theorem}
Before describing the general goal of the paper, to fix ideas and to whet the reader's appetite, let us record here one simple and concrete consequence of our work. For a manifold $X$, we let $P_2^k(X) = \oversetsupscript{X}{\circ}{k}$ be the ordered configuration space of $k$ distinct points on $X$, i.e. the open complement of all the diagonals $x_i=x_j$ in $X^k$. Then, for each cohomological degree $p$, the rational cohomology group $\Ho^p(P^k_2(X))$ admits a natural action of the symmetric group $S_k$ by permuting the factors. In fact, $\bigoplus_k \Ho^*(P^k_2(X))$ has a natural structure of a twisted commutative algebra (see~\S\ref{subsec:FB_twisted_comalg} for a review) and hence, in particular, admits actions by multiplying with $\Ho^0(X)$ and $\Ho^1(X)$.

\begin{thm} \label{thm:intro:sample}
  Let $X$ be a connected non-compact finite type manifold of dimension at least $4$. Then, there exists a function $r: \mathbb{N}_0 \to \mathbb{N}_0$ tending to infinity such that for $k > p-r(k)$, the images under the actions of $\Ho^0(X)$ and $\Ho^1(X)$ on $\Ho^p(P^{k-1}_2(X))$ and $\Ho^{p-1}(P^{k-1}_2(X))$ together span $\Ho^p(P^k_2(X))$ up to the action of the symmetric group $S_k$.
\end{thm}

This theorem is best understood in comparison with the following result.

\begin{thm}[\cite{church_fi-modules_2015}*{Theorem 6.4.3}, \cite{miller_higher_2016}*{Theorems 3.12 and 3.27}]
  Let $M$ be a noncompact connected manifold of dimension $d>2$. Then, for $k>p$, the image under the action of $\Ho^0(X)$ on $\Ho^p(P^{k-1}_2(X))$ spans $\Ho^p(P^k_2(X))$ up to the action of the symmetric group $S_k$.
\end{thm}

Here, we see that within the range where representation stability fails, there is a range in which the failure itself exhibits stability. The range improvement is measured by the function $r(k)$.

\subsection{The goal of this paper}

Theorem~\ref{thm:intro:sample} above is a combination two facts. Firstly, by~\cite{church_fi-modules_2015}, for each cohomological degree $p$, $\bigoplus_k \Ho^p(P^k_2(X))$ is a finitely generated free module over the free twisted commutative algebra generated by $\Ho^0(X)$, or equivalently, it is a finitely generated free $\FI$-module. Moreover, we know the range of the generators. Secondly, the space of generators, formulated in a precise sense using $\FI$-homology, is a finitely generated module over the free twisted commutative algebra generated by $\Ho^1(X)[-1]$. This is a special case of a type of higher representation stability statement we prove.

In this paper, we provide generalizations of these two statements which go beyond the smooth case, which work both in the topological and algebro-geometric settings, and which apply also to generalized configuration spaces. We refer to these as the two different flavors of higher stability. Together, they provide answers to questions posed above in many cases.

\subsubsection{First flavor of higher representation stability}
We prove that when $X$ satisfies $\TopTriv_{m}$ (see Definition~\ref{defn:TopTriv_m_condition}), the cohomology groups of ordered configuration spaces of $X$ form a free module over an explicit free twisted commutative algebra which is built out of the cohomology of ordered configuration spaces of $\mathbb{A}^{\dim X}$ and which provides operators witnessing (higher) representation stability of degrees from $1$ up to $m$. Moreover, we provide explicit bounds on the degrees of generators of the module. The answers to the first three questions i.--iii. follow, in this case, as immediate consequences. Note also that in this case, instead of an iterative procedure, representation stability from degrees $1$ to $m$ is seen simultaneously.

It is known, see~\cite{church_fi-modules_2015}, that when $X$ satisfies $\TopTriv_1$, which is equivalent to saying that $X$ is non-compact, the cohomology groups of ordered configuration spaces of $X$ form a free $\FI$-module. In the language of twisted commutative algebras of~\cite{sam_introduction_2012}, they form a free module over the free twisted commutative algebra generated by one element in graded degree $1$ and cohomological degree $0$. Our result could thus be viewed as a generalization of this statement. Note also that the twisted algebra used in~\cite{miller_higher_2016} to witness secondary representation stability is the degree $2$ part of ours.

\subsubsection{Second flavor of higher representation stability}
In general (i.e. when the module involved is not free), representation stability, or higher variants thereof, could be understood as the module involved being finitely generated. This is usually established in two steps. First, one proves that the relevant twisted commutative algebra involved is Noetherian. Second, one constructs a spectral sequence converging to our module whose first/second page satisfies the finite generation condition. The Noetherian condition then guarantees that finite generation persists and implies the same condition on page infinity. In some cases, finer analysis yields more explicit bounds on the degrees of the generators, establishing explicit stability range. One example of particular interest to us is~\cite{miller_$mathrmfi$-hyperhomology_2019} where the authors establish explicit bounds on the $\FI$-hyperhomology and then use a result of~\cite{gan_linear_2019} to produce explicit stability range.

With this in mind, we provide a generalization of the first step of~\cite{miller_$mathrmfi$-hyperhomology_2019}. More specifically, we give an iterative procedure to construct new operators which witness higher representation stability (without any extra condition on $X$) for \emph{generalized} ordered configuration spaces of $X$. Then, using these operators, we formulate higher analogs of $\FI$-hyperhomology, which fits nicely into the framework of derived indecomposables of~\cite{galatius_cellular_2018}, and provide explicit bounds in these cases. Note that the operators involved here are different from the one mentioned above in that they involve the cohomology of $X$ itself. The results we prove in this paper could thus be viewed as the analog of~\cite{miller_higher_2016}*{Conjecture 3.31} for these operators.

We expect that once analogs of~\cite{gan_linear_2019} are available for more general twisted commutative algebras, our results could be used to yield explicit bounds for higher representation stability in the general case. In this paper, using Noetherianity of free twisted commutative/anti-commutative algebras generated in degree $1$, we provide a qualitative version. Namely, we show that all the modules that appear in the iterative procedure described above are finitely generated with respect to higher representation actions.

\subsubsection{Factorization homology of twisted commutative algebras}
While the main goal of the paper is to understand higher representation stability, a large part of the paper is devoted to the introduction and study of factorization homology and Koszul duality for twisted commutative factorization algebras, a subject of independent interest. Using these techniques, we obtain expressions for the (rational Borel--Moore) homology\footnote{\label{ftn:renormalized_Borel_Moore}To keep the terminology less cluttered, we will, from now on, loosely use the term cohomology to refer to a renormalized version of Borel--Moore homology, i.e. sheaf cohomology with coefficients in the renormalized dualizing sheaf, see Notation~\ref{not:renormalized_dualizing_sheaf}.} of generalized ordered configuration spaces of $X$ in terms of Lie cohomology of some twisted Lie algebras.\footnote{Technically, we use the Koszul duality between commutative algebras and $\coLie$-coalgebras.} The two flavors of representation stability presented above are direct consequences of this analysis.

\subsection{Relation to other work}
\subsubsection{Relation to Miller-Wilson~\cite{miller_higher_2016}}
Our work is closely related to, and is inspired by, the work of Miller--Wilson~\cite{miller_higher_2016}, where, to our knowledge, the notion of secondary representation stability first appears in the literature. There, the authors introduce a pair of distinct points at the boundary of a $2$-dimensional manifold $M$ with boundary to get a map
\[
  \PConf_2 \mathbb{R}^2 \times \PConf_n M \to \PConf_{n+2}M,
\]
where $\PConf_n M = M^n \setminus \Delta$ is the ordered configuration space of $n$ distinct points on $M$. This induces an operator
\[
  \Ho_1(\PConf_2 \mathbb{R}^2) \otimes \Ho_*(\PConf_n M) \to \Ho_{*+1}(\PConf_{n+2} M)
\]
used to formulate and prove secondary stability.

In contrast, our actions, obtained algebraically, are defined for higher dimensional spaces (that are not necessarily manifolds) and provide examples for higher representation stability beyond the secondary case. However, already for secondary stability, we need the extra condition $\TopTriv_2$ (see Definition~\ref{defn:TopTriv_m_condition}) to obtain the kind of action~\cite{miller_higher_2016} constructed. On the other hand, we are able to prove certain free-ness result that is not available without this extra assumption, generalizing the result of~\cite{church_fi-modules_2015}.

\subsubsection{Relation to Bibby--Gadish~\cite{bibby_generating_2019}} \label{subsubsec:compare_to_Bibby_Gadish}
The paper~\cite{bibby_generating_2019} appeared while we were writing up the current paper. While the methods employed are completely different, the applications we have for configuration spaces share many similar features with the main results of~\cite{bibby_generating_2019}. Let us now comment on the main differences.

On the one hand, \cite{bibby_generating_2019} works with coefficients in a Noetherian ring (with some extra assumption on the cohomology of $X$ to make sure that \Kunneth{} formula holds), whereas our coefficients are always fields of characteristic $0$. Moreover, \cite{bibby_generating_2019} works in the equivariant setting and proves results about \emph{orbit configuration spaces} whereas we concern ourselves with only the trivial group case. We expect, however, that a relatively straightforward adaptation of our method to the equivariant setting will yield similar results for the case of orbit generalized configuration spaces as well.

On the other hand, our method works uniformly for both ordered configuration spaces and the generalized variants. Most importantly, unlike their work, we are able to deduce representation stability results even in situations where one cannot completely compute all the cohomology. Indeed, the main higher representation stability results~\cite{bibby_generating_2019}*{Theorem C, Theorem 4.4.2} require the space involved to be $i$-acyclic, leaving the more general cases open~\cite{bibby_generating_2019}*{Conjectures 4.3.11 and 4.4.3}. Here, the $i$-acyclicity condition ensures that a certain spectral sequence involved collapses, allowing one to completely compute all the cohomology groups. Working systematically with dg-algebras and $\Linfty$-algebras, we are able to go beyond the $i$-acyclic cases, settling various questions left open in~\cite{bibby_generating_2019}, such as Conjectures 4.3.11 and 4.4.3 as well as Remark 3.7.5.ii. Moreover, we are able to establish precise bounds for the derived indecomposables (i.e. higher analogs of $\FI$-hyperhomology) with respect to actions used to witness higher representation stability.

\subsubsection{Relation to Petersen~\cite{petersen_cohomology_2020}}
In this paper, we give an expression for the cohomology of generalized configuration spaces in terms of cohomological Chevalley complex of a certain dg-$\Lie$-algebra. This could be viewed as a categorical and model-free construction of the functor $\matheur{C}\matheur{F}$ appearing in~\cite{petersen_cohomology_2020}*{Theorem 1.15} when working over a field in characteristic $0$.

\subsection{An outline of the results}
\label{subsec:intro_outline_results}
While the paper's main goal is to understand higher representation stability of generalized ordered configurations spaces, most of the paper concerns itself with the development of a theory of factorization homology with coefficients in a twisted commutative algebra, which is of independent interest. In fact, configuration spaces do not make an appearance until the last section~\S\ref{sec:rep_stab}. The results about higher representation stability could thus be viewed as simple applications of the theory. We will therefore split this subsection into two parts, corresponding to the general theory~\S\ref{intro:subsubsec:results_gen_theory} and its applications~\S\ref{intro:subsubsec:results_applied_conf} respectively.

To keep the terminology and numerology uniform, we will, in this paper, restrict ourselves to the case where $X$ is a scheme even though the proofs also carry over to more general topological spaces, see also~\S\ref{subsubsec:use_scheme_as_default}. Throughout, we will also assume that $X$ is an equi-dimensional of dimension $d\geq 1$. We will use $\pi: X\to \pt = \Spec k$ to denote the structure map.

\subsubsection{Factorization (co)homology with coefficients in a twisted commutative algebra}
\label{intro:subsubsec:results_gen_theory}
Compared to factorization homology with coefficients in a commutative algebra, factorization homology with coefficients in a twisted commutative algebra gives rise to a twisted commutative algebra.\footnote{See~\S\ref{subsec:FB_twisted_comalg} for a quick overview of twisted commutative algebras.} Namely, the latter keeps track of the symmetric group $S_n$-actions. To set up the theory, we introduce a twisted commutative version of the $\Ran$ space (or more precisely, prestack, see~\S\ref{subsec:prestacks_and_sheaves}), denoted by $\Ran(X, \FBplus)$, playing the role of the usual (graded) $\Ran$ space in this new setting.

The space $\Ran(X, \FBplus)$ is easy to describe
\[
  \Ran(X, \FBplus) \simeq \bigsqcup_{k>0} X^k/S_k,
\]
where the quotient is to be understood in the sense of groupoid quotient. It is thus equipped with a natural morphism $\Ran(\pi)$ to the constant prestack
\[
  \Ran(\pt, \FBplus) = \FBplus = \bigsqcup_{k>0} \pt/S_k.
\]
Here, we view the groupoid $\FBplus$ as a constant prestack.

Observe that sheaves on $\Ran(X, \FBplus)$ are collections of $S_k$-equivariant sheaves on $X^k$ and similarly, sheaves on $\FBplus$ are collections of $S_k$-representations. Pushing forward along $\Ran(\pi)$, $\Ran(\pi)_!$, amounts to taking cohomology of $X^k$ while keeping track of all the $S_k$-actions. While simple, this is a very convenient way to package everything in terms of sheaf theory.

Following~\cite{beilinson_chiral_2004, francis_chiral_2011,gaitsgory_weils_2014}, we construct a convolution monoidal structure $\otimesstar$ on the category of sheaves $\Shv(\Ran(X, \FBplus))$, consider the category of commutative algebra objects on it $\ComAlgstar(\Ran(X, \FB_+))$, formulate the factorizability condition, and define the category twisted commutative factorizable algebras $\Factstar(X, \FBplus)$ as the full sub-category of $\ComAlgstar(\Ran(X, \FBplus))$ consisting of objects satisfying the factorizability condition. By factorization homology, we simply mean applying $\Ran(\pi)_!$ to an object in $\Factstar(X, \FBplus)$ and obtain an object in $\ComAlg(\Vect^{\FBplus})$, a twisted commutative algebra.\footnote{There is also a cohomological variant, which we will not elaborate in the introduction.}

Combinatorially, $\Ran(X, \FBplus)$ is more complicated than the usual $\Ran$ space, making several proofs more complicated than the corresponding ones for the usual $\Ran$ space. Modulo these technical complications, however, the general outlines of these proofs follow the same spirit as the usual case, and in the end, results regarding factorization algebras and homology are transported to the twisted settings in this paper.  Let us highlight the two main results that are important for the applications we have in mind. Note that these two results are well-known in the non-twisted setting.

The first one, Theorem~\ref{intro:thm:tcfa_vs_tcaX} below, says that twisted commutative factorization algebras are simply twisted commutative algebra objects over $X$, and moreover, the equivalence is given by restricting along the diagonal.

\begin{thm}[Theorem~\ref{thm:tcfa_vs_tcaX}] \label{intro:thm:tcfa_vs_tcaX}
  The functor
  \[
    \delta^!: \ComAlgstar(\Ran(X, \FBplus)) \to
    \ComAlg(\Shv(X)^{\FBplus})
  \]
  admits a fully-faithful left adjoint $\Fact$, whose essential image is $\Factstar(X, \FBplus)$. In particular, we have the following equivalence of categories
  \[
    \Fact: \ComAlg(\Shv(X)^{\FBplus}) \rightleftarrows \Factstar(X, \FBplus): \delta^!.
  \]
\end{thm}

Thanks to this theorem, we can speak of taking factorization homology/cohomology of a twisted commutative algebra on $X$. In what follows, we will denote these functors as $\pi_?$ and $\pi_{?*}$ respectively.

The main computational tool we use to analyze the output of factorization (co)homology is Koszul duality between twisted commutative algebras and twisted $\coLie$-coalgebras.\footnote{See~\S\ref{subsec:Koszul_duality} for a quick review of Koszul duality. As we only use very formal aspects of the theory, the readers who are unfamiliar with Koszul duality can take these statements as a blackbox without losing the gists of the arguments carried out in the paper.} To that end, the second main result we want to highlight shows that Koszul duality is compatible with factorization (co)homology.

\begin{thm}[Theorem~\ref{thm:pi_?*_vs_coChev}] \label{intro:thm:pi_?*_vs_coChev}
  We have a natural equivalence of functors
  \[
    \pi_{?*} \circ \coChev \simeq \coChev \circ \pi_*: \coLie(\Shv(X)^{\FBplus}) \to \ComAlg(\Vect^{\FBplus}).
  \]
\end{thm}

\subsubsection{Applications to generalized ordered configuration spaces}
\label{intro:subsubsec:results_applied_conf}
We will start by introducing the notation we use throughout the paper regarding generalized ordered configuration spaces. For any finite set $I$ or non-negative integer $k$, we write $P^I(X) = X^I$ and $P^k(X) = X^k$; when $I=\emptyset$ or $k=0$, the corresponding space is, by convention, just a point.

For any integer $n \geq 2$, we let $P^I_n(X)$ be the open subspace of $X^I$ consisting of ordered configurations of points on $X$ such that the multiplicity of each point is less than $n$. In other words, it is the space of maps $I \to X$ such that the fiber of any point is of size at most $n-1$. We define $P^k_n(X)$ in a similar way.

We will also use $P(X)$ and $P_n(X)$ to put them all in a family, i.e.
\[
  P(X) = \bigsqcup_{k\geq 0} P^k(X) = \bigsqcup_{k\geq 0} X^k \quad\text{and}\quad P_n(X) = \bigsqcup_{k\geq 0} P^k_n(X). \teq\label{eq:powers_of_X}
\]
and the ``non-unital'' variants $P(X)_+$ and $P_n(X)_+$ where we only take the union over positive $k$'s. We refer to these as generalized ordered configuration spaces of $X$. When $n=2$, this is precisely the usual ordered configuration space of $X$.

\subsubsection{}
For each $n\geq 2$, consider the augmented unital $n$-truncated twisted commutative algebra $\algOr{n} = \Lambda[x]/x^n$ and its augmentation ideal $\algOr{n, +}$, where $x$ lives graded degree $1$, cohomological degree $2d$, and Tate twist $d$. The starting point is the following result which connects the cohomology of $P_n(X)$ and factorization cohomology.

\begin{prop}[Proposition~\ref{prop:fact_coh_vs_pconf}]
  The factorization cohomology of $X$ with coefficients in the twisted commutative algebra $\pi^! (\algOr{n, +})$, denoted by $\alg{n, +}(X)$, computes the cohomology of $P_n(X)_+$.\footnote{See Footnote~\ref{ftn:renormalized_Borel_Moore}.} We will also use $\alg{n}(X)$, which is an augmented unital twisted commutative algebra, to denote the cohomology of $P_n(X)$.\footnote{Due to some technical reasons, when Koszul duality is concerned, it is more convenient to work in the non-unital setting. On the other hand, if we think about modules over an algebra, it is more convenient to work unitally. Fortunately, passing between non-unital and augmented unital settings is only a formal procedure involving either formally adding a unit or taking the augmentation ideal, see also~\S\ref{subsubsec:unit_and_augmentation}. We will pass between the two settings freely in the introduction.}
\end{prop}

Let $\liealgOr{n}$ be the twisted $\coLie$-coalgebra which is the Koszul dual of $\algOr{n}$. The compatibility between Koszul duality and factorization cohomology gives a new expression of the cohomology of generalized configuration spaces in terms of Lie cohomology and we arrive at the following result.

\begin{thm}[Theorem~\ref{thm:coh_conf_vs_coChev}]
  \label{intro:thm:coh_conf_vs_coChev}
  When $n\geq 2$, we have the following equivalence
  \[
    \alg{n, +}(X) \simeq \coChev(C^*(X, \omega_X) \otimes \liealgOr{n})
  \]
  and the corresponding unital version
  \[
    \alg{n}(X) \simeq \coChev^{\un}(C^*(X, \omega_X) \otimes \liealgOr{n}).
  \]
  In other words, $\liealg{n}(X) := C^*(X, \omega_X) \otimes \liealgOr{n}$ is the Koszul dual of $\alg{n, +}(X)$.
\end{thm}

\begin{rmk}
  If one works only with ordered configuration spaces of topological manifolds, a similar statement readily follows from the general theory of topological factorization homology, which is more flexible than what's available in the algebro-geometric setting; see~\cite{knudsen_higher_2018}*{Proof of Theorem C}. For ordered configuration spaces of smooth manifolds, a similar result was also obtained in~\cite{getzler_homology_1999}*{\S2}.
\end{rmk}

Roughly speaking, $\coChev$ behaves like $\Sym$, the functor of taking free (non-unital) twisted commutative algebra. Thus, knowledge about which cohomological and graded degrees $\liealgOr{n}$ is supported translates directly to that of $\alg{n, +}(X)$.

The situation is particularly simple when $n=2$. Indeed, $\liealgOr{2}$ is a free twisted $\coLie$-coalgebra generated by a one dimensional vector space sitting at graded degree $1$ and cohomological degree $2d-1$, and hence, in particular, at each graded degree, $\liealgOr{2}$ is concentrated in one cohomological degree, see Lemma~\ref{lem:a_2_is_free}. For $\liealgOr{n}$ in general, we only know some bound on where it is supported, see Lemma~\ref{lem:homological_est_a_n}. This purely algebraic question is the main obstacle for why we can say more about the cohomology of ordered configuration spaces than those of its generalized variants. For the rest of the introduction, to keep it readable and less notation heavy, we will only state the results for ordered configuration spaces and merely reference the body of the text for results concerning the generalized variants.

\subsubsection{} The expression in Theorem~\ref{intro:thm:coh_conf_vs_coChev} allows us to produce operators on $\alg{n}(X)$ in a convenient way. These operators are then used to formulate (higher) representation stability. The idea of the procedure is sketched in more details in~\S\ref{subsec:sketch_higher_rep_stab}. But roughly speaking, for each morphism of twisted $\coLie$-coalgebras $\mathfrak{o} \to \liealgOr{n}(X)$, we obtain, via Koszul duality, a morphism of twisted commutative algebras $\coChev^{\un} \mathfrak{o} \to \alg{n}(X)$. In this paper, the $\mathfrak{o}$ involved is always trivial (i.e. it has trivial $\coLie$-coalgebra structure), and hence, the map above is simply $\Sym \mathfrak{o}[-1] \to \alg{n}(X)$. In particular, $\alg{n}(X)$ is naturally equipped with the structure of a module over a free twisted commutative algebra $\Sym \mathfrak{o}[-1]$.

For example, when $X$ is irreducible, we can take $\mathfrak{o} = \Ho^{-1}(\liealg{n}(X))_1[1] = \Lambda_1[1]$, the piece of $\liealg{n}(X)$ sitting in graded degree $1$ and cohomological degree $-1$. The discussion above implies that $\alg{n}(X)$ has a natural structure as a module over $\Sym \Lambda_1 = \Lambda[x]$. In other words, $\alg{n}(X)$ is an $\FI$-module. Moreover, the procedure of taking $\FI$-hyperhomology corresponds to taking the relative (derived) tensor
\[
  \Lambda \otimes_{\Lambda[x]} \alg{n}(X) \simeq \coChev^{\un} (0 \sqcup_{\liealg{n}(X)} \Lambda[1]_1) = \coChev^{\un}(\cone(\Lambda[1]_1 \to \liealg{n}(X))). \teq\label{intro:eq:hyper_FI_homology}
\]

In general, $\mathfrak{o}$ is constructed by extracting certain graded and cohomological degree pieces of $\liealg{n}(X)$. The analog of $\FI$-hyperhomology, the so-called \emph{derived indecomposables} of a module in the language of~\cite{galatius_cellular_2018}, is defined via a relative tensor as the above. From~\eqref{intro:eq:hyper_FI_homology}, we see that in the $\coLie$-side, taking derived indecomposables corresponds to simply ``deleting'' certain graded and cohomological degree pieces from $\liealg{n}(X)$, and hence, its effect on the support of the final result can be read off in a straightforward manner.

The whole procedure itself can be iterated. Namely, after taking the derived indecomposables with respect to the action induced by $\mathfrak{o}$, we can find another $\mathfrak{o}'$ and continue as before. In this paper, higher representation stability comes in two flavors, corresponding to, roughly speaking, actions induced by cohomology of $X$ and actions induced by cohomology of ordered configuration spaces of $\mathbb{A}^{\dim X}$.

\subsubsection{The first flavor}
For the first flavor, we have the following

\begin{thm}[Corollary~\ref{cor:factoring_out_first_line_k=1}] \label{intro:cor:factoring_out_first_line_k=1}
  Let $X$ be an equidimensional scheme of dimension $d$, $c_0 \geq 0$ an integer, and $b=\min\left(\frac{2d-1}{2}, c_0 + 1\right)$. Let $\alg{2}^{(c_0)}(X)$ be the object obtained from iteratively taking derived indecomposables with respect to actions on $\alg{2}(X)$ by
  \[
    \Sym (\Ho^0(X, \sOmega_X)_1), \Sym (\Ho^1(X, \sOmega_X)[-1]_1), \dots, \Sym (\Ho^{c_0}(X, \sOmega_X)[-c_0]_1).
  \]
  Then, $\Ho^c(\alg{2}^{(c_0)}(X))_k \simeq 0$ when $c < bk$. In general, $\Ho^c(\alg{2}^{(c_0)}(X))_k$ is finite dimensional.
\end{thm}

In fact, this theorem is only the graded degree $1$ part of a more general procedure described in~\S\ref{subsubsec:iterative_procedure} and culminated in Theorem~\ref{thm:iterative_procedure_conf}. For manifolds, the case where $c_0 = 0$ was obtained in~\cite{miller_$mathrmfi$-hyperhomology_2019}. Moreover, for generalized configuration spaces, one has a similar result, see Proposition~\ref{prop:first_flavor_gen_conf}.

In the case where $C^*_c(X, \Lambda)$ has trivial multiplication structure (i.e. all multiplication operations vanish), we show, in Proposition~\ref{prop:spectral_seq_collapse_triv_alg}, that a simplification occurs and we get
\[
  \Ho^*(\alg{n}(X)) \simeq \Sym(\Ho^*(\liealg{n}(X))[-1]),
\]
as twisted commutative algebras, resolving the question stated in~\cite{bibby_generating_2019}*{Rmk 3.7.5.ii}. This implies that all modules involved are free over the corresponding twisted commutative algebra. As a result, taking derived indecomposables, which is an inherently derived construction, can be computed for each cohomological degree separately.\footnote{We will not, however, emphasize this point of view as the the re-indexing involved would be too much of a distraction. See also the discussion at~\S\ref{subsubsec:first_flavor_i-acyclic}}

Using this simplification, we obtain the a more traditionally sounding representation stability statement.

\begin{cor}[Corollary~\ref{cor:conf_spaces_k=1_i-acyclic}]
  Let $X$ be an equidimensional scheme of dimension $d$ such that the multiplication structure on $C^*_c(X, \Lambda)$ is trivial (e.g. when $X$ is $i$-acyclic), and $c_0, i \geq 0$ be integers such that $c_0 < \frac{2d-1}{2}$. Then, $\bigoplus_{t=0}^\infty \Ho^{i + tc_0}(\alg{2}^{(c_0-1)}(X))_{1+t}$ is a finitely generated free module over $\Sym \Ho^{c_0}(X, \sOmega_X)_1$ (resp. $\Alt \Ho^{c_0}(X, \sOmega_X)_1$) when $c_0$ is even (resp. odd), where $\alg{2}^{(c_0-1)}$ is defined in Corollary~\ref{cor:factoring_out_first_line_k=1}. Moreover, the generators live in graded degrees $k$ with
  \begin{myenum}{(\roman*)}
  \item $1\leq k\leq i-c_0$, when $c_0 < d-1$, and
  \item $1\leq k\leq 2(i-d+1)$, when $c_0 = d-1$.
  \end{myenum}
\end{cor}

As above, similar statements for generalized ordered configuration spaces can also be obtained, albeit with more complicated expressions for the bounds, and are worked out in Proposition~\ref{prop:first_flavor_gen_conf_degree_wise}.

\begin{rmk}
  By~\cite{petersen_cohomology_2020}*{Thm. 1.20}, $i$-acyclic spaces, i.e. those with vanishing maps $\Ho^i_c(X, \Lambda) \to \Ho^i(X, \Lambda), \forall i$, form a large class of examples of spaces where the multiplication structure on $C^*_c(X, \Lambda)$ is trivial. Examples of $i$-acyclic spaces include oriented manifolds $M$ with vanishing cup products on $\Ho^*_c(M, \Lambda)$ or spaces of the form $X \times \mathbb{A}^1$. More examples can be found in \cite{arabia_espaces_2018}*{Prop. 1.2.4}.
\end{rmk}

In general, one hopes that a precise bound for derived indecomposables could be used to yield explicit bounds for higher representation stability. However, little is known beyond the $\FI$ case, see~\cite{gan_linear_2019,miller_$mathrmfi$-hyperhomology_2019}. In this paper, we content ourselves with a more qualitative result. Indeed, a similar analysis as the above coupled with the Noetherian results of~\cite{sam_grobner_2016} yields the following finite generation statement. Since the statement for generalized configuration spaces is not more complicated than the usual one, we include it here as well. Replacing $n=2$ in the theorem below recovers the statement for ordered configuration spaces, which settles the non-equivariant case of~\cite{bibby_generating_2019}*{Conjecture 4.3.11}.

\begin{thm}[Theorem~\ref{thm:finite_generation_higher_rep_stab}]
  \label{intro:thm:finite_generation_higher_rep_stab}
  Let $X$ be an equidimensional scheme of dimension $d$, and $c_0, i\geq 0, n\geq 2$ be integers such that $c_0 < \frac{(2d-1)(n-1)}{n}$. Then $\bigoplus_{t=0}^\infty \Ho^{i+tc_0}(\alg{n}^{(c_0-1)}(X))_{t+1}$ is a finitely generated module over $\Sym \Ho^{c_0}(X, \sOmega_X)_1$ (resp. $\Alt \Ho^{c_0}(X, \sOmega_X)_1$) when $c_0$ is even (resp. odd).

  More concretely, when $t\gg 0$, the map
  \[
    \Ho^{c_0}(X, \sOmega_X)_1 \otimes \Ho^{i+(t-1)c_0}(\alg{n}^{(c_0-1)}(X))_{t} \to \Ho^{i+tc_0}(\alg{n}^{(c_0-1)}(X))_{t+1}
  \]
  is surjective. Here, the tensor on the LHS is the one in $\Vect^{\FB}$.
\end{thm}

\begin{rmk}
  When $X$ is non-compact, $\alg{2}(X)$ is a free $\FI$-module, as we will see below. The derived indecomposables $\Ho^*(\alg{2}^{(0)}(X))$ can thus be obtained by taking usual $\FI$-homology. In this case, the statement above is hence about the finite generation of $\FI$-homology with respect to the secondary stability action of $\Ho^1(X, \sOmega)_1$.
\end{rmk}

\subsubsection{The second flavor} We will now describe the second flavor, which was first introduced in~\cite{miller_higher_2016}. For this flavor, we only work with configuration spaces rather than the generalized variants. Moreover, we need an extra condition on $X$, the $\TopTriv_m$ condition, see Definition~\ref{defn:TopTriv_m_condition}. Roughly speaking, it has to do with requiring certain multiplications on $C^*_c(X, \Lambda)$ vanish.

We prove the following theorem, which, as a consequence (see Corollary~\ref{cor:miller_wilson_conj}) resolves~\cite{miller_higher_2016}*{Conjecture~3.31} in the case where $X$ satisfies $\TopTriv_m$.

\begin{thm}[Theorem~\ref{thm:stability_Tm}]
  Let $X$ be an irreducible scheme of dimension $d$ which satisfies $\TopTriv_m$ for some $m\geq 1$. Then the cohomology of the ordered configuration spaces of $X$, $\alg{2}(X) = \bigoplus_{k=0}^\infty C^*(P^k_2(X), \sOmega_{P^k_2(X)})$, is a free module over the free twisted commutative algebra
  \[
    \algOp{2, \leq m}(X) = \bigotimes_{k=1}^m \Sym(\Ho^{(2d-1)(k-1)}(P^k_2(\mathbb{A}^d))[-(2d-1)(k-1)]_k),
  \]
  generated by $\algGen{2, \leq m}(X) = \coChev^{\un}(\liealgGen{2, \leq m}(X))$, with $\liealgGen{2, \leq m}(X)$ given in Proposition~\ref{prop:splitting_n=2}, i.e.
  \[
    \alg{2}(X) \simeq \algOp{2, \leq m}(X) \otimes \algGen{2, \leq m}(X).
  \]
  Moreover, if $s\geq 1$ be such that $\Ho^k(X, \sOmega_X) = 0$ for all $k \in (0, s)$ and
  \[
    b = \min\left(\frac{(2d-1)m}{m+1}, s\right),
  \]
  then the support of $\algGen{2, \leq m}(X)$ lies below (and possibly including) the line $c=bk$ where $c$ and $k$ denote cohomological and graded degrees respectively. In other words, $\Ho^c(\algGen{2, \leq m}(X))$ lives in graded degrees $\leq c/b$.
\end{thm}

When $m=1$, this theorem is a generalization of a result by~\cite{church_fi-modules_2015} which says that when $X$ is a non-compact manifold, $\alg{2}(X)$ is a free $\FI$-module. Moreover, this theorem also provide precise bounds on the generator in this case (see Corollary~\ref{cor:T1_image_spanned}), generalizing~\cite{miller_higher_2016}*{Theorem 3.27} beyond the case where $X$ is a manifold.

When $m=2$, the condition $\TopTriv_2$ corresponds to the \emph{acyclic diagonal} condition of~\cite{bibby_generating_2019} and the theorem above strengthens~\cite{bibby_generating_2019}*{Theorem 4.4.2} beyond the $i$-acyclic case, settling~\cite{bibby_generating_2019}*{Conjecture 4.4.3} in the non-equivariant setting.

\subsection{An outline of the paper}
In~\S\ref{sec:prelim}, we set up the necessary notation and review results used throughout the paper. Then, in~\S\ref{sec:tcfas}, we introduce the twisted commutative $\Ran$ space and develop the theory of twisted commutative factorization algebras (tcfa's). These algebras serve as the coefficient objects for our version of factorization (co)homology, developed in~\S\ref{sec:factorization_coh}. Finally, in the last section~\S\ref{sec:rep_stab}, we apply the theory developed so far to problems regarding higher representation stability of generalized configuration spaces.

\section{Preliminaries}
\label{sec:prelim}

We will set up the necessary notation and review results used throughout the paper.

\subsection{Notation and conventions}
\subsubsection{Categories} \label{subsubsec:conventions_categories}
We will use $\DGCatprescont$ to denote the infinity category whose objects are presentable stable infinity categories linear over $\Lambda$, a field of characteristics $0$, and whose morphisms are continuous functors, i.e. functors that commute with colimits. The main references for the subject are~\cite{lurie_higher_2017-1, lurie_higher_2017, gaitsgory_study_2017}.

$\DGCatprescont$ is the natural setting where adjoint functor theorem holds~\cite{lurie_higher_2017-1}*{Corollary 5.5.2.9} and where the theory of factorization homology is developed in the literature~\cite{francis_chiral_2011,gaitsgory_weils_2014,gaitsgory_atiyah-bott_2015}. In some sense, it is the correct setting to perform universal constructions in homological algebra. In this paper, all ``linear categories,'' such as the categories of chain complexes of vector spaces or of sheaves, are viewed as objects in $\DGCatprescont$, and, unless otherwise specified, functors between them are assumed to be continuous. In particular, we assume that the tensor operations are continuous in each variable in all symmetric monoidal categories we consider. Moreover, all operations, e.g. push-forward, relative tensors etc., are assumed to be derived by default.

We should note that the theory of presentable stable infinity categories appears in the paper in an essential but superficial way: we use it only as much as to be able to cite the various results proved in the literature which are built on this technical foundation. Thus, the readers who are not familiar with the theory could ignore this technical point, and pretend that the categories involved are triangulated, without losing much of the gist of the paper.

\subsubsection{Geometry}
\label{subsubsec:prelim_geometry}
The arguments used in this paper take as input a category of \emph{spaces} $\matheur{S}$ together with a sheaf theory
\[
  (S \in \matheur{S}) \rightsquigarrow (\Shv(S) \in \DGCatprescont)
\]
equipped with a six-functor formalism.

For example, if we are interested in schemes over some algebraically closed ground field $k$, we take $\matheur{S} = \Sch$ to be the category of separated schemes of finite type over $k$ and
\begin{myenum}{(\alph*)}
\item when $k=\mathbb{C}$, we take $\Lambda$ to be any field of characteristic $0$, and for any $S\in \Sch$, $\Shv(S)$ is the ind-completion of the category of constructible sheaves on $S$ with $\Lambda$-coefficients,
\item when $k$ is any algebraically closed field in general, we take $\Lambda = \Ql, \Qlbar$, $\ell \neq \car k$, and for any $S\in \Sch$, $\Shv(S)$ is the ind-completion of the category of constructible $\ell$-adic sheaves on $S$ with $\Lambda$-coefficients.
\end{myenum}
See~\cite{gaitsgory_weils_2014}*{\S4}, \cite{liu_enhanced_2012,liu_enhanced_2014} for an account of the theory. We will also use the notation $\pt = \Spec k$ to denote the final object of $\Sch$.

If we are interested in topological spaces, we can take $\matheur{S} = \cTop$, the category of topological spaces, and the theory of sheaves to be the one developed in~\cite{schnurer_six_2018}.

\subsubsection{} \label{subsubsec:use_scheme_as_default} To keep the exposition streamlined, however, we opt to take schemes over $k = \Fqbar$ along with the theory of $\ell$-adic sheaves as our default language. One reason is that we want to emphasize the fact that our results are, by construction, compatible with the Frobenius actions. We expect that the topologically minded readers can easily extract the desired topological statements by ignoring the Tate twists in the cohomology groups and by taking into account the fact that $d$-dimensional varieties are ``secretly $2d$-dimensional'' as far as $\ell$-adic cohomology is concerned.

\subsection{Prestacks and sheaves on prestacks}
\label{subsec:prestacks_and_sheaves}
As mentioned in the introduction, the theory sheaves on prestacks provides a convenient and flexible framework to deal with equivariant sheaves and to keep track of $S_n$-actions. This theory has been developed extensively in~\cite{gaitsgory_weils_2014, gaitsgory_atiyah-bott_2015}. In this subsection, we will give a quick summary of the theory.

\subsubsection{Prestacks}
A prestack $\matheur{Y}$ is a contravariant functor $\matheur{Y}: \Sch^{\op} \to \Spc$, where $\Spc$ denotes the $\infty$-category of spaces, i.e. $\infty$-groupoids. The pre-stacks that we are interested in this paper in fact only land in the category of ($1$-)groupoids. Directly from the definition, we see that any such $\matheur{Y}$ can be tautologically written as a colimit of schemes $\matheur{Y} = \colim_{S\in \Sch_{/\matheur{Y}}} S$.

Most of the prestacks appearing in this paper are quotient prestacks. Let $X\in \Sch$ equipped with an action of a group $G$. Then, the prestack quotient $X/G$ has the following functor of points description: for each test scheme $S$, $(X/G)(S) = X(S)/G$. Here, the quotient on the right is in the groupoid sense. Namely, for any set $T$ with an action of a group $G$, $T/G$ is a groupoid whose set of objects is $T$ and for any $t_1, t_2 \in T/G$,
\[
  \Hom_{T/G}(t_1, t_2) = \{g\in G: gt_1 = t_2\}.
\]
$X^n/S_n$ and $\pt/S_n$ are two main examples of this construction that appear in the paper.

\subsubsection{Sheaves on prestacks}
The theory of sheaves on schemes provides us with a contravariant functor
\[
  \Shv: \Sch^{\op} \to \DGCatprescont,
\]
where we use the $!$-pullback functor to move between schemes. Right Kan extending $\Shv$ along the Yoneda embedding $\Sch^{\op} \hookrightarrow \PreStk^{\op}$, we obtain a functor
\[
  \Shv: \PreStk^{\op} \to \DGCatprescont. \teq\label{eq:shv_prestk_contra}
\]

By abstract nonsense, this functor preserves limits, i.e. $\Shv(\colim_i \matheur{Y}_i) \simeq \lim_i \Shv(\matheur{Y}_i)$. In particular, $\Shv(\matheur{Y}) \simeq \lim_{S \in (\Sch_{/\matheur{Y}})^{\op}} \Shv(S)$.
Unwinding the definition, we see that informally, each $\matheur{F} \in \Shv(\matheur{Y})$ corresponds to the following data
\begin{myenum}{(\roman*)}
\item a sheaf $\matheur{F}_{S, y} \in \Shv(S)$ for each $S\in \Sch$ and $y \in \matheur{Y}(S)$, i.e. $y: S \to \matheur{Y}$, and
\item an equivalence of sheaves $\matheur{F}_{S', f(y)} \to f^! \matheur{F}_{S, y}$ for each morphism of schemes $f: S \to S'$.
\end{myenum}

When $\matheur{Y} = \colim_i Y_i$ where $Y_i\in \Sch$, then $\Shv(\matheur{Y}) = \lim_i \Shv(Y_i)$ and we have a similar description as the above except that we only need to take $S$ to be $Y_i$'s. In particular, for any $X\in \Sch$, $\Shv(X^n/S_n)$ is the category of $S_n$-equivariant sheaves on $X^n$. When $X = \pt$, $\Shv(\pt/S_n)$ is precisely the category of representations of $S_n$.\footnote{Note that by our convention, this is the derived category of representations of $S_n$.}

\subsubsection{$f_! \dashv f^!$}
From~\eqref{eq:shv_prestk_contra}, we see that for any $f: \matheur{Y}_1 \to \matheur{Y}_2$, we obtain a functor $f^!: \Shv(\matheur{Y}_2) \to \Shv(\matheur{Y}_1)$ which can be shown to commute with both limits and colimits. It thus has a left adjoint, denoted by $f_!$, and we obtain a functor $\Shv_!: \PreStk \to \DGCatprescont$.  In the paper, if we want to emphasize the contravariant aspect of $\Shv$, we will also use the notation $\Shv^!$.

For $\matheur{Y} \in \PreStk$ with the structure map $f: \matheur{Y} \to \pt$ and $\matheur{F} \in \Shv(\matheur{Y})$, we write $C^*_c(\matheur{Y}, \matheur{F}) = f_! \matheur{F} \in \Shv(\pt) = \Vect$.

\subsubsection{$f^* \dashv f_*$}
When a morphism $f: \matheur{Y}_1 \to \matheur{Y}_2$ is schematic, i.e. any pullback to a scheme is a scheme, then we can also construct a pair of adjoint functors $f^* \dashv f_*$. Moreover, the functor $f_*$ satisfies base change with respect to the $(-)^!$ functor, see~\cite{ho_free_2017}*{\S2.6}.

\subsubsection{Monoidal structure}
Let $\matheur{F}_i \in \Shv(\matheur{Y}_i)$ where $\matheur{Y}_i \in \PreStk$, $i\in \{1, 2\}$. Then, $\matheur{F}_1 \boxtimes \matheur{F}_2 \in \Shv(\matheur{Y}_1 \times \matheur{Y}_2)$ is characterized by the condition that for each $f_i: S_i \to \matheur{Y}_i, i \in \{1, 2\}$ with $S_i$'s being schemes, we have
\[
  (f_1 \times f_2)^! (\matheur{F}_1 \boxtimes \matheur{F}_2) \simeq f_1^! \matheur{F}_1 \boxtimes f_2^! \matheur{F}_2.
\]

As for schemes, $!$-restricting along the diagonal map allows us to define $\otimesshriek$, and when the diagonal is schematic, $*$-restricting gives $\otimesstard$.

\subsection{$\FB$ and twisted commutative algebras (tca's)}
\label{subsec:FB_twisted_comalg}
In this subsection, we recall the notion of an $\FB$-object in a category. Historically, representation stability was formulated for $\FI$-objects by Church, Ellenberg, Farb in~\cite{church_fi-modules_2015} and was later understood to be an example of a category of modules over a twisted commutative algebra~\cite{sam_introduction_2012}, i.e. a commutative algebra in the category of $\FB$-objects. For the purpose of this paper, we will exclusively use the latter formulation as it is more convenient in our case.

\subsubsection{The category $\FB$}
Let $\FB$ be the category whose objects are finite sets and morphisms are bijections between finite sets. This category is equipped with a natural symmetric monoidal structure with disjoint union being the monoidal operation and the empty set $\emptyset$ being the unit. For each $n\in \mathbb{N}$, we fix, once and for all, an object $[n] \in \FB$ which is a linearly ordered set of $n$ elements. For each $n_1, n_2 \in \mathbb{N}$, we have a canonical map
\[
  [n_1] \sqcup [n_2] \simeq [n_1 + n_2].
\]
Unless confusion is likely to occur, we will sometimes write $n$ in place of $[n]$, as in $S_n$, the symmetric group of $n$ letters, and $X^n$, the $n$-th power of $X$, rather than $S_{[n]}$ and $X^{[n]}$. The ordering on $[n]$ gives canonical isomorphisms between $S_n$ and $S_{[n]}$, as well as $X^n$ and $X^{[n]}$. Moreover, such a choice of $[n]$'s gives us an equivalence of categories
\[
  \FB \simeq \bigsqcup_{n \geq 0} BS_n.
\]

We will use $\FBplus$ to denote the full-subcategory of $\FB$ containing only non-empty finite sets. Note that $\FBplus$ is a non-unital symmetric monoidal subcategory of $\FB$.

For future usage, we will also use $\FS$ to denote the category of finite sets together with surjections as morphisms. Moreover, $\FSplus$ will be used to denote the full-subcategory consisting of non-empty finite sets.

\subsubsection{$\FB$-objects}
\label{subsubsec:FB_objs}
Recall the following construction. For any symmetric monoidal categories $\matheur{I}$ and $\matheur{C}$ with $\matheur{C}$ being co-complete, the category of functors $\matheur{C}^\matheur{I} = \Fun(\matheur{I}, \matheur{C})$ is equipped with a symmetric monoidal structure given by Day convolution (see~\cite{lurie_higher_2017-1}*{\S2.2.6}). More explicitly, given $V, W \in \matheur{C}^\matheur{I}$, we have the following diagram
\[
  \begin{tikzcd}
    \matheur{I} \times \matheur{I} \ar{dr} \ar{r}{(V, W)} \ar{d}{\otimes}  & \matheur{C} \times \matheur{C} \arrow{d}{\otimes} \\
    \matheur{I} \ar[dashed]{r}{V\otimes W} & \matheur{C}
  \end{tikzcd}
\]
where $V\otimes W$ is defined to be the left Kan extension of the diagonal arrow along the vertical arrow on the left.

In this paper, we are mainly interested in the case where $\matheur{I}$ is either $\FB$ or $\FBplus$ and $\matheur{C}$ is either $\Vect$ or $\Shv(X)$ for some space $X$. Objects in these categories will be referred to as $\FB$-objects. Moreover, tensors of $\FB$-objects are always understood to be in the sense above.

Given an object $c\in \matheur{C}^{\FB}$, we will write $c_I \in \Rep_{S_I}(\matheur{C}) = \matheur{C}^{BS_I}$ to denote its $I$-th component. In particular, we have $c \simeq \bigoplus_{I \in \FB/\iso} c_I$. We will also use the phrase \emph{graded degrees} to refer to the cardinality $|I|$ of $I$, e.g. $c_I$ lives in graded degree $|I|$.

For each object $c\in \Rep_{S_n}(\matheur{C})$, we will write $c_n \in \matheur{C}^{\FB}$ obtained by putting $c$ in graded degree $n$ with the given $S_n$-action.

\subsubsection{Twisted algebra(s)}
Let $\matheur{C}$ be a stable symmetric monoidal category, we will use $\ComAlg(\matheur{C})$, resp. $\Lie(\matheur{C})$, and resp. $\coLie(\matheur{C})$ to denote the categories of commutative algebra, resp. $\Lie$-algebra, and resp. $\coLie$-coalgebra objects in $\matheur{C}$. We adopt the convention that our symmetric monoidal categories are not necessarily unital and hence, our commutative algebras are non-unital so that $\ComAlg(\matheur{C})$ makes sense even when the monoidal structure on $\matheur{C}$ is non-unital.

Using the Day convolution symmetric monoidal structures on $\matheur{C}^{\FB}$ and $\matheur{C}^{\FBplus}$ discussed above, we obtain the categories $\ComAlg(-)$, $\Lie(-)$, and $\coLie(-)$ where $-$ is either $\matheur{C}^{\FB}$ or $\matheur{C}^{\FBplus}$. We will refer to objects in these categories as twisted (co)algebras, adhering to the tradition that commutative algebras in $\Vect^{\FB}$ are usually called twisted commutative algebras.

\subsubsection{}
In terms of size, free twisted commutative algebras behave like associative algebras. Indeed, we have the following straightforward observation

\begin{lem}
  \label{lem:free_twisted_commutative_alg}
  Let $V_1 \in \Vect^{\FBplus}$ be a chain complex $V \in \Vect$ placed in graded degree $1$ with the (necessarily) trivial $S_1$ action. Then
  \[
    \Free_{\ComAlg} V_1 \simeq \bigoplus_{n>0} (V^{\otimes n})_n,
  \]
  where $V^{\otimes n}$ has a natural $S_n$-action by permuting the factors.
\end{lem}
\begin{proof}
  We have $V_1^{\otimes n} \simeq (\Lambda[S_n] \otimes V^{\otimes n})_n \in \Vect^{\FBplus}$ is equipped with two commuting $S_n$ actions, one coming from the fact that it is a degree $n$ part of an object in $\Vect^{\FB_+}$ and the other from the symmetric monoidal structure of $\Vect^{\FBplus}$ by permuting the factors of $V_1^{\otimes n}$. It is easy to see that the first action is by left multiplication on $\Lambda[S_n]$ whereas the second action is the diagonal action by inverse multiplying on the right of $\Lambda[S_n]$ and permuting $V^{\otimes n}$. Thus,
  \[
    (\Lambda[S_n] \otimes V^{\otimes n})_{S_n} \simeq V^{\otimes n}
  \]
  where we take co-invariants with respect to the diagonal action and where the remaining $S_n$-action on $V^{\otimes n}$ is given by permuting the factors.

  Now, the proof is concluded since
  \[
    \Free_{\ComAlg} V_1 = \bigoplus_{n>0} ((\Lambda[S_n] \otimes V^{\otimes n})_{S_n})_n \simeq \bigoplus_{n>0} (V^{\otimes n})_n.
  \]
\end{proof}

\subsubsection{Unit and augmentation}
\label{subsubsec:unit_and_augmentation}
When $\matheur{C}$ is unital symmetric monoidal, we use $\ComAlg^{\un}(\matheur{C})$, $\ComAlg^{\aug}(\matheur{C})$, and $\ComAlg^{\un, \aug}(\matheur{C})$ to denote the categories of unital, augmented, and augmented unital commutative algebra objects in $\matheur{C}$ respectively. Note that we have an equivalence of categories
\[
  \addUnit: \ComAlg(\matheur{C}) \rightleftarrows \ComAlg^{\un, \aug}(\matheur{C}): (-)_+
\]
given by formally adding a unit and taking the augmentation ideal.

When $\matheur{C}$ is non-unital but is a full-subcategory of a unital symmetric monoidal category $\matheur{C}^{\un}$, we let $\ComAlg^{\un, \aug}(\matheur{C})$ be the category of unital augmented commutative algebras in $\matheur{C}^{\un}$ whose augmentation ideal lies in $\matheur{C}$. More formally, we have the following pullback square of categories
\[
  \begin{tikzcd}
    \ComAlg^{\un, \aug}(\matheur{C}) \ar{d}{(-)_+} \ar[hookrightarrow]{r} & \ComAlg^{\un, \aug}(\matheur{C}^{\un}) \ar{d}{(-)_+} \\
    \ComAlg(\matheur{C}) \ar[hookrightarrow]{r} & \ComAlg(\matheur{C}^{\un})
  \end{tikzcd}
\]

As above, we have an equivalence of categories
\[
  \addUnit: \ComAlg(\matheur{C}) \rightleftarrows \ComAlg^{\un, \aug}(\matheur{C}): (-)_+
\]
Thus, as the notation suggests, the category $\ComAlg^{\un, \aug}(\matheur{C})$ does not depend on the embedding $\matheur{C} \subseteq \matheur{C}^{\un}$.

\subsection{Koszul duality}
\label{subsec:Koszul_duality}
The duality between co-commutative coalgebras and $\Lie$ algebras was first developed by Quillen in~\cite{quillen_notes_1971} and later generalized by Ginzburg--Kapranov in~\cite{ginzburg_koszul_1994} to general operads. In~\cite{francis_chiral_2011}, Francis and Gaitsgory realized this duality in the context of factorization algebras and formulated the notion of pro-nilpotence for categories, a sufficient condition for turning Koszul duality into an equivalence.

The passage between the two could be thought of as taking $\log$ and $\exp$, with the $\coLie$ coalgebra side being the target of $\log$. Under this analogy,\footnote{In fact, it is more than just an analogy. Decategorifying these constructions by, for instance, taking Euler characteristics, does indeed yield $\log$ and $\exp$. See~\cite{ho_homological_2021}*{\S5.3}.} the pro-nilpotence condition guarantees that certain limits converge.  The $\coLie$ coalgebra associated to a commutative algebra, also as suggested by the analogy, is much smaller in size compared to the original algebra, which allows us to deduce many properties of the algebra in a straightforward manner.

In this paper, we will use Koszul duality between commutative algebras and $\coLie$ coalgebras for categories of the form $\matheur{C}^{\FBplus}$ considered above. In fact, we will only make use of the very formal aspects of the theory, which, interestingly, are already enough to deduce strong representation stability results for ordered configurations spaces. In this subsections, we will review the basics of the theory in the form needed in the paper.

\subsubsection{Koszul duality for $\ComAlg$ and $\coLie$} Let $\matheur{C}$ be a symmetric monoidal category in $\DGCatprescont$. Then there is a pair of adjoint functors
\[
  \coPrim[1] : \ComAlg(\matheur{C}) \rightleftarrows \coLie(\matheur{C}) : \coChev \teq\label{eq:Koszul_duality}
\]
where $\coChev$ is the functor of taking cohomological Chevalley complex and $\coPrim[1]$ is a shift of the co-tangent fiber in the terminology of~\cite{gaitsgory_weils_2014}*{\S 6.1} or of the (operadic) derived indecomposables in the terminology of~\cite{galatius_cellular_2018}*{\S 8.2.3}. See also~\cite{ho_homological_2021}*{\S 2.2} for a brief overview in the notation that we will adopt in this paper.

When $\matheur{C}$ is pro-nilpotent, i.e. it can be given as a limit of stable nilpotent symmetric monoidal categories $\matheur{C} \simeq \lim_p \matheur{C}_p$ with transition functors commuting with both limits and colimits (see~\cite{francis_chiral_2011}*{Defn. 4.1.1} for the precise definition), we will use $\ev_p: \matheur{C} \to \matheur{C}_p$ to denote the canonical functor, which is symmetric monoidal and commutes with both limits and colimits. In this case, we have the following important result.

\begin{prop}[\cite{francis_chiral_2011}*{Prop. 4.1.2, Lem. 3.3.4, Lem. 4.1.6 (b)}] \label{prop:koszul_duality_pro-nilp}
  When $\matheur{C}$ is pro-nilpotent, the adjoint pair $\coPrim[1] \dashv \coChev$ at~\eqref{eq:Koszul_duality} is an equivalence. Moreover, they exchange trivial and free objects.\footnote{\label{ftn:trivial_algebra_meaning} A trivial (co)algebra is one where all operations vanish. For example, a trivial $\Lie$ algebra is what we usually call an abelian $\Lie$ algebra.}
\end{prop}

In general, it is quite hard to compute the Koszul dual of an object. However, we can sometimes extract good information by reducing to the case where the object involved is trivial.

\begin{prop}[\cite{ho_homological_2021}*{Cor. 3.3.6}] \label{prop:computing_coChev_as_limit}
  Let $\matheur{C}$ be a pro-nilpotent category and $\mathfrak{a} \in \coLie(\matheur{C})$. Then $\coChev \mathfrak{a}$ could be canonically written as a limit
  \[
    \coChev \mathfrak{a} \simeq \lim(\coChev^1 \mathfrak{a} \leftarrow \coChev^2 \mathfrak{a} \leftarrow \cdots) \simeq \lim_{i \geq 1} \coChev^i \mathfrak{a} \teq\label{eq:coChev_limit_pro-nilp}
  \]
  such that
  \[
    \Fib(\coChev^n \mathfrak{a} \to \coChev^{n-1} \mathfrak{a}) \simeq \Sym^n (\mathfrak{a}[-1]) \simeq (\mathfrak{a}[-1])_{S_n}^{\otimes n}, \quad\forall n\geq 1.
  \]
  Moreover, the limit~\eqref{eq:coChev_limit_pro-nilp} stabilizes in $\matheur{C}_p$ for each $p$, i.e. the sequence $\ev_p(\coChev^i \mathfrak{a})$ stabilizes when $i\gg 0$.
\end{prop}

Dually, we have the following statement about $\coPrim[1]$.
\begin{prop}[\cite{gaitsgory_study_2017}*{Vol. II, \S6.1.4}]
  \label{prop:computing_coPrim[1]_as_colimit}
  Let $\matheur{C}$ be a pro-nilpotent category and $A \in \ComAlg(\matheur{C})$. Then $\coPrim[1] A$ could be canonically written as a colimit
  \[
    \coPrim[1] A \simeq \colim(\coPrim^1[1] A \to \coPrim^2[1] A \to \dots) \simeq \colim_{i\geq 1} \coPrim^i[1] A \teq\label{eq:coPrim_colimit_pro-nilp}
  \]
  such that
  \[
    \coFib(\coPrim^{n-1}[1] A \to \coPrim^n[1] A) \simeq \coFree_{\coLie}(A)_n \simeq (\coLie(n) \otimes (A[1])^{\otimes n})_{S_n}, \quad \forall n\geq 1.
  \]
  Moreover, the colimit~\eqref{eq:coPrim_colimit_pro-nilp} stabilizes in $\matheur{C}_p$ for each $p$, i.e. the sequence $\ev_p(\coPrim^i[1] \mathfrak{a})$ stabilizes when $i\gg 0$.
\end{prop}

\subsubsection{Unital variant}
\label{subsubsec:unital_variant_Koszul}
We will adopt the following notation. Let $\matheur{C}$ be Proposition~\ref{prop:koszul_duality_pro-nilp} and suppose that we have a monoidal embedding $\matheur{C} \hookrightarrow \matheur{C}^{\un}$ as in \S\ref{subsubsec:unit_and_augmentation}. For $\mathfrak{a} \in \coLie(\matheur{C})$, we write
\[
  \coChev^{\un}(\mathfrak{a}) = \addUnit(\coChev(\mathfrak{a})) \in \ComAlg^{\un, \aug}(\matheur{C}).
\]

\subsubsection{$\FBplus$-objects}
\label{subsubsec:FBplus_cat_pro-nilp}
\begin{lem}
  Let $\matheur{C}$ be a symmetric monoidal category in $\DGCatprescont$, then $\matheur{C}^{\FBplus}$ is pro-nilpotent.
\end{lem}
\begin{proof}
  For each $i \in \mathbb{N}$, we let $\FBplus^{\leq i}$ (resp. $\FBplus^{>i}$) be the full-subcategory of $\FBplus$ consisting of non-empty sets of cardinality at most (resp. greater than) $i$. We have a fully faithful symmetric monoidal functor $\iota_{>i}: \matheur{C}^{\FBplus^{> i}} \hookrightarrow \matheur{C}^{\FBplus}$, given by the left Kan extension along the inclusion $\FBplus^{> i} \hookrightarrow \FBplus$. Concretely, this functor is given by extending by $0$. Thus, the essential image of this functor is a monoidal ideal in $\matheur{C}^{\FBplus}$, and hence, by~\cite{lurie_higher_2017}*{Prop. 2.2.1.9}, the quotient category $\matheur{C}^{\FBplus}/\matheur{C}^{\FBplus^{> i}}$ inherits a symmetric monoidal structure such that the quotient functor $\tr_{\leq i}: \matheur{C}^{\FBplus} \to \matheur{C}^{\FBplus}/\matheur{C}^{\FBplus^{> i}}$ is symmetric monoidal.

  By construction, we have the following localization sequence of categories
  \[
    \begin{tikzcd}
      \matheur{C}^{\FBplus^{> i}} \ar[hookrightarrow, shift left=\arrdisp]{r}{\iota_{>i}} & \ar[shift left=\arrdisp]{l}{\iota_{>i}^R}\matheur{C}^{\FBplus} \ar[shift left=\arrdisp]{r}{\tr_{\leq i}} & \ar[hookrightarrow, shift left=\arrdisp]{l}{\tr_{\leq i}^R} \matheur{C}^{\FBplus}/\matheur{C}^{\FBplus^{> i}}
    \end{tikzcd}
  \]
  where $\iota_{>i}^R$ is the right adjoint of $\iota_{>i}$ and $\tr_{\leq i}^R$ is the right adjoint of $\tr_{\leq i}$. Thus, for any $c\in \matheur{C}^{\FBplus}$, we have the following fiber sequence
  \[
    \iota_{>i}\iota_{>i}^R c \to c \to \tr_{\leq i}^R \tr_{\leq i} c.
  \]

  Observe that $\iota_{>i}^R$ is given by restricting along $\FBplus^{>i} \hookrightarrow \FBplus$, which means that $\tr_{\leq i}^R \tr_{\leq i} c$ is given by setting all components of $c$ in graded degree greater than $i$ to be $0$. In particular, the essential image of $\tr_{\leq i}^R$ is precisely $\matheur{C}^{\FBplus^{\leq i}}$. Thus, we have an equivalence of categories $\matheur{C}^{\FBplus}/\matheur{C}^{\FBplus^{> i}} \simeq \matheur{C}^{\FBplus^{\leq i}}$, and under this identification, $\tr_{\leq i}$ corresponds to the restriction functor $\matheur{C}^{\FBplus} \to \matheur{C}^{\FBplus^{\leq i}}$ (and $\tr_{\leq i}^R$ corresponds to the right adjoint of the restriction functor, which is given by extending by $0$). By abuse of notation, we will also denote the restriction functor $\matheur{C}^{\FBplus} \to \matheur{C}^{\FBplus^{\leq i}}$ by $\tr_{\leq i}$.

  The discussion above equips $\matheur{C}^{\FBplus^{\leq i}}$ with a natural symmetric monoidal structure such that the restriction functor $\matheur{C}^{\FBplus} \to \matheur{C}^{\FBplus^{\leq i}}$ is symmetric monoidal. But now we are done since $\matheur{C}^{\FBplus} \simeq \lim_i \matheur{C}^{\FBplus^{\leq i}}$.
\end{proof}

The lemma above implies that Propositions~\ref{prop:koszul_duality_pro-nilp} and~\ref{prop:computing_coChev_as_limit} hold in this case. For this reason, as far as Koszul duality is concerned, it is technically more advantageous to work with $\FBplus$-, rather than $\FB$-, objects.

\section{Twisted commutative factorization algebras (tcfa's)}
\label{sec:tcfas}
In this section, we develop the theory of twisted commutative algebras (tcfa's). More specifically, we will introduce in~\S\ref{subsec:twisted_commutative_Ran} a new gadget, the twisted commutative $\Ran$ space, which plays the same role as the $\Ran$ space in the usual theory of factorization algebras. In~\S\ref{subsec:cardinality_filtration}, we construct the cardinality filtration on this space and prove various technical results needed in~\S\ref{subsec:reduction_to_diagonal} to show that tcfa's, introduced in~\S\ref{subsec:tcfas}, are simply twisted commutative algebras over $X$. Even though the main spirit of the various proofs in this section is the same as for the usual (i.e. non-twisted) case, they are more technically involved since the twisted commutative $\Ran$ space is a more complicated object. Finally, in the last subsection~\S\ref{subsec:variants_tcfas}, we describe the procedure of truncating tcfa's which will later be used to connect factorization cohomology with coefficients in a certain truncated twisted commutative algebra with the cohomology of generalized configuration spaces.

\subsection{Twisted commutative $\Ran$ space}
\label{subsec:twisted_commutative_Ran}
In this subsection, we give the construction of and prove various elementary properties of the twisted commutative $\Ran$ space.

\subsubsection{The construction/definition}
Let $X$ be a scheme. We consider the prestack $\Ran(X, \FBplus)$ defined as follows:
\[
  \Ran(X, \FBplus) = \colim_{I\in \FBplus} X^I = \bigsqcup_{k > 0} X^k/S_k, \teq\label{eq:defn_Ran_FBplus}
\]
where the quotient should be understood in the groupoid sense. $\Ran(X, \FBplus)$ should be viewed as a way to package $P(X)_+$ (see~\eqref{eq:powers_of_X}) along with all the $S_k$ actions in one space.

More concretely, for each test scheme $S\in \Sch$, the groupoid $\Ran(X, \FBplus)(S)$ consists of pairs $(I, \gamma)$ where $I \in \FBplus$ and where $\gamma: I \to X(S)$ is a map of sets. Moreover, a map $(I, \gamma) \to (I', \gamma')$ is a map $I \to I'$ in $\FBplus$ that makes the following diagram commute
\[
  \begin{tikzcd}
    I \ar[d] \ar{r}{\gamma} & X(S) \\
    I' \ar{ur}[swap]{\gamma'}
  \end{tikzcd}
\]

A sheaf $\matheur{F} \in \Shv(\Ran(X, \FBplus))$ thus consists of a collection of sheaves $\matheur{F}^{(I)} \in \Shv(X^I)$ that are equivariant with respect to $S_I$ for each $I \in \FBplus$.

\subsubsection{}The category $\FBplus$ itself, which is a groupoid, could be viewed as a constant prestack, which will be denoted by the same notation. Namely, for each test scheme $S \in \Sch$,
\[
  \FBplus(S) = \FBplus = \bigsqcup_{n \geq 1} BS_n,
\]
where $BS_n$ is the groupoid associated to the group $S_n$.

Observe that
\[
  \Shv(\FBplus) \simeq \Vect^{\FBplus} \simeq \Fun(\FBplus, \Vect)
\]
and moreover, we have a canonical morphism between prestacks
\[
  \Ran(\pi): \Ran(X, \FBplus) \to \FBplus \teq\label{eq:Ran(pi)}
\]
where $\pi: X\to \pt$ is the structure map. Explicitly, this map sends a pair $(I, \gamma)$ to $I$ in the functor of points description. Or in a slightly less canonical way, this is given by the obvious map $X^n/S_n \to BS_n$. In particular, given a sheaf $\matheur{F} \in \Shv(\Ran(X, \FBplus))$, $\Ran(\pi)_! \matheur{F} \in \Vect^{\FBplus}$ is a sequence of representations of $S_n$, as expected.

\begin{rmk}
  \label{rmk:taking_coinvariant}
  Note that we can also consider the constant prestack (valued in sets) given by the set $\mathbb{N}_+$ and consider the following composition
  \[
    \Ran(X, \FBplus) \to \FBplus \to \mathbb{N}_+.
  \]
  Pushing forward along this diagram amounts to taking the cohomology (with compact support) of the $X^I$'s followed by taking $S_I$-coinvariants.
\end{rmk}

\begin{rmk}
  \label{rmk:Ran(f)}
  One should think of $\Ran(X, \FB_+)$ as the multi-colored version of $\FB_+$. Indeed, when $X = \pt$, $\Ran(X, \FB_+) = \FB_+$. The map $\Ran(X, \FBplus) \to \FBplus$ is a special case of the following: given a map $f: X\to Y$ between schemes, we obtain a map $\Ran(f): \Ran(X, \FBplus) \to \Ran(Y, \FBplus)$.
\end{rmk}

\subsubsection{Commutative semigroup structures}
As mentioned before, the category $\FBplus$ is a (non-unital) symmetric monoidal category and hence, its associated prestack is a semigroup object in the category of prestacks with the Cartesian symmetric monoidal structure.

The prestack $\Ran(X, \FBplus)$ is also equipped with a semigroup structure given by
\[
  \union_n: \Ran(X, \FBplus)^n \to \Ran(X, \FBplus)
\]
such that for each test scheme $S$, we have
\begin{align*}
  (\union_n)_S: \Ran(X, \FBplus)^n(S) &\to \Ran(X, \FBplus)(S) \\
  ((I_k, \gamma_k))_{k=1}^n &\mapsto (\sqcup_{k=1}^n I_k, \sqcup_{k=1}^n \gamma_k).
\end{align*}
Unless confusion might occur, we will suppress $n$ from the notation $\union_n$.

From the description, it is immediate that the morphism $p: \Ran(X, \FBplus) \to \FBplus$ is compatible with the commutative semigroup structures on both sides.

\subsubsection{Convolution symmetric monoidal structure}
Similar to how Day convolution monoidal structure is defined, the commutative semigroup structure on $\Ran(X, \FB_+)$ could be used to construct a convolution symmetric monoidal structure on $\Shv(\Ran(X, \FBplus))$. We start with the following

\begin{lem} \label{lem:pseudo_properness_union}
  For each $n$, the map
  \[
    \union_n: \Ran(X, \FBplus)^n \to \Ran(X, \FBplus)
  \]
  is pseudo-proper.\footnote{See~\cite{gaitsgory_atiyah-bott_2015}*{\S1.5} for the definition of pseudo-properness.}
\end{lem}
\begin{proof}
  Observe that for each $I\in \FBplus$, we have the following pullback square
  \[
    \begin{tikzcd}
      \matheur{Q} \ar{d} \ar{r} & \Ran(X, \FBplus)^n \ar{d} \\
      X^I \ar{r} & \Ran(X, \FBplus)
    \end{tikzcd}
  \]
  where
  \[
    \matheur{Q} = \colim_{I \simeq \sqcup_{k=1}^n I_k} X^{\sqcup_{k=1}^n I_k} \simeq \colim_{I\simeq \sqcup_{k=1}^n I_k} X^I,
  \]
  where the category that we are taking the colimit over is $(\FBplus^n)_{/I}$.

  Now, the map $\matheur{Q} \to X^I$ is visibly pseudo-proper since each term in the colimit maps to $X^I$ identically.
\end{proof}

\begin{lem} \label{lem:right_laxness_Shv_!}
  Let $\PreStk_{\pseudoproper}$ denote the (non-full) subcategory of $\PreStk$ where we only allow pseudo-proper maps. Then,
  \[
    \Shv_!|_{\PreStk_{\pseudoproper}}: \PreStk_{\pseudoproper} \to \DGCatprescont
  \]
  is right-lax monoidal.
\end{lem}
\begin{proof}
  For each family of maps $f_i: \matheur{X}_i \to \matheur{Y}_i$, $1\leq i \leq n$, between prestacks, we have the following commutative diagram
  \[
    \begin{tikzcd}
      \bigotimes_i \Shv(\matheur{X}_i) \ar{r}{\boxtimes} & \Shv(\sqcap_i \matheur{X}_i) \\
      \bigotimes_i \Shv(\matheur{Y}_i) \ar{u} \ar{r} {\boxtimes} & \Shv(\sqcap_i \matheur{Y}_i) \ar{u}
    \end{tikzcd}
  \]
  where the vertical arrows are given by $!$-pullbacks. Passing to left adjoints for the vertical arrows, we obtain the following diagram, which commutes up to a natural transformation
  \[
    \begin{tikzcd}
      \bigotimes_i \Shv(\matheur{X}_i) \ar{d} \ar{r}{\boxtimes} & \Shv(\sqcap_i \matheur{X}_i) \ar{d} \\
      \bigotimes_i \Shv(\matheur{Y}_i) \ar{r} {\boxtimes} & \Shv(\sqcap_i \matheur{Y}_i)
    \end{tikzcd}
  \]
  where the vertical arrows are given by $!$-pushforward. It suffices to show that this diagram in fact commutes.

  We will prove the case where $n=2$. The general case is exactly the same. Note that the lax commutativity of the diagram is given as follows: for each pair $\matheur{F}_i \in \Shv(\matheur{X}_i), i \in \{1, 2\}$, we have a natural map
  \[
    (f_1 \times f_2)_! (\matheur{F}_1 \boxtimes \matheur{F}_2) \to f_{1!} \matheur{F}_1 \boxtimes f_{2!} \matheur{F}_2, \teq\label{eq:Kuenneth_pseudo-proper}
  \]
  which is the adjoint of the following map
  \[
    \matheur{F}_1 \boxtimes \matheur{F}_2 \to f_1^! f_{1!} \matheur{F}_1 \boxtimes f_2^! f_{2!} \matheur{F}_2 \simeq (f_1\times f_2)^! (f_{1!} \matheur{F}_1 \boxtimes f_{2!}\matheur{F}_2).
  \]
  Now, the map~\eqref{eq:Kuenneth_pseudo-proper} being an equivalence when $f_1$ and $f_2$ are pseudo-proper is just the \Kunneth{} formula for pseudo-proper morphisms between prestacks, which was proved in~\cite{ho_free_2017}*{Lem. 2.10.7}.
\end{proof}

Lemma~\ref{lem:pseudo_properness_union} implies that $\Ran(X, \FBplus)$ is in fact a commutative semigroup in $\PreStk_{\pseudoproper}$. Using the fact that right-lax monoidal functors preserve algebra structures, Lemma~\ref{lem:right_laxness_Shv_!} directly implies the following

\begin{cor} \label{cor:otimesstar_monoidal_RanFB}
  The category $\Shv(\Ran(X, \FBplus))$ has a natural symmetric monoidal structure, which we will denote by $\otimesstar$, given concretely as follows: for each $(\matheur{F}_k)_{k=1}^n \in \Shv(\Ran(X, \FBplus))$,
  \[
    \matheur{F}_1 \otimesstar \cdots\otimesstar \matheur{F}_n = \union_{n!}(\boxtimes_{k=1}^n \matheur{F}_k).
  \]
\end{cor}

\subsubsection{Compatibility with $\Ran(f)$}
\label{subsubsec:conv_monoidal_vs_Ran(f)}
Let $f: X\to Y$ be a map between schemes. Then by \Kunneth{} formula, the induced map
\[
  \Ran(f)_!: \Shv(\Ran(X, \FBplus)) \to \Shv(\Ran(Y, \FBplus))
\]
is symmetric monoidal, where $\Ran(f)$ is as in Remark~\ref{rmk:Ran(f)}.

In the case where $Y = \pt$ and $f=\pi$, the structure map of $X$, we have $\Shv(\Ran(Y, \FBplus)) \simeq \Vect^{\FBplus}$. Unwinding the definition of the convolution monoidal structure, we see that it coincides with the Day convolution monoidal structure on $\Vect^{\FBplus}$. Thus, the map
\[
  \Ran(\pi)_!: \Shv(\Ran(X, \FBplus)) \to \Vect^{\FBplus}
\]
is symmetric monoidal.

\subsection{The cardinality filtration on $\Ran(X, \FBplus)$}
\label{subsec:cardinality_filtration}
In this subsection, we introduce the cardinality filtration, a natural filtration on $\Ran(X, \FBplus)$ which allows us to perform inductive arguments for sheaves on the $\Ran$ space. More specifically, we will realize $\Ran(X, \FBplus)$ as a colimit
\[
  \Ran(X, \FBplus) \simeq \colim(\Ran^{\leq 1}(X, \FBplus) \xrightarrow{\iota_{\leq 1, \leq 2}} \Ran^{\leq 2}(X, \FBplus) \xrightarrow{\iota_{\leq 2, \leq 3}} \cdots),
\]
where $\Ran^{\leq n}(X, \FBplus)$ parametrizes maps $\gamma: I\to X$ such that $|\im \gamma| \leq n$.

The main output of this subsection is the \devissage{} sequence given in Proposition~\ref{prop:devissage_Ran_leq_n}, which says that we have the following natural cofiber sequence for each $\matheur{F} \in \Shv(\Ran^{\leq n}(X, \FBplus))$
\[
  (\iota_{\leq n-1, \leq n})_! \iota_{\leq n-1, \leq n}^! \matheur{F} \to \matheur{F} \to (\jmath_n)_* (\jmath_n)^! \matheur{F}.
\]
This would be straightforward if we had the following \devissage{} diagram of prestacks
\[
  \begin{tikzcd}
    \Ran^{\leq n-1}(X, \FBplus) \ar{r}{\iota_{\leq n-1, \leq n}} & \Ran^{\leq n}(X, \FBplus) & \Ran^{=n}(X, \FBplus) \ar{l}[swap]{\jmath_n}
  \end{tikzcd}
\]
where $\iota_{\leq n-1, \leq n}$ is a closed embedding and $\jmath_n$ its open complement.  However, the map $\iota_{\leq n-1, \leq n}$ is not a closed embedding in general, and most of this subsection is there to deal with this issue. We recommend the reader to, at the first pass, take as faith Proposition~\ref{prop:devissage_Ran_leq_n} and return to this subsection to look up the notation when necessary.

\subsubsection{The diagonal of $\Ran(X, \FBplus)$}
We start by realizing $\FBplus \times X$ as the diagonal of $\Ran(X, \FBplus)$. Let $\Delta(X, \FBplus)$ be the full-sub-prestack of $\Ran(X, \FBplus)$ given by the following functor of points description: for each test scheme $S$, $\Delta(X, \FBplus)(S)$ consists of pairs $(I, \gamma)$ such that $\gamma(I)$ is a singleton set, i.e. we have a factorization
\[
  I \to \{*\} \to X(S).
\]

\begin{lem} \label{lem:diagonal_closed_embedding}
  Let $\delta: \Delta(X, \FBplus) \to \Ran(X, \FBplus)$ denote the inclusion map. Then, $\delta$ is a closed embedding.
\end{lem}
\begin{proof}
  It is immediate from the fact that the pullback of $\delta$ over $X^I \to \Ran(X, \FBplus)$ is given precisely by the small diagonal of $X^I$.
\end{proof}

From the functor of points description, it is immediate that
\[
  \Delta(X, \FBplus) \simeq \FBplus \times X \simeq \bigsqcup_{n\geq 1} X/S_n \simeq \bigsqcup_{n\geq 1} BS_n \times X \teq\label{eq:description_diagonal}
\]
where in the third term, $S_n$ acts on $X$ trivially, and as usual, the quotient is to be understood in the groupoid sense. Thus, $\FBplus\times X$ can be viewed in a natural way the diagonal of $\Ran(X, \FBplus)$.

\subsubsection{The sub-prestack $\Ran^{\leq n}(X, \FBplus)$}
The diagonal of $\Ran(X, \FBplus)$ is in fact part of a natural filtration. For each $n\geq 1$, we let $\Ran^{\leq n}(X, \FBplus)$ be the full sub-prestack of $\Ran(X, \FBplus)$ such that for each test scheme $S$,
\[
  \Ran^{\leq n}(X, \FBplus)(S) = \{(I, \gamma) \in \Ran(X, \FB_+)(S): |\im(\gamma)| \leq n\}.
\]

We will use the following notation to denote the various natural embeddings
\begin{align*}
  \iota_{\leq n}&: \Ran^{\leq n}(X, \FBplus) \to \Ran(X, \FBplus), \quad \forall n\geq 1, \\
  \iota_{\leq n, \leq m}&: \Ran^{\leq n}(X, \FBplus) \to \Ran^{\leq m}(X, \FBplus), \quad \forall n\leq m.
\end{align*}
Clearly, $\Ran^{\leq 1}(X, \FBplus) \simeq \Delta(X, \FB_+) \simeq \FBplus\times X$ is the diagonal of $\Ran(X, \FBplus)$ discussed above. Moreover, the following is immediate.

\begin{lem}
  \label{lem:cardinality_filtration}
  We have a natural equivalence
  \[
    \Ran(X, \FBplus) \simeq \colim_{n\geq 1} \Ran^{\leq n}(X, \FBplus)
  \]
  and hence
  \[
    \Shv(\Ran(X, \FBplus)) \simeq \lim_{n\geq 1} \Shv(\Ran^{\leq 1}(X, \FBplus)).
  \]
\end{lem}

The map $\iota_{\leq n}$ and $\iota_{\leq n, \leq m}$ enjoy nice geometric properties. Note, however, that they are not closed embeddings.

\begin{prop}
  The morphisms $\iota_{\leq n}$ and $\iota_{\leq n, \leq m}$ are finitary pseudo-proper in the sense of~\cite{gaitsgory_atiyah-bott_2015}*{\S 7.4.6}, i.e. for each test scheme $S$ mapping to $\Ran(X, \FBplus)$ (resp. $\Ran^{\leq m}(X, \FBplus)$), the fiber product $S\times_{\Ran(X, \FBplus)} \Ran^{\leq n}(X, \FBplus)$ (resp. $S \times_{\Ran^{\leq m}(X, \FBplus)} \Ran^{\leq n}(X, \FBplus)$) can be written as a finite colimit of schemes proper over $S$.
\end{prop}
\begin{proof}
  We start with $\iota_{\leq n}$. Using~\eqref{eq:defn_Ran_FBplus}, it suffices to show that for each $I \in \FBplus$, $X^I \times_{\Ran(X, \FBplus)} \Ran^{\leq n}(X, \FBplus)$ can be written as a finite colimits of schemes proper over $X^I$. In fact, we will show that
  \[
    X^I \times_{\Ran(X, \FBplus)} \Ran^{\leq n}(X, \FBplus) \simeq \colim_{(I \surjects J) \in \matheur{K}_{I, n}} X^J, \teq\label{eq:ins_leq_n_pseudo-proper}
  \]
  where $\matheur{K}_{I, n}$ is the category whose objects are surjections $I\surjects J$ for some finite set $J$ of size at most $n$ and morphisms from $I\surjects J_1$ to $I \surjects J_2$ are surjections $J_1 \surjects J_2$ that are compatible with the surjections from $I$. This will immediately imply pseudo-properness since $\matheur{K}_{I, n}$ is finite, and the map from each term in the colimit to $X^I$ is a closed embedding.

  To prove~\eqref{eq:ins_leq_n_pseudo-proper}, first note that
  \begin{align*}
    (X^I\times_{\Ran(X, \FBplus)} \Ran^{\leq n}(X, \FBplus))(S)
    &= \{(\gamma: I \to X(S)): |\im(\gamma)| \leq n\}.
  \end{align*}
  Now, for any set $T$ in general (which is $X(S)$ in the case of interest), we have
  \begin{align*}
    \{(\gamma: I \to T): |\im(\gamma)| \leq n\}
    &\simeq \colim_{\substack{\gamma: I \to T \\ |\im \gamma|\leq n}} \{*\}
    \simeq \colim_{\substack{(I\surjects J)\in \matheur{K}_{I, n} \\ J \to T}} \{*\}
    \simeq \colim_{I\surjects J \in \matheur{K}_{I, n}} \colim_{J\to T} \{*\}
    \simeq \colim_{I\surjects J \in \matheur{K}_{I, n}} T^J.
  \end{align*}
  The second equivalence is due to cofinality. The others are immediate. The proof of~\eqref{eq:ins_leq_n_pseudo-proper} thus concludes.

  From the calculation above, we obtain an expression of $\Ran^{\leq n}(X, \FBplus)$ as a colimit of schemes:
  \[
    \Ran^{\leq n}(X, \FBplus) \simeq \colim_{\substack{I \in \FBplus}} \colim_{(I \surjects J) \in \matheur{K}_{I, n}} X^J \simeq \colim_{I \surjects J \in \matheur{K}_n} X^J,
  \]
  where $\matheur{K}_n$ is the category consisting of surjections of non-empty finite sets $I \surjects J$ with morphisms given by commutative squares of the following form
  \[
    \begin{tikzcd}
      I_1 \ar{d}{\simeq} \ar[two heads]{r} & J_1 \ar[two heads]{d} \\
      I_2 \ar[two heads]{r} & J_2
    \end{tikzcd}
  \]

  Using a similar argument as the one above but for $X^J \times_{\Ran^{\leq m}(X, \FBplus)} \Ran^{\leq n}(X, \FBplus)$, we see that $\iota_{\leq n, \leq m}$ is also finitary pseudo-proper for all $n \leq m$.
\end{proof}

\subsubsection{The sub-prestack $\Ran^{=n}(X, \FBplus)$}
We will now study the ``complement'' of $\Ran^{\leq n-1}(X, \FBplus)$ inside $\Ran^{\leq n}(X, \FBplus)$. For each $I \in \FBplus$ and each $n\geq 1$, we let $X^{I}_{\leq n}$ denote the closed subscheme of $X^I$, which is the scheme-theoretic image of
\[
  \bigsqcup_{I\surjects J \in \matheur{K}_{I, n}} X^J \to X^I
\]
and $X^I_{=n}$ the complement of $X^I_{\leq n-1}$ inside $X^I_{\leq n}$. Here, $\matheur{K}_{I, n}$ is as in~\eqref{eq:ins_leq_n_pseudo-proper}.

Clearly, $X^I_{\leq n}$ and $X^I_{=n}$ are invariant under $S_I$ and we let
\[
  \lbar{\Ran}^{\leq n}(X, \FBplus) = \colim_{I\in\FBplus} X^I_{\leq n} \quad\text{and}\quad \Ran^{=n}(X, \FBplus) = \colim_{I \in \FBplus} X^I_{=n}.
\]

\subsubsection{} By construction, $\lbar{\Ran}^{\leq n}(X, \FBplus)$ is a closed sub-prestack of $\Ran(X, \FBplus)$ and moreover, the map $\iota_{\leq n}$ naturally factors through $\lbar{\Ran}^{\leq n}(X, \FBplus)$ and we have the following commutative diagram
\[
  \begin{tikzcd}
    \Ran^{\leq n}(X, \FBplus) \ar{dr}[swap]{h_{\leq n}} \ar{rr}{\iota_{\leq n}} && \Ran(X, \FBplus) \\
    & \lbar{\Ran}^{\leq n}(X, \FBplus) \ar{ur}[swap]{\lbar{\iota}_{\leq n}}
  \end{tikzcd} \teq\label{eq:insert_to_image_leq_n}
\]

Observe also that
\[
  X^I_{=n} = \bigsqcup_{\substack{I \surjects J \\ |J| = n}} \oversetsupscript{X}{\circ}{J}
\]
where $\oversetsupscript{X}{\circ}{J}$ is the open subscheme of $X^J$ where we remove all diagonals. Thus, we have the following commutative diagram
\[
  \begin{tikzcd}
    \Ran^{=n}(X, \FBplus) \ar{dr}{\jmath'_n} \ar{d}{\jmath_n} \ar{r}{\iota_n} & \Ran(X, \FBplus) \\
    \Ran^{\leq n}(X, \FBplus) \ar{r}{h_{\leq n}} & \lbar{\Ran}^{\leq n}(X, \FBplus) \ar{u}[swap]{\lbar{\iota}_{\leq n}}
  \end{tikzcd}
\]
where $\jmath_n$ and $\jmath'_n$ are open embeddings, and in fact, $\Ran^{=n}(X, \FBplus)$ is the open complement of $\lbar{\Ran}^{\leq n-1}(X, \FBplus)$ inside $\lbar{\Ran}^{\leq n}(X, \FBplus)$. Moreover, it admits the following functor of points description:
\[
  \Ran^{=n}(X, \FBplus)(S) = \{(I, \gamma) \in \Ran^{\leq n}(X, \FBplus)(S): |\im (\gamma|_{k(s)})| = n, \forall s\in S\} \teq\label{eq:functor_of_pts_description_Ran=n}
\]
for each test scheme $S$, where $\gamma|_{k(s)} \in \Ran(X, \FBplus)(\Spec k(s))$ is the pullback of $\gamma$ to $\Spec k(s) \to S$.

\subsubsection{A \devissage{} sequence} We will now prove a \devissage{} sequence arising from the cardinality filtration. To start, recall the following

\begin{lem}[\cite{gaitsgory_atiyah-bott_2015}*{Lem. 7.4.11}] \label{lem:gait_finitary_pseudoproper_and_sheaves}
  Let $f: \matheur{Y}_1 \to \matheur{Y}_2$ be a finitary pseudo-proper map between prestacks.
  \begin{myenum}{(\alph*)}
  \item If $f$ is injective, i.e. $f_S: \matheur{Y}_1(S) \to \matheur{Y}_2(S)$ is fully-faithful for each test scheme $S$, then the functor $f_!$ is fully faithful.

  \item If $f$ is surjective on $k$-points, then $f^!$ is conservative.

  \item If $f$ is injective, and surjective on $k$-points, then the functors $f_!$ and $f^!$ are mutual inverses.
  \end{myenum}
\end{lem}

\begin{cor}
  The natural map $h_{\leq n}$ of~\eqref{eq:insert_to_image_leq_n}, which is pseudo-proper, injective, and surjective on $k$-points, induces an equivalence of categories
  \[
    \Shv(\Ran^{\leq n}(X, \FBplus)) \simeq \Shv(\lbar{\Ran}^{\leq n}(X, \FBplus)). \teq \label{eq:equivalence_category_sheaves_leq_n}
  \]
\end{cor}
\begin{proof}
  Immediate from the construction of $\lbar{\Ran}^{\leq n}(X, \FBplus)$.
\end{proof}

\begin{prop} \label{prop:devissage_Ran_leq_n}
  Let $\matheur{F} \in \Shv(\Ran^{\leq n}(X, \FBplus))$. Then we have the following cofiber sequence
  \[
    (\iota_{\leq n-1, \leq n})_! \iota_{\leq n-1, \leq n}^! \matheur{F} \to \matheur{F} \to (\jmath_n)_* (\jmath_n)^! \matheur{F}.
  \]
\end{prop}
\begin{proof}
  Using the equivalence~\eqref{eq:equivalence_category_sheaves_leq_n}, it suffices to establish the corresponding cofiber sequence for sheaves on $\lbar{\Ran}^{\leq n}(X, \FBplus)$. But this is clear, since we have the following diagram of spaces
  \[
    \begin{tikzcd}
      \lbar{\Ran}^{\leq n-1}(X, \FBplus) \ar{r}{\lbar{\iota}_{\leq n-1, \leq n}} & \lbar{\Ran}^{\leq n}(X, \FBplus) & \Ran^{=n}(X, \FBplus) \ar{l}[swap]{\jmath'_n}
    \end{tikzcd}
  \]
  where $\lbar{\iota}_{\leq n-1, \leq n}$ is a closed embedding with $\jmath'_n$ being its open complement.
\end{proof}

\begin{cor} \label{cor:conservativity_i_n^!}
  Let $f: \matheur{F} \to \matheur{G}$ be a morphism of sheaves on $\Ran(X, \FBplus)$. Then, $f$ is an equivalence if and only if for each $n \geq 1$, the map $\iota^!_n(f): \iota_n^! \matheur{F} \to \iota_n^! \matheur{G}$ is an equivalence as objects in $\Shv(\Ran^{=n}(X, \FBplus))$.
\end{cor}
\begin{proof}
  This follows from Lemma~\ref{lem:cardinality_filtration} and Proposition~\ref{prop:devissage_Ran_leq_n}.
\end{proof}

\subsection{Twisted commutative factorizable algebras (tcfa's)}
\label{subsec:tcfas}
We will now come to the definition of twisted \emph{commutative factorization algebras}, the twisted commutative counterpart of commutative factorization algebras.

\subsubsection{The category $\ComAlgstar(\Ran(X, \FB_+))$}
Using the $\otimesstar$-monoidal structure on $\Ran(X, \FBplus)$, we can talk about commutative algebra objects in $\Shv(\Ran(X, \FBplus)^{\otimesstar})$ and define
\[
  \ComAlgstar(\Ran(X, \FBplus)) = \ComAlg(\Shv(\Ran(X, \FBplus))^{\otimesstar})
\]
to be the category of such objects.

\subsubsection{Twisted commutative factorizable algebras (tcfa's)}
For each $n\in \mathbb{N}_+$, let
\[
  j_n: (\Ran(X, \FBplus))^n_{\disj} \hookrightarrow \Ran(X, \FBplus)^n
\]
be the open sub-prestack given by the condition that the graphs of the components in the product are disjoint from each other. We will often suppress the subscript $n$ from $\jmath_n$ unless confusion is likely to occur.

Consider the following diagram
\[
  \begin{tikzcd}
    \Ran(X, \FBplus)^n_{\disj} \ar[bend right=10]{rr}[swap]{\union'_n} \ar{r}{j_n} & \Ran(X, \FBplus)^n \ar{r}{\union_n} & \Ran(X, \FBplus)
  \end{tikzcd} \teq\label{eq:union_disj_open_inclusion_notation}
\]
Given an object $\matheur{A} \in \ComAlgstar(\Ran(X, \FBplus))$, the multiplication maps give us the following
\[
  \union_{n!}(\matheur{A}^{\boxtimes n}) \to \matheur{A}
\]
for each $n$, which, in turns, induce maps
\[
  \matheur{A}^{\boxtimes n} \to \union_n^! \matheur{A}
\]
and hence
\[
  j^!_n(\matheur{A}^{\boxtimes n}) \to j_n^! \union_n^! \matheur{A} \simeq \union'_n{}^! \matheur{A}. \teq\label{eq:factorizable_map}
\]

The object $\matheur{A}$ is said to be factorizable if the map~\eqref{eq:factorizable_map} is an equivalence for each $n$. We let $\Factstar(X, \FBplus)$ be the full-subcategory of $\ComAlgstar(\Ran(X, \FBplus))$ spanned by factorizable objects. We will call these objects \emph{twisted commutative factorizable algebras} (tcfa's) over $X$.

\begin{rmk}
  \label{rmk:geometric_realization_sifted_colim}
  Let $\matheur{A}_\bullet$ be a simplicial object in $\ComAlgstar(\Ran(X, \FBplus))$. Then its geometric realization, $|\matheur{A}_\bullet|$, is factorizable if each term of $\matheur{A}_\bullet$ is. This is because the forgetful functor $\ComAlg(\matheur{C}) \to \matheur{C}$ in general commutes with sifted colimits, by~\cite{lurie_higher_2017}*{Cor. 3.2.3.2}, and the map~\eqref{eq:factorizable_map} is compatible with sifted colimits.
\end{rmk}

\subsection{Reduction to the diagonal}
\label{subsec:reduction_to_diagonal}
The notation $\Factstar(X, \FBplus)$ suggests that objects in this category can somehow be defined as the category of certain objects living over $X$ rather than over the whole of $\Ran(X, \FBplus)$. In fact, this is a well-known phenomenon in the non-twisted setting: in the case of curves, it is due to Beilinson--Drinfeld~\cite{beilinson_chiral_2004} whereas the general case is due to Gaitsgory--Lurie~\cite{gaitsgory_weils_2014}. The goal of this subsection is to prove Theorem~\ref{thm:tcfa_vs_tcaX}, which says that this is indeed the case in the twisted setting as well.

\subsubsection{Twisted commutative algebras (tca's) on $X$}
From~\eqref{eq:description_diagonal}, we get
\[
  \Shv(\Delta(X, \FBplus)) \simeq \Shv(\FBplus\times X) \simeq \Shv(X)^{\FBplus},
\]
and we are in the purview of~\S\ref{subsec:FB_twisted_comalg}. The category $\Shv(\FBplus\times X)$ is thus equipped with a Day convolution symmetric monoidal structure, where we use the $\otimesshriek$-symmetric monoidal structure on $\Shv(X)$. We will use $\otimesshriek$ to denote this symmetric monoidal structure on $\Shv(\FBplus\times X)$ as well.

It thus makes sense to talk about commutative algebras in $\Shv(\FBplus \times X)$. We will use $\ComAlg(\Shv(X)^{\FBplus})$ to denote the category of all such objects. We will call these twisted commutative algebras (tca's) on $X$.

\subsubsection{tcfa's vs. tca's on $X$}
We will now show that the two categories $\ComAlg(\Shv(X)^{\FBplus})$ and $\Factstar(X, \FBplus)$ are naturally equivalent. We will start with some preparation.

\begin{lem}
  The functor
  \[
    \delta^!: \Shv(\Ran(X, \FBplus)) \to \Shv(\FBplus \times X) \simeq \Shv(X)^{\FBplus}
  \]
  is a symmetric monoidal functor with respect to the $\otimesstar$-monoidal structure on the source and the $\otimesshriek$-monoidal structure (the one coming from Day convolution) on the target.
\end{lem}
\begin{proof}
  The proof is similar to that of~\cite{gaitsgory_weils_2014}*{Prop. 5.5.19}. We will show that for $\matheur{F}, \matheur{G} \in \Shv(\Ran(X, \FBplus))$, we have a canonical equivalence
  \[
    \delta^!(\matheur{F} \otimesstar \matheur{G}) \simeq \delta^! \matheur{F} \otimesshriek \delta^! \matheur{G}.
  \]
  Consider the following pullback diagram
  \[
    \begin{tikzcd}
      (\FBplus\times \FBplus) \times X \ar{r}{\id \times \delta} \ar{d}{\union \times \id_X} & (\FBplus \times X) \times (\FBplus \times X) \ar{r}{\delta\times \delta} & \Ran(X, \FBplus) \times \Ran(X, \FBplus) \ar{d}{\union} \\
      \FBplus \times X \ar{rr}{\delta} && \Ran(X, \FBplus)
    \end{tikzcd}
  \]
  where on the top left, $\id\times \delta$ is used to denote the map which sends the two factors of $\FBplus$ identically to the target and which acts at the diagonal map on the $X$ factor.

  Let $\matheur{F}, \matheur{G} \in \Shv(\Ran(X, \FBplus))$. Now, by pseudo-proper base change, see~\cite{gaitsgory_atiyah-bott_2015}*{Cor. 1.5.4}, we see that
  \[
    \delta^!(\matheur{F} \otimesstar \matheur{G}) = \delta^!\union_! (\matheur{F}\boxtimes \matheur{G}) \simeq (\union\times\id_X)_! (\id\times \delta)^!(\delta^! \matheur{F} \boxtimes \delta^! \matheur{G}). \teq\label{eq:delta^!_monoidal}
  \]
  Note that $(\id\times \delta)^!(\delta^! \matheur{F} \boxtimes \delta^! \matheur{G})$ in the last term is an object in $\Shv((\FBplus\times \FBplus) \times X) \simeq \Shv(X)^{\FBplus\times \FBplus}$. But now, observe that $(\union\times\id_X)_!$ is the the left Kan extension
  \[
    \Shv(\FB_+ \times \FBplus \times X) \simeq \Shv(X)^{\FBplus\times \FBplus} \to \Shv(X)^{\FBplus} \simeq \Shv(\FBplus \times X),
  \]
  used to define the Day convolution symmetric monoidal structure on $\Shv(\FBplus \times X) \simeq \Shv(X)^{\FBplus}$. Indeed, this is because its right adjoint, the pullback, is also the right adjoint to the left Kan extension. Thus, the last term of~\eqref{eq:delta^!_monoidal} is naturally identified with $\delta^! \matheur{F} \otimesshriek \delta^!\matheur{G}$, where here, $\otimesshriek$ denotes the Day convolution symmetric monoidal structure on $\Shv(\FBplus \times X)$.
\end{proof}

Since the functor $\delta^!$ is symmetric monoidal, it upgrades to a functor
\[
  \delta^!: \ComAlgstar(\Ran(X, \FBplus)) \to \ComAlg(\Shv(X)^{\FBplus})
\]
In the remainder of this subsection, we will prove the following result which is a straightforward adaptation of the corresponding result in the non-twisted setting~\cite{gaitsgory_weils_2014}*{Thm. 5.6.4}.

\begin{thm} \label{thm:tcfa_vs_tcaX}
  The functor
  \[
    \delta^!: \ComAlgstar(\Ran(X, \FBplus)) \to
    \ComAlg(\Shv(X)^{\FBplus})
  \]
  admits a fully-faithful left adjoint $\Fact$, whose essential image is $\Factstar(X, \FBplus)$. In particular, we have the following equivalence of categories
  \[
    \Fact: \ComAlg(\Shv(X)^{\FBplus}) \rightleftarrows \Factstar(X, \FBplus): \delta^!.
  \]
\end{thm}

By Lemma~\ref{lem:diagonal_closed_embedding}, we know that $\delta$ is a closed embedding. Pseudo-proper base change then implies that the unit map $\id_{\Shv(X)^{\FBplus}} \to \delta^! \delta_!$ is an equivalence, which means that $\delta_!$ is fully-faithful. The fully-faithfulness of $\Fact$ is then given by the following general category-theoretic result, see~\cite{gaitsgory_weils_2014}*{Prop. 5.6.6}.

\begin{prop}
  Let $\matheur{C}, \matheur{D}$ be non-unital symmetric monoidal categories (satisfying the conventions given in~\S\ref{subsubsec:conventions_categories}). Suppose we are given a pair of adjoint functors
  \[
    f: \matheur{C} \rightleftarrows \matheur{D}: g
  \]
  where $g$ is right-lax symmetric monoidal, so that it induces a functor $G: \ComAlg(\matheur{D}) \to \ComAlg(\matheur{C})$. Then $G$ admits a left adjoint functor $F: \ComAlg(\matheur{C}) \to \ComAlg(\matheur{D})$.

  Suppose that $g$ is continuous\footnote{i.e. it commutes with all colimits} and symmetric monoidal, and $f$ is fully-faithful, then $F$ is fully-faithful.
\end{prop}

To characterize the essential image, we need some more preparation.

\begin{prop}
  \label{prop:check_isom_on_diagonal}
  Let $f: \matheur{A} \to \matheur{B}$ be a morphism in $\Factstar(X, \FBplus)$. Then, $f$ is an equivalence if and only if $\delta^!(f)$ is an equivalence.\footnote{There is a slight ambiguity in the notation $\delta^!(f)$: we can view it as a morphism in $\ComAlg(\Shv(X)^{\FBplus})$ or as a morphism between the underlying objects in $\Shv(X)^{\FBplus}$. However, it makes no difference, since forgetting the algebra structure is conservative.}
\end{prop}

\begin{proof}
  We only need to prove the if direction. By Corollary~\ref{cor:conservativity_i_n^!}, it suffices to show that
  \[
    \iota_n^!(f): \iota_n^! \matheur{A} \to \iota_n^! \matheur{B} \teq\label{eq:equivalence_on_strata}
  \]
  is an equivalence for all $n$.

  The case where $n=1$ is given by the hypothesis since $\iota_1 = \delta$.

  Now, for any $n$, by the functor of points description of $\Ran^{=n}(X, \FBplus)$ in~\eqref{eq:functor_of_pts_description_Ran=n}, we have the following pullback square
  \[
    \begin{tikzcd}
      (\Delta(X, \FBplus))^n_{\disj} \ar{r}{\delta^n} \ar{d}{p_n} & \Ran(X, \FBplus)^n_{\disj} \ar{d}{\union_n \circ j_n = \union_n'} \\
      \Ran^{=n}(X, \FBplus) \ar{r}{\iota_n} & \Ran(X, \FBplus)
    \end{tikzcd} \teq\label{eq:pullback_union_Ran=n}
  \]
  where $\Delta(X, \FBplus) \simeq \FBplus \times X$ is the diagonal of $\Ran(X, \FBplus)$, see~\eqref{eq:description_diagonal}. Moreover, since $p_n$ is a finitary pseudo-proper morphism that is surjective on $k$-points,\footnote{In fact, $(\Delta(X, \FBplus))^n_{\disj}$ is the disjoint union of $n!$ copies of $\Ran^{=n}(X, \FBplus)$.} we know, by Lemma~\ref{lem:gait_finitary_pseudoproper_and_sheaves} b., that $p_n^!$ is conservative. Thus, to prove that~\eqref{eq:equivalence_on_strata} is an equivalence, it suffices to show that it is so after composing with $p_n^!$.

  The fact that $f$ is a morphism in $\Factstar(X, \FBplus)$ means that we have the following commutative diagram
  \[
    \begin{tikzcd}
      j_n^! \matheur{A}^{\boxtimes n} \ar{d} \ar{r}{\simeq} & j_n^!\union_n^! \matheur{A} \ar{d}  \\
      j_n^! \matheur{B}^{\boxtimes n} \ar{r}{\simeq} & j_n^!\union_n^! \matheur{B}
    \end{tikzcd}
  \]
  Applying $(\delta^n)^!$ to this diagram, we get
  \[
    \begin{tikzcd}
      (\delta^! \matheur{A})^{\boxtimes n}|_{(\Delta(X, \FBplus))^n_{\disj}} \ar{d}{\simeq} \ar{r}{\simeq} & p_n^! \iota_n^! \matheur{A} \ar{d} \\
      (\delta^! \matheur{B})^{\boxtimes n}|_{(\Delta(X, \FBplus))^n_{\disj}} \ar{r}{\simeq} & p_n^! \iota_n^! \matheur{B}
    \end{tikzcd}
  \]
  Here, the objects in the column on the right are obtained by using the commutativity of~\eqref{eq:pullback_union_Ran=n}. Moreover, the vertical arrow on the left is an equivalence by the current hypothesis. Thus, the map on the right
  \[
    p_n^! \iota_n^! \matheur{A} \to p_n^! \iota_n^! \matheur{B}
  \]
  is an equivalence as well, and we are done.
\end{proof}

\begin{lem}
  \label{lem:free_com_is_fact}
  Let $\matheur{F} \in \Shv(X)^{\FBplus}$, then
  \[
    \Free_{\ComAlg} \delta_!\matheur{F} = \bigoplus_{n>0} \Sym^n \delta_! \matheur{F} = \bigoplus_{n>0} (\delta_! \matheur{F})^{\otimesstar n}_{S_n} \in \Factstar(X, \FBplus),
  \]
  i.e. it is a twisted commutative factorization algebra over $X$.
\end{lem}
\begin{proof}
  This is a standard computation.
\end{proof}

\begin{proof}[Proof of Theorem~\ref{thm:tcfa_vs_tcaX}]
  \label{proof:thm:tcfa_vs_tcaX}
  As mentioned above, it remains to characterize the essential image of the left adjoint $\Fact$ of $\delta^!$.

  We will start by proving that for each $\matheur{A} \in \ComAlg(\Shv(X)^{\FBplus})$, $\Fact(A) \in \Factstar(X, \FBplus)$. By~\cite{lurie_higher_2017}*{\S3.2.3 \& Prop. 4.7.3.14}, we can write $\matheur{A}$ as a geometric realization of a simplicial object $\matheur{A}_\bullet$, where each term is free in $\ComAlg(\Shv(X)^{\FBplus})$. Being a left adjoint, $\Fact$ commutes with colimits, and hence, $\Fact \matheur{A} \simeq |\Fact \matheur{A}_\bullet|$. But now, since each term of $\Fact \matheur{A}_\bullet$ is, by Lemma~\ref{lem:free_com_is_fact}, factorizable, we know that $\Fact \matheur{A}$ itself is also factorizable, by Remark~\ref{rmk:geometric_realization_sifted_colim}.

  Conversely, let $\matheur{A} \in \Factstar(X, \FBplus)$ and consider the co-unit map $\Fact(\delta^!\matheur{A}) \to \matheur{A}$. By Proposition~\ref{prop:check_isom_on_diagonal}, to check that this is an equivalence, it suffices to do so after applying $\delta^!$. But now, since $\Fact$ is fully-faithful, we see that $\delta^!(\Fact(\delta^! \matheur{A})) \simeq \delta^! \matheur{A}$, and we are done.
\end{proof}

\subsection{Variants}
\label{subsec:variants_tcfas}
The main goal of the current paper is to study the cohomology groups of generalized configuration spaces of a given scheme $X$. Our strategy is to relate these to the factorization cohomology of $X$ with coefficients in some explicit factorization algebras. This subsection provides a mechanism to construct these algebras.

The material in this subsection is a straightforward adaptation of~\cite{ho_homological_2021}*{\S6.6}, with essentially identical proofs. We will thus only indicate the main constructions and results. The interested readers can consult loc. cit. for the details.

\subsubsection{One-color case}
\label{subsubsec:motivation_for_comm_alg_truncated}
For each integer $n$, we consider the category $\FBplus^{< n}$ which is a full-subcategory of $\FBplus$ consisting of non-empty finite sets of size $< n$. Let $\matheur{C}$ be a symmetric monoidal stable infinity category. Then we have the following pair of adjoint functors
\[
  \tr_{< n}: \matheur{C}^{\FBplus} \rightleftarrows \matheur{C}^{\FBplus^{< n}}: \iota_{< n} \teq\label{eq:truncation_functors_FB}
\]
with $\tr_{< n}$ being the left adjoint. Here, $\tr_{< n}$ (truncation) is given by precomposing with the embedding $\FBplus^{< n} \to \FBplus$ whereas $\iota_{< n}$ is given by extension by $0$. We will also use the notation $\tr_{\leq n}$ and $\iota_{\leq n}$ to denote the obvious functors where $\FBplus^{<n}$ is replaced by $\FBplus^{\leq n} = \FBplus^{\leq n+1}$. Moreover, we will sometimes abuse notation and use these functors to denote truncation and extension by zero functors for the category $\FB$ as well.

The category $\matheur{C}^{\FBplus^{< n}}$ admits a natural symmetric monoidal structure in the obvious way which makes $\tr_{< n}$ symmetric monoidal. Hence, $\iota_{< n}$ is right lax symmetric monoidal. The adjoint pair $\tr_{< n} \dashv \iota_{< n}$ thus upgrades to an adjoint pair between the categories of algebras
\[
  \tr_{< n}: \ComAlg(\matheur{C}^{\FBplus}) \rightleftarrows \ComAlg(\matheur{C}^{\FBplus^{< n}}): \iota_{< n}. \teq\label{eq:one_color_adjunction_truncation}
\]

\begin{rmk}
  \label{rmk:iota_<n_left_adjoint}
  Note that $\iota_{<n}$ is also a left adjoint to $\tr_{<n}$. Thus, $\iota_{<n}$ is also left-lax symmetric monoidal, and hence, we obtain a pair of adjoint functors
  \[
    \iota_{<n}: \coLie(\Vect^{\FBplus}) \rightleftarrows \coLie(\Vect^{\FBplus^{<n}}): \tr_{<n}.
  \]
  A similar discussion applies for the functors $\iota_{\leq n} \dashv \tr_{\leq n}$.
\end{rmk}

\subsubsection{}
\label{subsubsec:computation_truncation_monochrome}
When $\matheur{C} = \Vect$ and $A = \Lambda[x]_+ \in \ComAlg(\Vect^{\FBplus})$ with $x$ in degree $1$ and $S_n$ acting trivially for all $n$. Then, $\iota_{< n}(\tr_{< n} A) \simeq (k[x]/x^n)_+$. If we view $\FBplus$ as $\Ran(\pt, \FBplus)$ then $\FBplus^{< n}$ could be viewed as the sub-prestack of $\Ran(\pt, \FBplus)$ where we only allow the multiplicity of any point (in this case, only one point) to be at most $n-1$.\footnote{This is not to be confused with $\Ran^{\leq n-1}(\pt, \FBplus)$ which is in fact equivalent to $\Ran(\pt, \FBplus)$.} $\FBplus^{< n}$ is thus the generalized ordered configuration space of a point.\footnote{When $n=2$, $\FBplus^{< 2} = \{*\}$ is exactly the ordered configuration space of a point.}

In what follows, we will give a multi-colored version of this construction. A similar calculation will allows us to produce commutative factorization algebras whose factorization cohomology compute the cohomology groups of generalized ordered configuration.

\subsubsection{Generalized ordered configuration spaces}
For each $n\geq 2$, we define $\Ran(X, \FBplus)_n$ to be the open sub-prestack of $\Ran(X, \FBplus)$ where we require that the multiplicity of any point is less than $n$. More formally, for any test scheme $S$,
\[
  \Ran(X, \FBplus)_n(S) = \{(I, \gamma)\in \Ran(X, \FBplus)(S) : \gamma|_{k(s)} \text{ is at most $n-1$ to 1}, \forall s\in S\}
\]
where $\gamma|_{k(s)}: I \to X(k(s))$ is the pullback of $\gamma$ to $\Spec k(s)$.

We also have the following decomposition of $\Ran(X, \FBplus)_n$
\[
  \Ran(X, \FBplus)_n \simeq \bigsqcup_{I \in \FBplus/\iso} P^I_n(X) /S_I \teq\label{eq:decompo_Ran(X, FBplus)_n}
\]
where $P^I_n(X)$ is defined as in~\S\ref{subsec:intro_outline_results}. As in the case of $\Ran(X, \FBplus)$, $\Ran(X, \FBplus)_n$ should be viewed as a way to package $P_n(X)$ along with all the $S_k$-actions in one space.

\subsubsection{}
We will equip $\Ran(X, \FBplus)_n$ with the structure of a commutative semi-group in $\Corr(\PreStk)$, the category of prestacks with correspondences as morphisms. Note that $\Ran(X, \FBplus)$ is also a commutative semi-group in $\Corr(\PreStk)$ using the following diagram
\[
  \begin{tikzcd}
    & \Ran(X, \FBplus)^k \ar[equal]{dl} \ar{dr}{\union_k} \\
    \Ran(X, \FBplus)^k && \Ran(X, \FBplus)
  \end{tikzcd}
\]

Now, pulling back $\union_k$ along the open embedding $\iota_n: \Ran(X, \FBplus)_n \to \Ran(X, \FBplus)$, we get
\[
  \begin{tikzcd}
    & \Ran(X, \FBplus)^k \times_{\Ran(X, \FBplus)} \Ran(X, \FBplus)_n  \ar{dl}[swap]{j} \ar{dr}{\union_k}\\
    (\Ran(X, \FBplus)_n)^k && \Ran(X, \FBplus)_n
  \end{tikzcd}
\]
where $\union_k$ is pseudo-proper (being a pullback of a pseudo-proper map, Lemma~\ref{lem:pseudo_properness_union}), and $j$ an open embedding.

\subsubsection{}
This induces the $\otimesstar$-symmetric monoidal structure on $\Shv(\Ran(X, \FBplus)_n)$ as follows: for $\matheur{F}_1, \dots, \matheur{F}_k \in \Shv(\Ran(X, \FBplus)_n)$,
\[
  \matheur{F}_1\otimesstar \cdots \otimesstar \matheur{F}_k = (\union_k)_! j^!(\matheur{F}_1 \boxtimes \cdots \boxtimes \matheur{F}_k).
\]
Using this structure, we can make sense of the category of commutative algebras
\[
  \ComAlgstar(\Ran(X, \FBplus)) = \ComAlg(\Shv(\Ran(X, \FBplus))^{\otimesstar})
\]
and the full-subcategory of commutative factorization algebras $\Factstar(X, \FBplus)_n$ consisting of algebras $\matheur{A}$ such that the maps (which are adjoints to the multiplication maps)
\[
  j^!(\matheur{F} \boxtimes \cdots \boxtimes \matheur{F}) \to \union^! \matheur{F}
\]
induce equivalences
\[
  (j^!(\matheur{F} \boxtimes \cdots \boxtimes \matheur{F}))|_{(\Ran(X, \FBplus)_n)^k_{\disj}} \simeq (\union^! \matheur{F})|_{(\Ran(X, \FBplus)_n)^k_{\disj}}.
\]

\subsubsection{}
We have a pair of adjoint functors
\[
  \iota_n^! \simeq \iota_n^*: \Shv(\Ran(X, \FBplus)) \rightleftarrows \Shv(\Ran(X, \FBplus)_n): \iota_{n,*}.
\]
Arguing similarly to~\cite{ho_homological_2021}*{6.6.6--6.6.9}, one checks that $\iota_n^! \simeq \iota_n^*$ and $\iota_{n, *}$ are symmetric monoidal and moreover they preserve factorizability. Thus, we see that this upgrades to a pair of adjoint functors
\[
  \iota_n^! \simeq \iota_n^*: \Factstar(X, \FBplus) \rightleftarrows \Factstar(X, \FBplus)_n: \iota_{n, *}. \teq\label{eq:adjoint_fact_truncation}
\]

\subsubsection{Truncating a tcfa}
The reader might wonder why it is that we take $\iota_{n,*}$ rather than $\iota_{n, !}$ in~\eqref{eq:adjoint_fact_truncation} even though extension by zero appears in~\S\ref{subsubsec:motivation_for_comm_alg_truncated}. This is because with respect to $!$-pullback to the closed complement, $*$-push forward along an open sub-prestack is indeed extension by zero. We will now explain a truncation process of twisted commutative factorization algebras which is the multi-colored analog of~\S\ref{subsubsec:motivation_for_comm_alg_truncated}.

Consider the following pullback square
\[
  \begin{tikzcd}
    \FBplus^{<n} \times X \ar{r}{\delta} \ar{d}{\iota_n} & \Ran(X, \FBplus)_n \ar{d}{\iota_n} \\
    \FBplus \times X \ar{r}{\delta} & \Ran(X, \FBplus)
  \end{tikzcd}
\]
Then, we have the following equivalences
\[
  \delta^! \circ \iota_{n, *} \circ \iota_n^! \circ \Fact \simeq \iota_{n, *} \circ \delta^! \circ \iota_n^! \circ \Fact \simeq \iota_{n, *} \circ \iota_n^! \circ \delta^! \circ \Fact \simeq \iota_{n, *} \circ \iota_n^! \teq\label{eq:truncation_multi_color}
\]
of functors
\[
  \ComAlg(\Shv(X)^{\FBplus}) \to \ComAlg(\Shv(X)^{\FBplus}).
\]

Now, the RHS of~\eqref{eq:truncation_multi_color} is just truncating everything of degrees $\geq n$ and we have the following

\begin{prop}
  \label{prop:truncation}
  Let $\matheur{A} \in \ComAlg(\Shv(X)^{\FBplus})$. Then $\delta^!(\iota_{n, *}(\iota_n^!(\Fact(\matheur{A}))))$ is obtained from $\matheur{A}$ by setting all components $\matheur{A}^{(I)} \in \Shv(X^I/S_I)$ to $0$ for all $|I| \geq n$.
\end{prop}

\section{Factorization (co)homology}
\label{sec:factorization_coh}

Having laid down the foundation for the coefficient objects, twisted commutative factorization algebras, we will now turn to factorization (co)homology itself. In~\S\ref{subsec:mult_push_forward}, we introduce factorization homology as a multiplicative push-forward functor, acting on algebra objects rather than mere sheaves, and show that under the equivalence given in Theorem~\ref{thm:tcfa_vs_tcaX}, this multiplicative push-forward functor amounts to pushing forward along $\Ran(X, \FBplus) \to \FBplus$. Then, we show, in~\S\ref{subsec:Koszul_duality_vs_mult_pushforward} that factorization homology behaves nicely with respect to Koszul duality. In~\S\ref{subsec:fact_coh} we introduce the cohomological version of factorization homology, and show that it also behaves nicely with respect Koszul duality. The materials in this section is mostly routine; we only need adapt slightly the usual story to the twisted commutative setting.

\subsection{A multiplicative push-forward}
\label{subsec:mult_push_forward}
Let $\pi: X\to \pt$ be the structure map of $X$ and
\[
  \Ran(\pi): \Ran(X, \FBplus) \to \Ran(\pt, \FBplus) \simeq \FBplus
\]
be the induced map, see~\eqref{eq:Ran(pi)} and Remark~\ref{rmk:Ran(f)}. By \S\ref{subsubsec:conv_monoidal_vs_Ran(f)}, the functor
\[
  \Ran(\pi)_!: \Shv(\Ran(X, \FBplus)) \to \Vect^{\FBplus}
\]
is symmetric monoidal with respect to the $\otimesstar$-monoidal structure on the LHS and the Day convolution monoidal structure on the RHS. Thus, its right adjoint, $\Ran(\pi)^!$ is right lax symmetric monoidal, and hence, the adjoint pair $\Ran(\pi)_! \dashv \Ran(\pi)^!$ upgrades to a pair of adjoint functors
\[
  \Ran(\pi)_!: \ComAlgstar(\Ran(X, \FBplus)) \rightleftarrows \ComAlg(\Vect^{\FBplus}): \Ran(\pi)^!.
\]

\begin{defn}
  Let $\matheur{A} \in \Factstar(X, \FBplus)$, the factorization homology of $X$ with coefficients in $\matheur{A}$ is defined to be $\Ran(\pi)_!(\matheur{A}) \in \ComAlg(\Vect^{\FBplus})$.
\end{defn}

\begin{rmk}
  Note that this is not the same as $C^*_c(\Ran(X, \FBplus), \matheur{A})$. Indeed, $C^*_c(\Ran(X, \FBplus), \matheur{A})$ can be obtained from $\Ran(\pi)_! (\matheur{A})$ by taking $S_n$-coinvariant at each degree. See also Remark~\ref{rmk:taking_coinvariant}.
\end{rmk}

\subsubsection{} Observe that the functor
\[
  \pi^!: \Vect^{\FBplus} \to \Shv(X)^{\FBplus} \label{eq:pi_upper_!}
\]
is symmetric monoidal, where, $\pi$ is the structure map of $X$.\footnote{Alternatively, we can also view $\pi: \FBplus \times X \to \FBplus$ as the structure map of $\FBplus \times X$ induced by the structure map of $X$. The name structure map is justified because the prestacks we consider are defined over the constant prestack $\FBplus$, and we can view $\FBplus$ as the base prestack. In the same vein, $\Ran(\pi): \Ran(X, \FBplus) \to \FBplus$ can also be viewed as the structure map of $\Ran(X, \FBplus)$.} Thus, it upgrades to a functor
\[
  \pi^!: \ComAlg(\Vect^{\FBplus}) \to \ComAlg(\Shv(X)^{\FBplus}).
\]
By adjoint functor theorem~\cite{lurie_higher_2017-1}*{Cor. 5.5.2.9}, we have the following pair of adjoint functors
\[
  \pi_{?}: \ComAlg(\Shv(X)^{\FBplus}) \rightleftarrows \ComAlg(\Vect^{\FBplus}): \pi^!.
\]

\subsubsection{} These two constructions in fact coincide and we have the following result.

\begin{prop}
  Let $\matheur{A} \in \Factstar(X, \FBplus)$. Then we have a natural equivalence
  \[
    \pi_? (\delta^! \matheur{A}) \simeq \Ran(\pi)_!(\matheur{A}).
  \]
\end{prop}
\begin{proof}
  Consider the following diagram
  \[
    \begin{tikzcd}
      \ComAlg(\Shv(X)^{\FBplus}) \ar[shift right=\arrdisp]{dr}[swap]{\pi_?} \ar[shift left=\arrdisp]{r}{\Fact} & \ComAlgstar(\Ran(X, \FBplus)) \ar[shift left=\arrdisp]{l}{\delta^!} \ar[shift right=\arrdisp]{d}[swap]{\Ran(\pi)_!}  \\
      & \ComAlg(\Vect^{\FBplus}) \ar[shift right=\arrdisp]{ul}[swap]{\pi^!} \ar[shift right=\arrdisp]{u}[swap]{\Ran(\pi)^!}
    \end{tikzcd}
  \]
  Clearly, the right adjoints commute, and hence, so do the left adjoints, and we are done.
\end{proof}

\subsection{Koszul duality}
\label{subsec:Koszul_duality_vs_mult_pushforward}
As discovered in~\cite{francis_chiral_2011}, the theory of Koszul duality interacts nicely with factorization homology. In this subsection, we will explain the extension to the twisted commutative setting.

\subsubsection{} By~\S\ref{subsubsec:FBplus_cat_pro-nilp}, $\Shv(X)^{\FBplus}$ is pro-nilpotent, and hence, Proposition~\ref{prop:koszul_duality_pro-nilp} says that we have mutually inverse functors
\[
  \coPrim[1]: \ComAlg(\Shv(X)^{\FBplus}) \rightleftarrows \coLie(\Shv(X)^{\FBplus}): \coChev
\]

\subsubsection{}
The functor $\pi^!$ of~\eqref{eq:pi_upper_!} is symmetric monoidal and hence, $\pi_!$, its left adjoint, is left-lax monoidal. Thus, we get a pair of adjoint functors
\[
  \pi_!: \coLie(\Shv(X)^{\FBplus}) \rightleftarrows \coLie(\Vect^{\FBplus}): \pi^!.
\]

Since $\pi^!$, being a right adjoint, commutes with limits, the right adjoints of the diagram below commute. Thus, so do the left adjoints.
\[
  \begin{tikzcd}[column sep=large]
    \ComAlg(\Shv(X)^{\FBplus}) \ar[shift left=\arrdisp]{r}{\coPrim[1]} \ar[shift right=\arrdisp]{d}[swap]{\pi_?} & \coLie(\Shv(X)^{\FBplus}) \ar[shift left=\arrdisp]{l}{\coChev} \ar[shift right=\arrdisp]{d}[swap]{\pi_!} \\
    \ComAlg(\Vect^{\FBplus}) \ar[shift right=\arrdisp]{u}[swap]{\pi^!} \ar[shift left=\arrdisp]{r}{\coPrim[1]} & \coLie(\Vect^{\FBplus}) \ar[shift right=\arrdisp]{u}[swap]{\pi^!} \ar[shift left=\arrdisp]{l}{\coChev}
  \end{tikzcd}
\]

Since all the horizontal arrows are equivalences, we immediately obtain the following

\begin{prop} \label{prop:coChev_coPrim_vs_pi!}
  We have the following natural equivalences of functors
  \[
    \pi_? \circ \coChev \simeq \coChev \circ \pi_! \qquad\text{and}\qquad \coPrim[1] \circ \pi^! \simeq \pi^! \circ \coPrim[1].
  \]
\end{prop}

\subsection{Factorization cohomology}
\label{subsec:fact_coh}
For the applications of this paper, we will need the cohomological variant of factorization homology. This is because our main results are really about the cohomology groups of ordered configuration spaces.

In the theory of constructible sheaves on schemes, one starts with the construction of sheaf cohomology. Then, for any scheme $X$ which admits a compactification $j: X\hookrightarrow \lbar{X}$, cohomology with compact support is defined by first extending the sheaf on $X$ by zero using $j_!$ and then apply the usual sheaf cohomology to the resulting sheaf on $\lbar{X}$.

Since the sheaf theory on prestacks is biased toward the $!$-functors, we can only take cohomology with compact support usually. However, suppose $X$ admits a compactification $j: X\to \lbar{X}$, we can make sense of factorization cohomology. Unsurprisingly, this is obtained by first extending the sheaf from $X$ to $\lbar{X}$ using $j_*$ and then take factorization homology.

\subsubsection{Construction} Let $j: X\to \lbar{X}$ be a compactification of $X$, i.e. $j$ open embedding and $\lbar{X}$ is proper. Note that both functors
\[
  j^* \simeq j^!: \Shv(\lbar{X}) \rightleftarrows \Shv(X): j_*
\]
are symmetric monoidal with respect to the $\otimesshriek$-monoidal structures on both sides. Thus, the right/left adjoints commute with right/left adjoints in the diagram below.
\[
  \begin{tikzcd}[column sep=large]
    \ComAlg(\Shv(\lbar{X})^{\FBplus}) \ar[shift left=\arrdisp]{r}{\coPrim[1]} \ar[shift right=\arrdisp]{d}[swap]{j^*} & \coLie(\Shv(\lbar{X})^{\FBplus}) \ar[shift left=\arrdisp]{l}{\coChev} \ar[shift right=\arrdisp]{d}[swap]{j^*} \\
    \ComAlg(\Shv(X)^{\FBplus}) \ar[shift right=\arrdisp]{u}[swap]{j_*} \ar[shift left=\arrdisp]{r}{\coPrim[1]} & \coLie(\Shv(X)^{\FBplus}) \ar[shift right=\arrdisp]{u}[swap]{j_*} \ar[shift left=\arrdisp]{l}{\coChev}
  \end{tikzcd}
\]
Since all horizontal arrows are equivalences, everything commutes.

\subsubsection{} Arguing similarly, everything in the diagram below commutes
\[
  \begin{tikzcd}[column sep=large]
    \ComAlg(\Shv(\lbar{X})^{\FBplus}) \ar[shift left=\arrdisp]{r}{\Fact} \ar[shift right=\arrdisp]{d}[swap]{j^*} & \Factstar(\lbar{X}, \FBplus) \ar[shift left=\arrdisp]{l}{\delta^!} \ar[shift right=\arrdisp]{d}[swap]{\Ran(j)^*} \\
    \ComAlg(\Shv(X)^{\FBplus}) \ar[shift right=\arrdisp]{u}[swap]{j_*} \ar[shift left=\arrdisp]{r}{\Fact} & \Factstar(X, \FBplus) \ar[shift right=\arrdisp]{u}[swap]{\Ran(j)_*} \ar[shift left=\arrdisp]{l}{\delta^!}
  \end{tikzcd}
\]
Here, $\Ran(j)$ is defined as in Remark~\ref{rmk:Ran(f)}. Moreover, it is easy to see that $\Ran(j)$ is an open embedding, since $j$ is an open embedding.

\begin{defn}
  We define the functor of taking factorization cohomology
  \[
    \pi_{?*}: \ComAlg(\Shv(X)^{\FBplus}) \to \ComAlg(\Vect^{\FBplus})
  \]
  to be $\pi_{?*} = \lbar{\pi}_{?} \circ j_*$ where $\pi$ and $\lbar{\pi}$ are the structure maps of $X$ and $\lbar{X}$ respectively.
\end{defn}

\subsubsection{}
Observe that
\begin{align*}
  \pi_{?*} \simeq \lbar{\pi}_{?} \circ j_* \simeq \Ran(\lbar{\pi})_! \circ \Fact \circ j_* \simeq \Ran(\lbar{\pi})_! \circ \Ran(j)_* \circ \Fact.
\end{align*}
It is easy to see that $\Ran(\lbar{\pi})$ is a proper morphism, $\Ran(j)$ and $\Ran(\lbar{\pi}) \circ \Ran(j)$ are schematic, i.e. the pullback over any scheme is a scheme. Thus, by~\cite{ho_free_2017}*{Prop. 2.11.4},
\[
  \Ran(\lbar{\pi})_! \circ \Ran(j)_* \simeq \Ran(\pi)_*
\]
and hence, $\pi_{?*} \simeq \Ran(\pi)_* \circ \Fact$.

In more concrete terms, using the fact that $\Fact$ and $\delta^!$ are mutually inverses, we see that for any $\matheur{A} \in \Factstar(X, \FBplus)$, we have
\[
  \pi_{?*}(\delta^! \matheur{A}) \simeq \bigoplus_{I \in \FBplus/\iso} C^*(X^I, \matheur{A}^{(I)}) \in \ComAlg(\Vect^{\FBplus}) \teq \label{eq:fact_coh_vs_powers}
\]
where $\matheur{A}^{(I)}$ is the $I$-th component of $\matheur{A}$, which is an $S_I$-equivariant sheaf on $X^I$. In particular, $\pi_{?*}$ is independent of the choice of a compactification of $X$.

\subsubsection{Interaction with Koszul duality}
Similarly to $\pi_?$, the functor $\pi_{?*}$ also behaves nicely with respect to Koszul duality and we have the following result.

\begin{thm} \label{thm:pi_?*_vs_coChev}
  We have a natural equivalence of functors
  \[
    \pi_{?*} \circ \coChev \simeq \coChev \circ \pi_*
  \]
  as functors
  \[
    \coLie(\Shv(X)^{\FBplus}) \to \ComAlg(\Vect^{\FBplus}).
  \]
\end{thm}
\begin{proof}
  Indeed, for any $\mathfrak{a} \in \coLie(\Shv(X)^{\FBplus})$, we have
  \[
    \pi_{?*}(\coChev \mathfrak{a}) \simeq \lbar{\pi}_?(j_*(\coChev \mathfrak{a})) \simeq \lbar{\pi}_?(\coChev(j_* \mathfrak{a})) \simeq \coChev(\lbar{\pi}_!(j_* \mathfrak{a})) \simeq \coChev(\pi_*(\mathfrak{a})),
  \]
  where, as above, $\pi$ and $\lbar{\pi}$ are structure maps of $X$ and $\lbar{X}$ respectively. Moreover, the first and last equivalences are by definition, the second equivalence by the fact that $j_*$ commutes with limit (being a right adjoint) and is symmetric monoidal, and the third equivalence is by Proposition~\ref{prop:coChev_coPrim_vs_pi!}. Note also that here, the $\coLie$-coalgebra structure on $\pi_* \mathfrak{a}$ is given by the fact that $\pi_*$ is left-lax monoidal. Indeed, $\pi_* \simeq \lbar{\pi}_! \circ j_*$ where $j_*$ is symmetric monoidal and $\lbar{\pi}_!$ is left-lax.
\end{proof}

\begin{rmk}
  \label{rmk:coLie_structure_on_pi_*pi^!a}
  When the $\coLie$-coalgebra involved is of the form $\pi^! \mathfrak{a}$ for some $\mathfrak{a} \in \coLie(\Vect^{\FBplus})$, we have a more explicit way to think about the $\coLie$-coalgebra structure on $\pi_* \mathfrak{a}$. Indeed,
  \[
    \pi_* \pi^! \mathfrak{a} \simeq C^*(X, \omega_X) \otimes \mathfrak{a} \teq\label{eq:pull_push_coLie}
  \]
  where $\omega_X$ is the dualizing sheaf on $X$. Here, $C^*(X, \omega_X) \in \Vect$ and the tensor written above is given by the obvious action of $\Vect$ on $\Vect^{\FBplus}$. Now, $C^*(X, \omega_X)$ has a natural co-commutative coalgebra since $C^*(X, -)$ is left lax symmetric monoidal and $\omega_X$ is naturally a co-commutative coalgebra via $\omega_X \otimesshriek \omega_X \simeq \omega_X$. The $\coLie$-coalgebra structure on $C^*(X, \omega_X) \otimes \mathfrak{a}$ is induced by that on $\mathfrak{a}$ and the co-commutative coalgebra structure on $C^*(X, \omega_X)$. See also~\cite{gaitsgory_study_2017}*{Vol. II, \S6.1.2} and \cite{ho_homological_2021}*{\S6.3.7}.
\end{rmk}

\begin{rmk}
  \label{rmk:triv_coLie_structure_C^*(X, omega)_trivial}
  When the co-commutative coalgebra structure on $C^*(X, \omega_X)$ is trivial (for example when $X = \mathbb{A}^d$) then the induced $\coLie$-coalgebra structure on $\pi_* \mathfrak{a}$ is trivial. Thus,
  \[
    \pi_{?*}(\coChev(\pi^!\mathfrak{a})) \simeq \coChev(\pi_*(\pi^! \mathfrak{a})) \simeq \coChev(C^*(X, \omega_X) \otimes \mathfrak{a}) \simeq \Sym(C^*(X, \omega_X)\otimes \mathfrak{a}[-1])_+.
  \]
\end{rmk}

\subsubsection{Unital factorization cohomology} We will now discuss a unital variant of factorization cohomology. Let $\matheur{A} \in \ComAlg^{\un, \aug}(\Shv(X)^{\FBplus})$ with augmentation ideal $\matheur{A}_+$ (see \S\ref{subsubsec:unit_and_augmentation} for the notation). We define
\[
  \pi_{?*}^{\un} \matheur{A} = \Lambda \oplus \pi_{?*} \matheur{A}_+ \in \ComAlg^{\un, \aug}(\Vect^{\FBplus})
\]
where $\Lambda$ is in degree $0$.

Let $\mathfrak{a} \in \coLie(\Shv(X)^{\FBplus})$ be the Koszul dual of $\matheur{A}_+$. Then
\[
  \pi_{?*}^{\un} \matheur{A} \simeq \Lambda \oplus \coChev(\pi_* \mathfrak{a}) \simeq \coChev^{\un}(\pi_* \mathfrak{a}),
\]
where the last equivalence is by definition, see~\S\ref{subsubsec:unital_variant_Koszul}.

When $\matheur{A} \in \ComAlg(\Shv(X)^{\FBplus})$ with Koszul dual $\mathfrak{a} \in \coLie(\Shv(X)^{\FBplus})$, we will abuse notation and also write
\[
  \pi^{\un}_{?*} \matheur{A} = \Lambda \oplus \pi_{?*} \matheur{A} \simeq \coChev^{\un}(\pi_* \mathfrak{a}). \teq\label{eq:unital_pi?*_vs_Koszul}
\]

\subsubsection{}
We also have the following unital version of the twisted commutative $\Ran$ space, $\Ran(X, \FB)$, given by the following functor of points description: for each scheme $S$
\[
  \Ran(X, \FB)(S) = \{(I, \gamma): I\in \FB, I \xrightarrow{\gamma} X(S)\}.
\]
Clearly, $\Ran(X, \FB) \simeq \{\pt\} \sqcup \Ran(X, \FBplus)$. And moreover, if we write
\[
  \Ran^{\un}(\pi): \Ran(X, \FB) \to \Ran(\pt, \FB) \simeq \FB
\]
for the obvious projection map, then for any $\matheur{A} \in \ComAlg(\Shv(X)^{\FBplus})$, we have
\[
  \pi_{?*}^{\un}(\matheur{A}) \simeq \Ran^{\un}(\pi)_*(\Fact(\matheur{A})) \simeq \bigoplus_{I \in \FB/\iso} C^*(X^I, \matheur{A}^{(I)}). \teq\label{eq:fact_coh_vs_powers_unital}
\]

\section{Representation stability for generalized ordered configuration spaces}
\label{sec:rep_stab}

We will now use the machinery developed thus far to study higher representation stability for generalized ordered configuration spaces. In~\S\ref{subsec:tcfa_for_gen_conf}, we relate the cohomology of generalized configuration spaces to factorization cohomology with coefficients in certain truncated twisted commutative algebras. In~\S\ref{subsec:coChev_gen_conf}, using Koszul duality, we rewrite them in terms of cohomological Chevalley complexes of certain twisted $\coLie$-coalgebras.

For our purposes, it is most convenient to phrase $\FI$-modules and higher analogs in terms of modules over certain twisted commutative algebra. When the $\coLie$-coalgebra involved is trivial, all the modules that appear will be free over the respective twisted commutative algebra. We make the necessary preparation for this case in in~\S\ref{subsec:splitting_Lie_i-acyclic}. In~\S\ref{subsec:splitting_Lie_general}, we introduce the $\TopTriv_m$ condition, a strictly weaker notion, and show that in this case, the cohomology of ordered configuration spaces still form a free module over a certain explicit twisted commutative algebra,  generalizing a previous result of~\cite{church_fi-modules_2015}.

Finally, in~\S\ref{subsec:sketch_higher_rep_stab}, we sketch the main ideas that go into the two flavors of higher representation stability studied in the paper. The actual execution of these ideas are then carried out in~\S\ref{subsec:higher_rep_stab_general} and~\S\ref{subsec:higher_rep_stab_TopTrivm}.

Throughout this section, we assume that $X$ is an equi-dimensional scheme of dimension $d$.\footnote{Most proofs will still go through without this restriction, but with more complicated notation to keep track of the dimension.}

\subsection{Twisted commutative factorization algebras}
\label{subsec:tcfa_for_gen_conf}
In this subsection, we will produce twisted commutative factorization algebras whose factorization cohomologies compute cohomology of generalized configuration spaces.

\subsubsection{The simplest case: powers of the space}
We start with a warm-up: the cohomology of $\bigsqcup_{I\in \FB/\iso} X^I$. We will use the following

\begin{notation}[Renormalized dualizing sheaf]
  \label{not:renormalized_dualizing_sheaf}
  Let $X$ be an equi-dimensional scheme of dimension $d$. We write
  \[
    \sOmega_X = \omega_X[-2d](-d),
  \]
  and
  \[
    \sOmega_{\Ran(X, \FBplus)} = \bigoplus_{I\in \FBplus/\iso} \omega_{X^I}[-2d|I|](-d|I|).
  \]
\end{notation}

\begin{rmk}
  Sheaf cohomology of $X$ with coefficients in $\sOmega_X$ could be thought of as a shifted (or renormalized) version of Borel--Moore homology. When $X$ is smooth of dimension $d$, all of these sheaves are just constant sheaves with values in $\Lambda$ (with no cohomological shift and no Tate twist). In this case, renormalized Borel--Moore homology is just the usual cohomology. Because of this reason, we will loosely use the term cohomology throughout the paper to refer to renormalized Borel--Moore homology.
\end{rmk}

\subsubsection{} Let $\freeAlgOr_+ \in \ComAlg(\Vect^{\FB})$ be the free twisted commutative algebra generated by $\Lambda_1[-2d](-d) \in \Vect^{\FBplus}$, where $\Lambda_1$ is the $1$-dimensional $\Lambda$-vector space put in degree $1$ with the (necessarily) trivial $S_1$ action. By Lemma~\ref{lem:free_twisted_commutative_alg},
\[
  \freeAlgOr_+ = \Free_{\ComAlg} (\Lambda_1[-2d](-d))
  \simeq \bigoplus_{n>0} \Lambda_n[-2dn](-dn)
  \simeq \Lambda[x]_+,
\]
where $x$ is in graded degree $1$, cohomological degree $2d$, and Tate twist $-d$. Moreover, all the $S_n$-actions are trivial. $\freeAlgOr_+$ is the augmentation ideal of $\freeAlgOr = \Lambda[x]$.

In what follows, we will use $\freeAlgOr_+(X) \in \ComAlg(\Shv(X)^{\FBplus})$ and $\freeAlgOr(X) \in \ComAlg^{\un, \aug}(\Shv(X)^{\FB})$ to denote $\pi^!(\freeAlgOr_+)$ and $\pi^!(\freeAlgOr)$ respectively.

\begin{lem}
  \label{lem:computing_fact_pi!_freeAlgOr}
  We have a natural equivalence
  \[
    \Fact(\freeAlgOr_+(X)) \simeq \sOmega_{\Ran(X, \FBplus)}
  \]
  of sheaves on $\Ran(X, \FBplus)$.
\end{lem}
\begin{proof}
  Since $\pi^!$ is continuous and symmetric monoidal, it preserves free commutative algebra objects, i.e. $\pi^! \circ \Free_{\ComAlg} \simeq \Free_{\ComAlg} \circ \pi^!$. Using the fact that $\Fact$ is the left adjoint of $\delta^!$ (Theorem~\ref{thm:tcfa_vs_tcaX}), it is also easy to see that $\Fact \circ \Free_{\ComAlg} \simeq \Free_{\ComAlg} \circ \delta_!$. Applying to the case of $\freeAlgOr_+$, we see that
  \[
    \Fact(\freeAlgOr_+(X)) = \Fact(\pi^! \freeAlgOr_+) = \Fact(\pi^!(\Free_{\ComAlg} \Lambda_1[-2d](-d))) \simeq \Free_{\ComAlg}(\delta_! \sOmega_X),
  \]
  where in the last term, $\sOmega_X \in \Shv(X)$ is viewed as an object in $\Shv(X)^{\FBplus}$ by putting it in degree $1$. Moreover, also in the last term, $\Free_{\ComAlg}$ is taken in $\ComAlgstar(\Ran(X, \FBplus))$.

  For any $I \in \FBplus$, $(\delta_! \sOmega_X)^{\boxtimes I} \in \Shv(\Ran(X, \FBplus)^{I})$ is supported only on $X^I$ embedded diagonally. Moreover, by degree considerations, we see that the $I$-th component $\Free_{\ComAlg}(\delta_! \sOmega_X)^{(I)} \in \Shv(X^I)$ is given by $(\delta_! \sOmega_X)^{\otimesstar I}_{S_I}$. Thus, using the pseudo-proper base change theorem for the following diagram,
  \[
    \begin{tikzcd}
      S_I \times X^I \ar{dd}{p} \ar{r} & X^I \ar{d} \\
      & \Ran(X, \FBplus)^{I} \ar{d} \\
      X^I \ar{r} & \Ran(X, \FBplus)
    \end{tikzcd}
  \]
  we get
  \[
    \Free_{\ComAlg}(\delta_! \sOmega_X)^{(I)} \simeq (\delta_! \sOmega_X)^{\otimesstar I}_{S_I} \simeq (p_! \sOmega_{S_I\times X^I})_{S_I} \simeq \sOmega_{X^I},
  \]
  where $p$ is just a covering of $X^I$ by $|S_I|$ disjoint copies of itself. Thus, we are done.
\end{proof}

The computation above allows us to realize cohomology of $X^I$ as factorization cohomology.
\begin{cor}
  \label{cor:fact_coh_vs_X^I}
  We have natural equivalences
  \[
    \pi_{?*} (\freeAlgOr_+(X)) \simeq \Ran(\pi)_*(\Fact(\pi^! \freeAlgOr_+)) \simeq \Ran(\pi)_*(\sOmega_{\Ran(X, \FBplus)}) \simeq \bigoplus_{n>0} C^*(X^n, \sOmega_{X^n}).
  \]
  And hence,
  \[
    \pi_{?*}^{\un} (\freeAlgOr_+(X)) \simeq \bigoplus_{n\geq 0} C^*(X^n, \sOmega_{X^n}).
  \]
\end{cor}
\begin{proof}
  This is a direct consequence of~\eqref{eq:fact_coh_vs_powers},~\eqref{eq:fact_coh_vs_powers_unital}, and Lemma~\ref{lem:computing_fact_pi!_freeAlgOr}.
\end{proof}

\subsubsection{Generalized configuration spaces}
For each integer $n\geq 2$, we let
\[
  \algOr{n} = \freeAlgOr/x^n = \Lambda[x]/x^n = \bigoplus_{k=0}^{n-1} \Lambda_k[-2dk](-dk)
\]
obtained from $\freeAlgOr$ by truncating the part of degree $n$ or above. As above, we will use $\algOr{n, +}$ to denote the augmentation ideal of $\algOr{n}$, and $\algOr{n, +}(X) \in \ComAlg(\Shv(X)^{\FBplus})$ (resp. $\algOr{n}(X) \in \ComAlg^{\un, \aug}(\Shv(X)^{\FBplus})$) the pullbacks $\pi^! \algOr{n, +}$ (resp. $\pi^!\algOr{n}$).

\begin{lem}
  We have a natural equivalence
  \[
    \Fact(\algOr{n, +}(X)) \simeq \iota_{n, *}\iota_n^! \sOmega_{\Ran(X, \FBplus)} \simeq \iota_{n, *}\iota_n^* \sOmega_{\Ran(X, \FBplus)}
  \]
  where $\iota_n: \Ran(X, \FBplus)_n \hookrightarrow \Ran(X, \FBplus)$ is the open embedding.
\end{lem}
\begin{proof}
  The desirable conclusion is obtained by combining Lemma~\ref{lem:computing_fact_pi!_freeAlgOr}, Proposition~\ref{prop:truncation}, and Theorem~\ref{thm:tcfa_vs_tcaX}.
\end{proof}

Using this Lemma and we immediately obtain the following result.

\begin{prop}
  \label{prop:fact_coh_vs_pconf}
  We have natural equivalences
  \[
    \pi_{?*}(\algOr{n, +}(X)) \simeq \bigoplus_{k>0} C^*(P^k_n(X), \sOmega_{P^k_n(X)}) \quad\text{and}\quad \pi_{?*}^{\un}(\algOr{n, +}(X)) = \bigoplus_{k\geq 0} C^*(P^k_n(X), \sOmega_{P^k_n(X)}).
  \]
\end{prop}

In what follows, to simplify the notations, we will use the following notation.
\begin{notation}
  Let $n\geq 2$ be an integer, we use $\alg{n, +}(X)$ to denote
  \[
    \alg{n, +}(X) = \pi_{?*}(\algOr{n, +}(X)) \simeq \bigoplus_{k>0} C^*(P^k_n(X), \sOmega_{P^k_n(X)}),
  \]
  and similarly for the unital version
  \[
    \alg{n}(X) = \pi_{?*}^{\un}(\algOr{n, +}(X)) = \bigoplus_{k\geq 0} C^*(P^k_n(X), \sOmega_{P^k_n(X)}).
  \]
\end{notation}

\subsection{Cohomological Chevalley complexes}
\label{subsec:coChev_gen_conf}
We will now apply the machinery of Koszul duality to study $\alg{n, +}(X)$. To start, we will study the Koszul duals $\liealgOr{n} = \coPrim[1](\algOr{n, +}) \in \coLie(\Vect^{\FBplus})$ of $\algOr{n, +}$. The case of $n=2$ is particularly simple.

\begin{lem}
  \label{lem:a_2_is_free}
  We have
  \[
    \liealgOr{2} \simeq \coFree_{\coLie}(\Lambda_1[-2d+1](-d)) = \bigoplus_{k>0} (\coLie(k) \otimes (\Lambda_1[-2d+1](-d))^{\otimes k})_{S_k}.
  \]
  In particular, for each $I\in \FBplus$, the $I$-th component $(\liealgOr{2})_I$ of $\liealgOr{2}$ concentrates in one degree $(2d-1)|I|$ and is of finite dimension.\footnote{See also Corollary~\ref{cor:a_2_vs_coh_P_2} for an expression of $\liealgOr{2}$ in terms of the cohomology of $P_2(\mathbb{A}^d)$, the ordered configuration space of $\mathbb{A}^d$.}
\end{lem}
\begin{proof}
  The first statement is due to the fact that Koszul duality exchanges free objects with trivial objects (see Proposition~\ref{prop:koszul_duality_pro-nilp}) and moreover, $\algOr{2, +}$ is a trivial algebra object generated by $\Lambda_1[-2d](-d)$ (note the shift involved in Koszul duality). The second statement is immediate from the first.
\end{proof}

For a general $n$, the control we have over $\liealgOr{n}$ is not as precise. However, we have the following.

\begin{lem}
  \label{lem:homological_est_a_n}
  Let $n>2$ be an integer.
  \begin{myenum}{(\alph*)}
  \item The degree $1$ component of $\liealgOr{n}$ is $\Lambda_1[-2d+1](-d)$.
  \item The $I$-th component $(\liealgOr{n})_I$ vanishes when $1 < |I| < n$.
  \item When $|I| \geq n$, $(\liealgOr{n})_I$ concentrates in degrees $[(2d-1)|I|, (2d(n-1)-1)|I|/(n-1)]$ with finite dimensional cohomology.
  \end{myenum}
\end{lem}
\begin{proof}
  The truncation functor $\tr_{<n}: \Vect^{\FB_+} \to \Vect^{\FBplus^{<n}}$ of~\S\ref{subsubsec:motivation_for_comm_alg_truncated} is symmetric monoidal, continuous, and co-continuous. Thus, it commutes with the functor $\coPrim[1]$. In particular
  \[
    \tr_{<n} \liealgOr{n} = \tr_{<n}(\coPrim[1](\algOr{n,+})) \simeq \coPrim[1](\tr_{<n}(\algOr{n, +})).
  \]
  But note that $\tr_{<n} \algOr{n, +}$ is the free commutative algebra in $\Vect^{\FBplus^{<n}}$, and since Koszul duality exchanges free and trivial objects, see Proposition~\ref{prop:koszul_duality_pro-nilp}, we obtain parts (a) and (b).

  For part (c), we will make use of Proposition~\ref{prop:computing_coPrim[1]_as_colimit}. Indeed, we have
  \[
    \coPrim[1] (\algOr{n, +}) \simeq \colim_{k \geq 1} \coPrim^k[1](\algOr{n, +})
  \]
  where
  \begin{align*}
    \coFib(\coPrim^{k-1}[1] \algOr{n, +} \to \coPrim^k[1] \algOr{n, +})
    &\simeq (\coLie(k) \otimes (\algOr{n,+}[1])^{\otimes k})_{S_k}
  \end{align*}

  From the explicit description of $\algOr{n, +}$, we see that for each $I\in \FBplus$ with $|I| \geq n$, $((\algOr{n,+}[1])^{\otimes k})_{I}$ lives in cohomological degrees $[(2d-1)|I|, (2d(n-1)-1)|I|/(n-1)]$ with finite dimensional cohomology. Moreover, for each $I$, the colimit is finite, i.e. it is equivalent to $(\coPrim^k[1] \algOr{n, +})_I$ when $k\gg 0$. Part (c) thus follows.
\end{proof}

\subsubsection{} As in the case of $\freeAlgOr_+$ and $\algOr{n, +}$, we will use $\trivLieAlgOr(X)$ and $\liealgOr{n}(X) \in \coLie(\Shv(X)^{\FBplus})$ to denote the pullbacks $\pi^! \trivLieAlgOr$ and $\pi^! \liealgOr{n}$. Since $\pi^!$ is symmetric monoidal and is both continuous and co-continuous, it is compatible with Koszul duality. Namely $\trivLieAlgOr(X)$ and $\liealgOr{n}(X)$ are Koszul duals of $\freeAlgOr_+(X)$ and $\algOr{n, +}(X)$ respectively. We also let
\[
  \liealg{n}(X) = C^*(X, \omega_X) \otimes \liealgOr{n} \simeq \pi_* \pi^! \liealgOr{n}.
\]
The following theorem says that $\liealg{n}(X)$ is the Koszul dual of $\alg{n, +}(X)$.

\begin{thm}
  \label{thm:coh_conf_vs_coChev}
  When $n\geq 2$, we have the following equivalences
  \[
    \alg{n, +}(X) \simeq \coChev(C^*(X, \omega_X) \otimes \liealgOr{n}) = \coChev(\liealg{n}(X)).
  \]
  We also have similar equivalences for the unital versions (by adding a copy of $\Lambda$ at degree $0$).
\end{thm}
\begin{proof}
  This is a direct consequence of Proposition~\ref{prop:fact_coh_vs_pconf}, Theorem~\ref{thm:pi_?*_vs_coChev},~\eqref{eq:unital_pi?*_vs_Koszul}, and~\eqref{eq:pull_push_coLie}.
\end{proof}

\begin{rmk}
  Note that the corresponding statement for $X^I$ is trivial since it can be seen directly by \Kunneth{} theorem that
  \[
    \bigoplus_{k>0} C^*(X^k, \sOmega_{X^k}) \simeq \Sym(C^*(X, \sOmega_X))_+
  \]
  where $\Sym$ is taken inside $\Vect^{\FBplus}$, see also Lemma~\ref{lem:free_twisted_commutative_alg}.
\end{rmk}

\subsubsection{}
Theorem~\ref{thm:coh_conf_vs_coChev} allows us to understand the cohomology of generalized ordered configuration spaces in terms of $\Lie$-cohomology: $\alg{n}(X) \simeq \coChev^{\un}(\liealg{n}(X))$. In the case of interest, the $\coLie$ structures on $\liealg{n}(X) = C^*(X, \omega_X) \otimes \liealgOr{n}$ are somewhat inexplicit since they are induced by the co-commutative coalgebra structure on $C^*(X, \omega_X)$ and the original $\coLie$ structures on $\liealgOr{n}$, see also Remark~\ref{rmk:coLie_structure_on_pi_*pi^!a}. Note that since these are dg-(co)algebras, these operations contain a large amount of homotopical information that is not easily accessible. However, already with the simplest cohomological estimate arguments (such as those performed in this paper), we are able to obtain non-trivial new results and recover many known results.

Let $\mathfrak{a} \in \coLie(\Vect^{\FBplus})$ such that it can be decomposed into a direct sum $\mathfrak{a} \simeq \mathfrak{o}\oplus \mathfrak{g}$ in $\coLie(\Vect^{\FBplus})$ (i.e. no non-trivial operations between the factors), then we know that
\[
  \coChev(\mathfrak{a}) \simeq \coChev(\mathfrak{o}) \sqcup \coChev(\mathfrak{g}) \teq\label{eq:coChev_coprod}
\]
where the co-product is taken inside $\ComAlg(\Vect^{\FBplus})$. Equivalently, we have
\[
  \coChev^{\un}(\mathfrak{a}) \simeq \coChev^{\un}(\mathfrak{o}) \otimes \coChev^{\un}(\mathfrak{g}),  \teq\label{eq:coChev_coprod_unital}
\]
where the tensor, taken in $\Vect^{\FB}$, computes the co-product in $\ComAlg^{\un, \aug}(\Vect^{\FBplus})$. In particular, we see that $\coChev^{\un}(\mathfrak{a})$ is a free module over the twisted commutative algebra $\coChev^{\un}(\mathfrak{o})$.\footnote{Thus, we should think of $\mathfrak{o}$ as the operators and $\mathfrak{g}$ as the generators of the free module.} When $\mathfrak{o}$ happens to be trivial, then $\coChev^{\un}(\mathfrak{o}) \simeq \Sym (\mathfrak{o}[-1])$, and $\coChev^{\un} \mathfrak{a}$ is a free module over an explicit twisted symmetric algebra.\footnote{\label{ftn:Sym_contains_both_sym_and_anti-sym} Note that the symmetric monoidal structure on chain complexes involves signs. Thus, for $V \in \Vect^{\FBplus}$, what we call symmetric algebra, $(\Sym V)_+$, can contain both symmetric and and skew-symmetric algebras in the usual sense.} Moreover, we have explicit bounds on the generators $\coChev^{\un}(\mathfrak{g})$.

In many situations, the splitting of the $\coLie$-coalgebra $\liealg{n}(X) = C^*(X, \omega_X)\otimes \liealgOr{n}$ into a direct sum like the above can be obtained by simple co-homological estimates, vastly generalizing a theorem of Church--Ellenberg--Farb on the freeness of $\FI$-modules coming from ordered configuration spaces of non-compact manifolds~\cite{church_fi-modules_2015}*{\S6.4}. The next two subsections~\S\ref{subsec:splitting_Lie_i-acyclic} and~\S\ref{subsec:splitting_Lie_general} will be devoted to studying situations where we have such a splitting on $\liealg{n}(X)$.

\begin{rmk}
  The unital version~\eqref{eq:coChev_coprod_unital} is more convenient and familiar than the non-unital version~\eqref{eq:coChev_coprod}. On the other hand, non-unital commutative algebras live in the pro-nilpotent setting, which ensures that Koszul duality is an equivalence. This necessitates the use of both unital and non-unital commutative algebras in the paper.
\end{rmk}

\subsection{Splitting $\Lie$-algebras: the case of $i$-acyclic spaces}
\label{subsec:splitting_Lie_i-acyclic}
In this subsection, we will concern ourselves with the simplest case where $C^*_c(X, \Lambda)$ (resp. $C^*(X, \omega_X)$) has a trivial algebra (resp. co-algebra) structure in the sense of footnote~\ref{ftn:trivial_algebra_meaning}, or equivalently, when $C^*_c(X, \Lambda) \simeq \Ho^*_c(X, \Lambda)$ (resp. $C^*(X, \omega_X) \simeq \Ho^*(X, \omega_X)$) as objects in $\ComAlg(\Vect)$ (resp. $\ComCoAlg(\Vect)$) and the cup-products on $\Ho^*_c(X, \Lambda)$ (co-products on $\Ho^*(X, \omega_X)$) vanish. By~\cite{petersen_cohomology_2020}*{Thm. 1.20}, a large class of examples for such spaces is given by $i$-acyclic spaces, i.e. those with vanishing maps $\Ho^i_c(X, \Lambda) \to \Ho^i(X, \Lambda), \forall i$. In this case, $C^*_c(X, \Lambda) \otimes \liealgOr{n}^\vee$ (resp. $C^*(X, \omega_X) \otimes \liealgOr{n}$) has a trivial $\Lie$-algebra (resp. $\coLie$-coalgebra) structure, see Remarks~\ref{rmk:coLie_structure_on_pi_*pi^!a} and~\ref{rmk:triv_coLie_structure_C^*(X, omega)_trivial}.

\begin{prop}
  \label{prop:spectral_seq_collapse_triv_alg}
  Let $n\geq 2$ and $X$ is so that the algebra structure on $C^*_c(X, \Lambda)$ is trivial (e.g. when $X$ is $i$-acyclic). Then, we have an isomorphism of twisted commutative algebras
  \[
    \Ho^*(\alg{n}(X)) = \bigoplus_{k \geq 0} \Ho^*(P^k_n(X), \sOmega_{P^k_n(X)}) \simeq \Sym(\Ho^*(X, \omega_X) \otimes \Ho^*(\liealgOr{n})[-1]) = \Sym(\Ho^*(\liealg{n}(X))[-1]),
  \]
  where $\Sym$ is taken inside $\Vect^{\FB}$. In particular, the cohomology of generalized configuration spaces of $X$ depends only on the graded vector space $\Ho^*(X, \omega_X)$.
\end{prop}
\begin{proof}
  This is direct from Theorem~\ref{thm:coh_conf_vs_coChev}, the fact that $\liealg{n}(X) = C^*(X, \omega) \otimes \liealgOr{n}$ has trivial $\coLie$-coalgebra structure as discussed above, and that Koszul duality exchanges free and trivial objects, Proposition~\ref{prop:koszul_duality_pro-nilp}.
\end{proof}

\begin{rmk}
  \label{rmk:spectral_sequence_coChev_gConf}
  Combining Proposition~\ref{prop:computing_coChev_as_limit} and Theorem~\ref{thm:coh_conf_vs_coChev}, we see that there is a spectral sequence of twisted commutative algebras converging to $\bigoplus_{k \geq 0} \Ho^*(P^k_n(X), \sOmega_{P^k_n(X)})$ whose first page is given by $\Sym(\Ho^*(X, \omega_X) \otimes \Ho^*(\liealgOr{n}))$. From this point of view, the trivial algebra structure on $C^*_c(X, \Lambda)$ ensures that this spectral sequence collapse and we have the equivalence given in Proposition~\ref{prop:spectral_seq_collapse_triv_alg}.

  In the case where $n=2$, $\liealgOr{2}$ is the free $\coLie$-coalgebra generated by a one dimensional vector space. Now using the fact that $\Lie(n)$ (the $n$-th component of the $\Lie$-operad) is known to be the same as the top homology $\Ho(\Pi_n)$ of the lattice of partition $\Pi_n$, see e.g.~\cite{wachs_cohomology_1998}, we recover a spectral sequence similar to the one of~\cite{bibby_generating_2019}*{Thm. A}.
\end{rmk}

\begin{cor}
  \label{cor:a_2_vs_coh_P_2}
  $\liealgOr{2}$ has the following explicit description in terms of the cohomology of the configuration spaces of $\mathbb{A}^d$:
  \[
    (\liealgOr{2})_I \simeq \Ho^{(2d-1)(|I| - 1)}(P^I_2(\mathbb{A}^d))[-(2d-1)|I|](-d), \quad \forall I \in \FBplus,
  \]
  In particular,
  \[
    \bigoplus_{k\geq 0} \Ho^*(P^k_2(\mathbb{A}^d)) \simeq \bigotimes_{k\geq 1} \Sym (\Ho^{(2d-1)(k - 1)}(P^k_2(\mathbb{A}^d))[-(2d-1)(|I| - 1)]_k), \teq\label{eq:the_case_of_A^d}
  \]
  where $\otimes$ and $\Sym$ on the RHS are taken in $\Vect^{\FB}$ and the subscript $k$ inside $\Sym$ denotes the graded degree (see the end of~\S\ref{subsubsec:FB_objs}).
\end{cor}
\begin{proof}
  From Lemma~\ref{lem:a_2_is_free}, we know that for any $I$, $(\liealgOr{2})_I$ is concentrated at degree $(2d-1)|I|$. Applying Proposition~\ref{prop:spectral_seq_collapse_triv_alg} to the case where $X = \mathbb{A}^d$ and $n=2$, we have
  \[
    \bigoplus_{k\geq 0} \Ho^*(P_2^k(\mathbb{A}^d)) \simeq \bigoplus_{k\geq 0} \Ho^*(P_2^k(\mathbb{A}^d), \sOmega_{P_2^k(\mathbb{A}^d)}) \simeq \Sym(\liealgOr{2}[2d-1](d)). \teq\label{eq:coh_P^k_2(A^d)_via_Koszul}
  \]
  By inspection, we see that the top cohomological degree of $\Sym(\liealgOr{2}[2d-1])_I$ is $(2d-1)|I| - (2d-1) = (2d-1)(|I| - 1)$ and is given precisely by $(\liealgOr{2}[2d-1](d))_I$, i.e.
  \[
    (\liealgOr{2}[2d-1](d))_I \simeq \Ho^{(2d-1)(|I| - 1)}(P_2^I(\mathbb{A}^d))[-(2d-1)(|I| - 1)] \teq\label{eq:a_2[2d-1](d)}
  \]
  and hence
  \[
    (\liealgOr{2})_I \simeq \Ho^{(2d-1)(|I| - 1)}(P_2^I(\mathbb{A}^d))[-(2d-1)|I|](-d).
  \]

  The second statement is direct from~\eqref{eq:coh_P^k_2(A^d)_via_Koszul} and~\eqref{eq:a_2[2d-1](d)}.
\end{proof}

%\begin{cor}
%Let $n$ and $X$ be as in Proposition~\ref{prop:spectral_seq_collapse_triv_alg}. Assume further that $\Ho^*(X, \omega_X)$ is concentrated in only even degrees. Then we have a (non-canonical) isomorphism of bi-graded vector spaces\footnote{Here, we ignore any Frobenius action if presented.}
%\[
%  \bigoplus_{k \geq 0} \Ho^*(P^k_n(X), \sOmega_{P^k_n(X)}) \simeq \bigotimes_{i=0}^{-d} \left(\bigoplus_{k\geq 0} \Ho^*(P^k_n(X), \sOmega_{P^k_n(X)})[-2k(i+d)]\right)^{\otimes \dim \Ho^{2i}(X, \omega_{X})}. \teq\label{eq:recover_gConf_X_from_Ad}
%\]
%Namely, the cohomology of the generalized ordered configuration spaces of $X$ can be recovered from that of $\mathbb{A}^d$.
%\end{cor}
%\begin{proof}
%We have
%\begin{align*}
%  \bigoplus_{k \geq 0} \Ho^*(P^k_n(X), \sOmega_{P^k_n(X)})
%  &\simeq \Sym\left(\bigoplus_{i=0}^{-d} \Ho^{2i}(X, \omega_X)[-2i] \otimes \Ho^*(\liealgOr{n})\right)
%  \simeq \bigotimes_{i=0}^{-d} \Sym(\Ho^*(\liealgOr{n})[-2i])^{\otimes \dim \Ho^{2i}(X, \omega_X)} \\
%  &\simeq \bigotimes_{i=0}^{-d} \Sym(\Ho^*(\mathbb{A}^d, \omega_{\mathbb{A}^d}) \otimes \Ho^*(\liealgOr{n})[-2(i+d)])^{\otimes \dim \Ho^{2i}(X, \omega_X)}.
%\end{align*}
%Finally, the last term is isomorphic to the RHS of~\eqref{eq:recover_gConf_X_from_Ad} by applying Proposition~\ref{prop:spectral_seq_collapse_triv_alg} to the case where $X = \mathbb{A}^d$ to the $\Sym$ factor.
%\end{proof}

\subsection{Splitting $\Lie$-algebras: the $\TopTriv_m$ case}
\label{subsec:splitting_Lie_general}
For a general $X$, we do not have a nice presentation as in Proposition~\ref{prop:spectral_seq_collapse_triv_alg}. However, when $n=2$ and when $X$ satisfies some vanishing of the cup products (and higher analogs thereof, see Definition~\ref{defn:TopTriv_m_condition}), we still have a decomposition of $\liealg{2}(X) = C^*(X, \omega_X) \otimes \liealgOr{2}$ into a direct sum that could be used to prove higher representation stability for ordered configuration spaces of $X$. This is the content of the main result of this subsection, Proposition~\ref{prop:splitting_n=2}.

For the current purpose, it is convenient to use the language of $\Linfty$-algebras. The $\Linfty$-algebras involved are obtained by applying homotopy transfer theorem to the dual of $\liealg{2}(X)$. We will now describe it in more details.

\subsubsection{}
For each $I\in \FBplus$, $(\liealg{n})_I$ is a bounded chain complex with finite dimensional cohomology. Thus, we can take the linear dual $(\liealg{n}(X))^\vee$ (i.e. component-wise linear dual) to obtain an object in $\Lie(\Vect^{\FBplus})$ whose dual is, again, the original $\coLie$-coalgebra. Statements about $\liealg{n}(X)$ can be easily obtained from those about $(\liealg{n}(X))^\vee$ and vice versa. Since it is more convenient to speak about $\Lie$ algebras, we will perform most of our arguments on this side.

By Verdier duality, we have
\[
  (\liealg{n}(X))^\vee = (C^*(X, \omega_X) \otimes \liealgOr{n})^\vee \simeq C^*_c(X, \Lambda) \otimes \liealgOr{n}^\vee.
\]
Moreover, the $\Lie$ algebra structure on $C^*_c(X, \Lambda) \otimes \liealgOr{n}^\vee$ can also be obtained using the commutative (i.e. $E_\infty$) algebra structure on $C^*_c(X, \Lambda)$ and the $\Lie$-algebra structure on $\liealgOr{n}^\vee$, see also~\cite{gaitsgory_study_2017}*{Vol. II, \S6.1.2} and~\cite{ho_free_2017}*{Example 4.2.8}.

Since $\Vect^{\FBplus}$ has homological dimension $0$, we can apply the homotopy transfer theorem~\cite{loday_algebraic_2012}*{\S10.3}, which endows $\Ho^*(\liealg{n}(X)^\vee) = \Ho^*_c(X, \Lambda)\otimes \Ho^*(\liealgOr{n}^\vee)$ with a structure of an $\Linfty$-algebra in $\Vect^{\FB_+}$. Namely, for each $k\geq 2$, we have the following $k$-ary operation which is of cohomological degree $2-k$ (\cite{loday_algebraic_2012}*{Prop. 10.1.7})
\[
  l_k: (\Ho^*_c(X, \Lambda)\otimes \Ho^*(\liealgOr{n}^\vee))^{\otimes k} \to \Ho^*_c(X, \Lambda)\otimes \Ho^*(\liealgOr{n}^\vee)[2-k],
\]
where, as before, the tensor $\otimes k$ is taken in $\Vect^{\FB_+}$.

\subsubsection{}
By homotopy transfer, $\Ho^*_c(X, \Lambda)$ is also equipped with the structure of a $\Cinfty$-algebra. Namely, for each $k\geq 2$, we have the following $k$-ary operation of cohomological degree $2-k$
\[
  m_k: \Ho^*_c(X, \Lambda)^{\otimes k} \to \Ho^*_c(X, \Lambda)[2-k].
\]
Moreover, the $\Linfty$-algebra structure on $\Ho^*(\liealg{n}(X)^\vee) = \Ho^*_c(X, \Lambda) \otimes \Ho^*(\liealgOr{n}^{\vee})$ is induced by that of $\Ho^*(\liealgOr{n}^{\vee})$ and the $\Cinfty$-algebra structure on $\Ho^*_c(X, \Lambda)$.  This can be used to analyze $\liealg{n}(X)^\vee$ or, equivalently, $\liealg{n}(X)$.

\begin{rmk}
  \label{rmk:mixing_Cinfty_Linfty_n=2}
  When $n=2$, the $\Lie$-algebra $\liealgOr{2}^{\vee}$ is formal, being a free algebra generated by a one-dimensional vector space. Thus, the $k$-ary operations on $\Ho^*(\liealg{2}(X)^\vee)$ is obtained by combining the $k$-ary operations of the $\Cinfty$-algebra $\Ho^*_c(X, \Lambda)$ and $2$-ary operations of $\Ho^*(\liealgOr{2}^{\vee})$.
\end{rmk}

In the simplest case where $\Ho^*_c(X, \Lambda)$ has a trivial $\Cinfty$-algebra structure, the $\Linfty$-algebra structure on $\Ho^*(\liealg{n}(X)^\vee)$ is also trivial, and hence, so is the $\coLie$-coalgebra structure on $\liealg{n}(X)$. This was treated in~\S\ref{subsec:splitting_Lie_i-acyclic}. In what follows, we will make use of the following weaker condition.

\begin{defn}[$\TopTriv_m$ condition] \label{defn:TopTriv_m_condition}
  Let $X$ be a scheme. We say that $X$ satisfies condition $\TopTriv_1$ if $\Ho^0_c(X, \Lambda) = 0$.

  Let $m \geq 2$ be an integer. We say that $X$ satisfies condition $\TopTriv_m$ if there is a transferred $\Cinfty$-algebra structure on $\Ho_c^{*}(X, \Lambda)$ such that for all $k\leq m$, all $k$-ary operations on $\Ho_c^{*}(X, \Lambda)$ which land in $\Ho_c^{2d}(X, \Lambda)$ are trivial.
\end{defn}

\begin{rmk}
  When $m=2$, this simply says that the cup products
  \[
    \Ho^{k}_c(X, \Lambda) \otimes \Ho^{2d-k}_c(X, \Lambda) \to \Ho^{2d}_c(X, \Lambda)
  \]
  vanish, or dually, the map
  \[
    \Ho^{-2d}(X, \omega_X) \to \Ho^{-2d}(X\times X, \omega_{X\times X})
  \]
  induced by the diagonal $X \to X\times X$ vanishes.
\end{rmk}

\begin{rmk}
  \label{rmk:T2_and_above_vanishing_bottom}
  It is easy to see that when $X$ satisfies $\TopTriv_2$, then $X$ also satisfies $\TopTriv_1$, i.e. $\Ho^0_c(X, \Lambda) = 0$. In other words, $X$ is forced to be non-proper. Indeed, suppose otherwise, then $X$ has (at least) a proper connected component. Taking the cup product of the non-zero element class in $\Ho^0_c(X, \Lambda)$ and a class in $\Ho^{2d}_c(X, \Lambda)$ from the same component will yield a non-zero result.

  Clearly, the conclusion still holds when $X$ satisfies $\TopTriv_m$ for any $m\geq 2$. Thus, $\TopTriv_{m} \Rightarrow \TopTriv_{m'}$ when $m'\leq m$.
\end{rmk}

\begin{rmk}
  When $X$ is such that $\Ho^*_c(X, \Lambda)$ is pure, the algebra $C^*_c(X, \Lambda)$ is formal, and hence, all higher operations of $\Ho^*_c(X, \Lambda)$ vanish. For such an $X$, being $\TopTriv_m$ is equivalent to being $\TopTriv_2$.

  We do not know of an example of a scheme $X$ satisfying $\TopTriv_m$ for some $m > 2$ where $C^*_c(X, \Lambda)$ is not formal. It would be interesting to find such an example.
\end{rmk}

\subsubsection{Lie algebras of operators}
We let
\[
  \liealgOp{n}(X) = \Ho^{-2d}(X, \omega_X)[2d] \otimes \liealgOr{n} \quad\text{and its dual} \quad \liealgOp{n}(X)^{\vee} = \Ho^{2d}_c(X, \Lambda)[-2d] \otimes \liealgOr{n}^{\vee}. \teq\label{eq:defn_liealgOp}
\]
For any $m \geq 1$, we let $\liealgOp{n, \leq m}(X)$ (resp. $\liealgOp{n, \leq m}(X)^\vee$) be obtained from $\liealgOp{n}(X)$ (resp. $\liealgOp{n}(X)^\vee$) by setting all components of graded degrees greater than $m$ to zero. We equip $\liealgOp{n}(X)$ and $\liealgOp{n, \leq m}(X)$ (resp. $\liealgOp{n}(X)^\vee$ and $\liealgOp{n, \leq m}(X)^\vee$) with the trivial $\coLie$-coalgebra (resp. $\Lie$-algebra) structures. Thus, we have a splitting inside $\coLie(\Vect^{\FBplus})$
\[
  \liealgOp{n, \leq m}(X) \simeq \bigoplus_{k=1}^m\liealgOp{n, k}(X) \qquad\text{and}\qquad \liealgOp{n}(X) \simeq \bigoplus_{k=1}^\infty \liealgOp{n, k}(X)
\]
where $\liealgOp{n, k}(X)$ is the graded degree $k$ part of $\liealg{n}(X)$.

\begin{rmk}
  \label{rmk:liealgOp_irred_case}
  When $X$ is irreducible, $\Ho^{-2d}(X, \omega_X)$ and $\Ho^{2d}_c(X, \Lambda)$ are one dimensional. Hence, $\liealgOp{n}(X) \simeq \liealgOr{n}[2d](d)$ and $\liealgOp{n}(X)^{\vee} \simeq \liealgOr{n}^\vee[-2d](-d)$ in this case.
\end{rmk}

We have good control in the case where $n=2$: since $\liealgOr{2}$ is concentrated in exactly one cohomological degree per graded degree, see Lemma~\ref{lem:a_2_is_free}, so are $\liealgOp{2}(X)$ and $\liealgOp{2}(X)^\vee$. More precisely, $(\liealgOp{2}(X))_I$ (resp. $(\liealgOp{2}(X)^{\vee})_I$) concentrates in cohomological degree $(2d-1)(|I| - 1)-1$ (resp. $-(2d-1)(|I| - 1)+1$). For a general $n$, the only precise information we have about $\liealgOp{n}(X)$ is from Lemma~\ref{lem:homological_est_a_n}.

\begin{lem}
  \label{lem:unit_for_liealg_n(X)}
  We have natural morphisms in $\coLie(\Vect^{\FBplus})$:\footnote{See~\eqref{eq:truncation_functors_FB} and the discussion around it for the definitions of $\tr_{\leq n}$ and $\iota_{\leq n}$.}
  \[
    \liealgOp{n, \leq 1}(X) \to \iota_{\leq 1}\tr_{\leq 1} \liealg{n}(X) \to \liealg{n}(X),
  \]
  or, by abuse of notation,
  \[
    \liealgOp{n, \leq 1}(X) \to \liealg{n}(X)_1 \to \liealg{n}(X),
  \]
  where the first two terms are equipped with the trivial $\coLie$-coalgebra structures.
\end{lem}
\begin{proof}
  Using the unit map of the adjunction in Remark~\ref{rmk:iota_<n_left_adjoint}, we obtain a morphism in $\coLie(\Vect^{\FBplus})$
  \[
    \iota_{\leq 1}(\tr_{\leq 1} \liealg{n}(X)) \to \liealg{n}(X).
  \]
  By degree consideration, the $\coLie$-coalgebra structure on $\iota_{\leq 1}(\tr_{\leq 1} \liealg{n}(X))$ is necessarily trivial. Thus, we have a morphism between trivial $\coLie$-coalgebras
  \[
    \tau^{\leq -1} \iota_{\leq 1}(\tr_{\leq 1} \liealg{n}(X)) \to \iota_{\leq 1}(\tr_{\leq 1} \liealg{n}(X)),
  \]
  where $\tau^{\leq -1}$ is the cohomological truncation. Observe that $-1$ is the bottom cohomological degree of the LHS, we get
  \[
    \liealgOp{n, \leq 1}(X) \simeq \Ho^{-1}(\iota_{\leq 1}(\tr_{\leq 1} \liealg{n}(X)))[1] \simeq \tau^{\leq -1} \iota_{\leq 1}(\tr_{\leq 1} \liealg{n}(X)).
  \]
  Composing these morphisms together, we obtain the desired map $\liealgOp{n, \leq 1}(X) \to \liealg{n}(X)$.
\end{proof}

When $n=2$, we can obtain more refined results.

\begin{prop}
  \label{prop:splitting_n=2}
  Suppose $X$ satisfies $\TopTriv_m$ for some $m\geq 1$. Then we have a direct sum decomposition in $\coLie(\Vect^{\FBplus})$
  \[
    \liealg{2}(X) \simeq \liealgOp{2, \leq m}(X) \oplus \liealgGen{2, \leq m}(X) \simeq \bigoplus_{k=1}^m \liealgOp{2, k}(X) \oplus \liealgGen{2, \leq m}(X)
  \]
  for some $\liealgGen{2, \leq m}(X) \in \coLie(\Vect^{\FBplus})$ whose cohomological support is obtained from that of $\liealg{2}(X)$ minus that of $\liealgOp{2, \leq m}(X)$. In particular, for any $I\in \FBplus$, the cohomological support of $\liealgGen{2, \leq m}(X)_I$ is concentrated on cohomological degrees
  \[
    \begin{cases}
      [(2d-1)|I| - 2d + 1, (2d-1)|I| -1 ], &\text{when } |I| \leq m, \\
      [(2d-1)|I|-2d, (2d-1)|I| - 1], &\text{when } |I| > m
    \end{cases}
  \]
\end{prop}
\begin{proof}
  We will prove everything on the dual side and work with $\Linfty$-algebras. More specifically, for any $m$, the underlying object $\Ho^*(\liealg{2}(X)^{\vee})$ in $\Vect^{\FBplus}$ naturally decomposes as a direct sum
  \[
    \Ho^*(\liealg{2}(X)^{\vee}) \simeq \liealgOp{2, \leq m}(X)^\vee \oplus \liealgGen{2, \leq m}(X)^\vee
  \]
  for some $\liealgGen{2, \leq m}(X)^\vee \in \Vect^{\FBplus}$.\footnote{More accurately, we view these objects as doubly graded vector spaces: one coming from the cohomological grading and the other from the $\FBplus$ grading. The splitting is most naturally seen from this point of view.} Our task now is to see that this is in fact a splitting of $\Linfty$-algebras. To do that, it suffices to check two facts:
  \begin{myenum}{(\roman*)}
  \item the result of any $\Linfty$-operation involving elements from only one summand stays in that summand, and
  \item the result of any $\Linfty$-operation involving elements from both summands vanish.
  \end{myenum}

  When $X$ satisfies $\TopTriv_m$, (i) is verified for the $\liealgGen{2, \leq m}(X)^{\vee}$ factor by design, using the fact that $k$-ary operations on $\Ho^*(\liealg{2}(X)^\vee)$ is obtained by combining $k$-ary operations on $\Ho^*_c(X, \Lambda)$ and $2$-ary operations on $\Ho^*(\liealgOr{2}^\vee)$, see Remark~\ref{rmk:mixing_Cinfty_Linfty_n=2}. In general, (i) is also verified for the $\liealgOp{2, \leq m}(X)^{\vee}$ factor. Indeed, by degree considerations, it is easy to see that all $\Linfty$-operations of $\Ho^*(\liealg{2}(X)^{\vee})$ vanish when restricted to $\liealgOp{2}(X)^{\vee}$. Thus, it remains to check (ii).

  For this, we only need the fact that $X$ satisfies $\TopTriv_1$, which is implied by $\TopTriv_m$ for any $m\geq 1$, by Remark~\ref{rmk:T2_and_above_vanishing_bottom}. In this case, we know that $\Ho^*(\liealg{2}(X)^{\vee}_I)$ lives in cohomological degrees $[-(2d-1)|I| + 1, -(2d-1)|I| + 2d]$. We will show that the result of any $k$-ary operation, with $k\geq 2$, involving at least one element in $\liealgOp{2}(X)^{\vee}$ will lie outside of this range, forcing it to vanish, concluding (ii).

  Indeed, the smallest cohomological degree obtained from such a $k$-ary operation is given by elements of ``coordinates'' $(-(2d-1)|I_i|+1, |I_i|)$ for $i=1, \dots, k-1$ and $(-(2d-1)|I_k|+2d, |I_k|)$, where the first and second coordinates refer to cohomological degree and graded degree respectively. Now, the result of a $k$-ary operation involving these will live in
  \[
    (-(2d-1)\sum_{i=1}^k |I_i| + k-1 + 2d + (2-k), \sum_{i=1}^k |I_i|) = (-(2d-1)\sum_{i=1}^k |I_i| + 2d + 1, \sum_{i=1}^k |I_i|)
  \]
  which has cohomological degree $> -(2d-1)\sum_{i=1}^k |I_i| + 2d$, the highest degree supported by $\Ho^*(\liealg{2}(X)^\vee_{\sqcup_{i=1}^k I_i})$.

  The last statement about cohomological amplitude of $\liealgGen{2, \leq m}(X)$ comes directly from that of $\liealg{2}(X)$ and the construction.
\end{proof}

\begin{rmk}
  We could have run a similar argument as in Proposition~\ref{prop:splitting_n=2} for $\liealg{n}(X)$ where $n>2$ if we had more information about the cohomological amplitude of $\liealgOr{n}$, the Koszul dual of the truncated twisted polynomial algebra $(\Lambda[x]/x^n)_+$. Even though this is a purely algebraic question, its answer would hold the key to understanding higher representation stability of generalized configuration spaces.

	For usual commutative algebras (as opposed to TCAs), the Koszul duals of the truncated polynomial algebras have indeed been used in the author's previous paper~\cite{ho_homological_2021} to completely compute (the analog of) $\liealgOr{n}$. The twisted case, however, is more subtle. The reason is that the twisted polynomial algebra generated by a single generator in degree $n$ is much bigger than just $\Lambda[x^n]$. Because of that, we do not have a nice pushout square to work with.
\end{rmk}

\subsection{Higher representation stability: a sketch}
\label{subsec:sketch_higher_rep_stab}
Having defined all the necessary objects, we will now sketch the ideas involved in establishing (higher) representation stability for ordered configuration spaces. The execution of these ideas will be carried out in the next two subsections.

As before, we will assume throughout that $X$ is an equidimensional scheme of dimension $d$. This is not a serious restriction. In fact, with more complicated notation to keep track of the dimension, most results in this paper are also valid when the equidimension condition is relaxed. We choose not to do this to avoid needlessly complicating the formulations while not introducing any new idea.

\subsubsection{}
Theorem~\ref{thm:coh_conf_vs_coChev} allows us to understand the cohomology of generalized ordered configuration spaces in terms of $\Lie$-cohomology: $\alg{n}(X) \simeq \coChev^{\un}(\liealg{n}(X))$. From this point of view, each morphism $\liealgOp{} \to \liealg{n}(X)$ in $\coLie(\Vect^{\FBplus})$ induces, via Koszul duality, a morphism $\coChev^{\un} \liealgOp{} \to \coChev^{\un} \liealg{n}(X) \simeq \alg{n}(X)$ in $\ComAlg^{\un, \aug}(\Vect^{\FBplus})$. The latter can thus be viewed as a module over the former. The induced action could then be used to formulate and prove representation stability.

For example, when $X$ is irreducible, we take $\liealgOp{} = \Ho^{-1}(\liealg{n}(X)_1)[1] \simeq \Ho^{-2d}(X, \omega_X)[2d] \otimes (\liealgOr{n})_1$ (the bottom cohomology of the graded degree $1$ part) which is a $1$-dimensional trivial $\coLie$-coalgebra. In this case, $\coChev^{\un} \liealgOp{} \simeq \Sym \Lambda_1$, and note that a module over this twisted commutative algebra is the same as an $\FI$-module, see~\cite{sam_introduction_2012}. We thus obtain an $\FI$-module structure on $\alg{n}(X)$ which is classically used to formulate representation stability. In general, $\liealgOp{}$ is extracted from $\liealg{n}(X)$ in a similar way, namely, by picking out certain cohomological and graded degree pieces of $\liealg{n}(X)$.

\subsubsection{}
This procedure can be iterated. Namely, we can take the relative tensor
\[
  \alg{n}(X) \otimes_{\coChev^{\un} \liealgOp{}} \Lambda_0 \simeq \coChev^{\un} (\liealg{n}(X)) \otimes_{\coChev^{\un} \liealgOp{}} \Lambda_0 \simeq \coChev^{\un} (\liealg{n}(X) \sqcup_{\liealgOp{}} 0) \teq \label{eq:factoring_out}
\]
where the last equivalence is due to the fact that Koszul duality, being an equivalence, preserves pushouts. In the language of~\cite{galatius_cellular_2018}, the process of taking relative tensor is known as taking derived indecomposables.

Now, note that since the forgetful functor
\[
  \coLie(\Vect^{\FBplus}) \to \Vect^{\FBplus},
\]
is a left adjoint, it commutes with colimits, and hence, the underlying chain complex of $\liealg{n}(X) \sqcup_{\liealgOp{}} 0$ is just the mapping cone of $\liealgOp{} \to \liealg{n}(X)$. We thus have a very good control over the cohomological amplitude of $\liealg{n}(X) \sqcup_{\liealgOp{}} 0$ and hence, also of its cohomological Chevalley complex, via Proposition~\ref{prop:computing_coChev_as_limit}. Now, we can apply the procedure above to $\liealg{n}(X) \sqcup_{\liealgOp{}} 0$ itself. Namely, we can produce more actions by finding a map $\liealgOp{}' \to \liealg{n}(X) \sqcup_{\liealgOp{}} 0$ in $\coLie(\Vect^{\FBplus})$.

Back to the example above where $\liealgOp{} = \Ho^{-1}(\liealg{n}(X)_1)[1]$. The first term in~\eqref{eq:factoring_out} is precisely the $\FI$-hyperhomology appearing in~\cite{gan_linear_2019, miller_$mathrmfi$-hyperhomology_2019} (hyper because all our tensors are derived by convention). In the non-twisted settings, i.e. homological stability rather than representation stability, similar quotients also appear in~\cite{galatius_$e_2$-cells_2019} under the guise of relative homology and is used to capture failure of homological stability and to formulate higher homological stability. Similarly, in~\cite{miller_higher_2016}, $\FI$-homology is used to capture the failure of representation stability ($\FI$-hyperhomology does not appear since the objects involved are free $\FI$-modules). In both case, secondary homological/representation stability is formulated in terms of quotients in this form. The iterated nature of the procedure above can thus be viewed as a natural way to capture and formulate higher stability phenomena.

\subsubsection{}
We will now describe how we can iteratively construct maps between $\coLie$-coalgebras of the form mentioned above. To fix ideas, we will restrict to the case of ordered configuration spaces (rather than the generalized variants) to the former case in the remainder of this subsection.

The iterative procedure we will describe is best seen graphically. Let us start with Figure~\ref{fig:supp_a_2(X)[-1]_and_A_2(X)}, where the supports of $\liealg{2}(X)[-1]$ and $\alg{2}(X)$ are illustrated. Here, we choose to draw $\liealg{2}(X)[-1]$ rather than $\liealg{2}(X)$ itself since the passage from $\coLie$-coalgebras to commutative algebras via Koszul duality involves a shift to the right by $1$, see Proposition~\ref{prop:computing_coChev_as_limit}.

\begin{figure}[ht]
  \begin{tikzpicture}[scale=0.8, every node/.style={scale=0.8}]
    {
      \def\dimension{2}
      \draw[thick, opacity=0.25] pic {SuppLines={4}};
      \draw[thick, opacity=0.55] pic {LineThruOrgn={1}{4}{14.4}{node[above, pos=0.98, opacity=1]{$L$}}};
      \draw pic {Coord={14}{5}};
      \draw pic {Grid={1}{4}};

      \foreach \x in {0, 1}
      \draw pic {xCoordtick={\x}{node[anchor=north] {$\x$}}};

      \draw pic {xCoordtick={3}{node[anchor=north] {$2d-1$}}};
      \draw pic {xCoordtick={4}{node[anchor=north] {$2d$}}};

      \draw pic {xCoordtick={6}{node[anchor=north] {$2(2d-1)$}}};

      \foreach \y in {0,1,2,3,4}
      \draw pic {yCoordtick={\y}{node[anchor=east] {$\y$}}};
    }
  \end{tikzpicture}
  \caption{The supports of $\liealg{2}(X)[-1]$ and $\alg{2}(X)$} \label{fig:supp_a_2(X)[-1]_and_A_2(X)}
\end{figure}

Let us now explain the notation used in Figure~\ref{fig:supp_a_2(X)[-1]_and_A_2(X)}:
\begin{myenum}{(\roman*)}
\item The horizontal (resp. vertical) axis labeled by $c$ (resp. $k$) represents the cohomological (resp. graded, or $\FB$) degrees.
\item The lines $L_B$, $L_1$, $L_T$, and $L$ are given by the following equations
  \begin{align*}
    L_B&: c = (2d-1)(k - 1), & L_1&: c = (2d-1)k, \\
    L_T&: c = (2d-1)k + 1, & L&: c=2dk.
  \end{align*}
  The support of $\liealg{2}(X)[-1]$, represented as $\times$ in the figure, lies between the lines $L_B$ and $L_T$ ($B$ and $T$ stand for bottom and top respectively). When $X$ satisfies $\TopTriv_1$, then $\liealg{2}(X)[-1]_k$ is supported in cohomological degrees $[(2d-1)(k-1), (2d-1)k]$, and hence, the support of $\liealg{2}(X)[-1]$ lies between $L_B$ and $L_1$.

\item The support of $\Sym (\liealg{2}(X)[-1])$ lies between the $k$-axis and the line $L$, the smallest cone through the origin containing the support of $\liealg{2}(X)[-1]$.

  By Proposition~\ref{prop:computing_coChev_as_limit}, we know that $\coChev \liealg{2}(X)$ could be computed as a limit with successive fibers together form $\Sym \liealg{2}(X)_+$. Moreover, for each $k$, by graded degree considerations, we see that the limit computing $(\coChev \liealg{2}(X))_k$ stabilizes. Thus, the cohomological support of $\coChev \liealg{2}(X)$, and hence, of $\alg{2}(X) = \coChev^{\un} \liealg{2}(X)$, also lies between the $k$-axis and $L$.

  When $X$ satisfies $\TopTriv_1$, by a similar argument, we see that the support of $\alg{2}(X)$ lies between the $k$-axis and $L_1$.
\end{myenum}

For ordered configuration spaces, the iterative process that we will consider in this paper will come in two flavors, illustrated by the orange and blue strips in Figure~\ref{fig:iterative_rep_stability}, which will be considered in the next two subsections~\S\ref{subsec:higher_rep_stab_general} and~\S\ref{subsec:higher_rep_stab_TopTrivm} respectively. The former is available for all spaces $X$ whereas the latter is only available when $X$ satisfies $\TopTriv_m$ for some $m$.\footnote{For the blue strip, Figure~\ref{fig:iterative_rep_stability} illustrates the situation where $X$ satisfies $\TopTriv_3$.}

\begin{figure}[ht]
  \begin{tikzpicture}[scale=0.8, every node/.style={scale=0.8}]
    {
      \def\dimension{2}

      % setup coordinates
      \draw pic {Coord={14}{5}};
      \draw[thick, opacity=0.25] pic {SuppLines={4}};

      % highlight the operators
      \draw[orange, opacity=0.40] pic {FillRow={1}};
      \draw[orange, opacity=0.30] pic {FillRow={2}};
      \draw[orange, opacity=0.25] pic {FillRow={3}};
      \draw[orange, opacity=0.20] pic {FillRow={4}};
      \draw[blue, opacity=0.30] pic {FillFirstLine={3}};

      % support grid
      \draw pic {Grid={1}{4}};

      % make coordinate ticks
      \foreach \x in {0, 1}
      \draw pic {xCoordtick={\x}{node[anchor=north] {$\x$}}};

      \draw pic {xCoordtick={3}{node[anchor=north] {$2d-1$}}};
      \draw pic {xCoordtick={4}{node[anchor=north] {$2d$}}};

      \draw pic {xCoordtick={6}{node[anchor=north] {$2(2d-1)$}}};

      \foreach \y in {0,1,2,3,4}
      \draw pic {yCoordtick={\y}{node[anchor=east] {$\y$}}};
    }
  \end{tikzpicture}
  \caption{$\liealg{2}(X)[-1] = C^*(X, \omega_X)\otimes \liealgOr{2}[-1]$} \label{fig:iterative_rep_stability}
\end{figure}

\subsection{Higher representation stability: the general case}
\label{subsec:higher_rep_stab_general}
This subsection is devoted to the study of higher representation stability phenomena coming from the yellow strips illustrated in Figure~\ref{fig:iterative_rep_stability}. This is an iterative version of Lemma~\ref{lem:unit_for_liealg_n(X)}.

\subsubsection{}
At the first level $k=1$, we have a morphism in $\coLie(\Vect^{\FBplus})$,
\[
  \begin{tikzcd}[column sep=tiny, row sep=small]
    \Ho^{-1}(\liealg{2}(X)_1)[1] \ar[phantom]{d}[description]{\rotatebox[origin=c]{90}{$\simeq$}} \ar{r} & \liealg{2}(X)_1 \ar[phantom]{d}[description]{\rotatebox[origin=c]{90}{$\simeq$}} \ar{r} & \liealg{2}(X) \\
    \Ho^0(X, \sOmega_X)[1]_1 \ar{r} & C^*(X, \sOmega_X)[1]_1
  \end{tikzcd} \teq\label{eq:higher_degree_stabilization}
\]
with the first two terms on the first row being, necessarily, equipped with the trivial $\coLie$-coalgebra structure. In Figure~\ref{fig:iterative_rep_stability},
\[
  (\liealg{2}(X)[-1])_1 \simeq C^*(X, \sOmega_X)_1 \quad\text{and}\quad
  \Ho^{-1}((\liealg{2}(X))_1) \simeq \Ho^0((\liealg{2}(X)[-1])_1) \simeq \Ho^0(X, \sOmega_X)_1.
\]
are illustrated by the bottom orange strip and, respectively, by the point with coordinate $(c, k) = (0, 1)$.

This induces a morphism of twisted commutative algebras
\[
  \Sym(\Ho^{-1}(\liealg{2}(X))_1) \to \alg{2}(X)
\]
and hence, turns the latter into a module over the former. When $X$ is irreducible, $\Ho^{-1}(\liealg{2}(X)_1) \simeq \Lambda_1$, and hence, $\alg{2}(X)$ has a module structure over $\Sym \Lambda_1$, which is equivalent to saying that $\alg{2}(X)$ is naturally an $\FI$-module, see~\cite{sam_introduction_2012}. Representation stability is usually studied via establishing finite generation of $\Ho^c(\alg{2}(X))$ as a module over $\Sym(\Lambda_1)$ along with the graded-degree bound on the generator. Moreover, the failure of representation stability of $\Ho^c(\alg{2}(X))$ is measured by its $\FI$-homology.

In fact, from our point of view, since $\alg{2}(X)$ is naturally a chain complex, $\FI$-hyperhomology of $\alg{2}(X)$ rather than $\FI$-homology of its cohomology groups considered separately is the more natural object to study. Even in the case of homological stability, this point of view is taken in~\cite{galatius_$e_2$-cells_2019}, where the failure of homological stability is taken to be the cone of the stabilizing map, the analog of $\FI$-hyperhomology in the non-twisted setting. Note that in~\cite{miller_higher_2016}, $\FI$-hyperhomology does not make an appearance since all the modules involved are free.

In general, the link between $\FI$-hyperhomology and their higher analogs to (higher) representation stability is somewhat illusive. We will return to this question in some special cases in Corollary~\ref{cor:Tm_image_spanned} and~\ref{cor:T1_image_spanned}; see also Remark~\ref{rmk:hyper_vs_actual_stability}. In the current subsection, we focus on the study of $\FI$-hyperhomology and their higher analogs.

Back to our problem, taking derived indecomposables with respect to the action of $\Sym(\Ho^{-1}(\liealg{2}(X))_1)$ on $\alg{2}(X)$, i.e. taking its $\FI$-hyperhomology, gives
\[
  \Lambda_0 \otimes_{\Sym(\Ho^{-1}(\liealg{2}(X))_1)} \alg{2}(X) \simeq \coChev^{\un}(\liealg{2}'(X)) =: \alg{2}'(X) \teq\label{teq:hyper_FI}
\]
where $\liealg{2}'(X) \simeq \liealg{2}(X) \sqcup_{\Ho^{-1}(\liealg{2}(X)_1)[1]} 0$ is obtained from $\liealg{2}(X)$ by truncating away the lowest term in graded degree $0$, see also~\eqref{eq:factoring_out}. In terms of Figure~\ref{fig:iterative_rep_stability} above, the support of $\liealg{2}'(X)[-1]$ is given by removing the point at $(c, k) = (0, 1)$ from the support of $\liealg{2}(X)[-1]$.

As above, the support of $\coChev^{\un}(\liealg{2}'(X))$ is contained in the smallest cone through the origin containing the support of $\liealg{2}'(X)[-1]$. We thus obtain the following result.

\begin{prop} \label{prop:primary_hyper_FI_homology}
  Let $X$ be an irreducible scheme of dimension $d$ and $b = \min\left(\frac{2d-1}{2}, 1\right)$. Then, the support of the $\FI$-hyperhomology of the cohomology of ordered configuration spaces of $X$ lies below the line given by $c=bk$ (where $c$ and $k$ as as labeled in Figure~\ref{fig:iterative_rep_stability}). In particular, if we let $\HHo^c_{\FI}(-)$ denote the functor of taking the $c$-th $\FI$-hyperhomology, then $\HHo_{\FI}^c(\alg{2}(X))_k \simeq 0$ when
  \begin{myenum}{(\alph*)}
  \item $k>2c$ and $d = 1$, or
  \item $k>c$ and $d > 1$.
  \end{myenum}
  In general, for any $c$ and $k$, $\HHo^c_{\FI}(\alg{2}(X))_k$ is finite dimensional.

  More generally, for any $X$ equidimensional of dimension $d$, we have the same conclusion when replacing $\HHo^c_{\FI}(\alg{2}(X))_k$ by $\Ho^c(\Lambda_0 \otimes_{\Sym(\Ho^{-1}(\liealg{2}(X))_1)} \alg{2}(X))_k$.
\end{prop}

\begin{rmk}
  \label{rmk:hyper_vs_actual_stability}
  For manifolds, a similar result is obtained in~\cite{miller_$mathrmfi$-hyperhomology_2019} and serves as the starting point for their new bounds on representation stability of ordered configuration spaces. We expect that, at least for rational coefficients, similar considerations might allow us to generalize the result there beyond the case of manifolds.
\end{rmk}

\subsubsection{}
\label{subsubsec:iterative_procedure}
This process can be studied iteratively. Namely, after taking derived indecomposables with respect to the action induced by the bottom left corner as above, we proceed to the next one on the first row $k=1$, i.e. the one at $(c, k) = (1, 1)$ in Figure~\ref{fig:iterative_rep_stability}. More explicitly, we obtain from~\eqref{eq:higher_degree_stabilization} a morphism in $\coLie(\Vect^{\FB_+})$
\[
  \Ho^1(X, \sOmega_X) \simeq \Ho^{0}(\liealg{2}(X)_1) \to \liealg{2}'(X),
\]
where the LHS is equipped with the structure of a trivial $\coLie$-coalgebra. This induces a $\Sym (\Ho^1(X, \sOmega_X)[-1]_1)$-module structure on $\Lambda_0 \otimes_{\Sym(\Lambda_1)} \alg{2}(X) \simeq \coChev^{\un}(\liealg{2}'(X))$.

We can then factor out the action induced by this one as well, that is, we can take
\[
  \Lambda_0 \otimes_{\Sym (\Ho^1(X, \sOmega_X)[-1])} \alg{X}'(X) \simeq \coChev^{\un}(\liealg{2}^{(2)}(X)) =: \alg{2}^{(2)}(X),
\]
which could be viewed as an analog of $\FI$-hyperhomology. As above, the support of $\liealg{2}^{(2)}(X)$ is obtained from that of $\liealg{2}'(X)$ by deleting the point at $(c, k) = (1, 1)$.

We can repeat this procedure until we exhaust the line $k=1$. Once this happens, we move onto the line $k=2$, work iteratively from left to right like the above, and repeat the procedure for $k=3, 4, \dots$. In terms of Figure~\ref{fig:iterative_rep_stability}, at each step, the effect of taking derived indecomposables with respect to an action induced by a piece at coordinate $(c, k)$ corresponds to, on the $\coLie$ side, to removing the piece at $(c, k)$ (up to a shift to the left by $1$, since Figure~\ref{fig:iterative_rep_stability} shifts everything to the right by $1$). The support of the resulting object is hence also easy to read off. Indeed, if we let $\matheur{Q} = \coChev^{\un} \mathfrak{q}$ be the resulting object at some step, then the support of $\matheur{Q}$ is the smallest cone through the origin containing the support of $\mathfrak{q}[-1]$.

\begin{thm}
  \label{thm:iterative_procedure_conf}
  There is an iterative procedure of producing actions on $\alg{2}(X)$ and taking derived indecomposables with respect to these actions which generalizes the $\FI$-module structure on $\alg{2}(X)$ and the process of taking $\FI$-hyperhomology. Moreover, at each step, the cohomological support of the resulting object can be explicitly described and the dimension of any cohomological and graded degree piece is finite.
\end{thm}

Since it is easier and more illuminating to describe (which we did above) than to give explicit formulae, we omit the discussion of the precise formulae here. However, to demonstrate the simplicity of the method, we explicitly spell out the case where $k=1$ (the bottom orange strip in Figure~\ref{fig:iterative_rep_stability}) in the Corollary below.

\begin{cor} \label{cor:factoring_out_first_line_k=1}
  Let $X$ be an equidimensional scheme of dimension $d$, $c_0 \geq 0$ an integer, and $b=\min\left(\frac{2d-1}{2}, c_0 + 1\right)$. Let $\alg{2}^{(c_0)}(X) = \coChev^{\un} (\liealg{2}^{(c_0)}(X))$ be the object obtained from iteratively taking derived decomposables with respect to actions on $\alg{2}(X)$ by
  \[
    \Sym (\Ho^0(X, \sOmega_X)_1), \Sym (\Ho^1(X, \sOmega_X)[-1]_1), \dots, \Sym (\Ho^{c_0}(X, \sOmega_X)[-c_0]_1). \teq\label{eq:higher_dim_stab_algs}
  \]
  Then, $\Ho^c(\alg{2}^{(c_0)}(X))_k \simeq 0$ when $c < bk$. In general, $\Ho^c(\alg{2}^{(c_0)}(X))_k$ is finite dimensional.
\end{cor}

\subsubsection{}
Besides~\cite{gan_linear_2019}, not much is known about the link between taking derived indecomposables and higher representation stability. We expect that once analogs of~\cite{gan_linear_2019} are established, these results could be used to prove stronger statements about higher representation stability.

\subsubsection{The $i$-acyclic case}
\label{subsubsec:first_flavor_i-acyclic}
In the case where $\alg{2}(X)$ is free as a module over the algebras listed in~\eqref{eq:higher_dim_stab_algs}, such as when $X$ is $i$-acyclic, we obtain a more traditionally looking representation stability statement, which we now turn our attention to.

To start, note that taking derived indecomposables of a free module is particularly simple: suppose $\alg{2}(X) \simeq \matheur{O} \otimes \matheur{G}$, i.e. $\matheur{A}$ is a free module, generated by $\matheur{G}$, over a (twisted) commutative algebra $\matheur{O}$, then taking derived indecomposables with respect to $\matheur{O}$ simply gives $\matheur{G}$, the generators. In particular, the derived nature of taking derived indecomposables goes away.

Moreover, traditionally, higher representation stability is usually considered in the case where $\matheur{O} = \Sym \mathfrak{o}$ for some $\mathfrak{o}$ which concentrates in one cohomological and graded degree. An $\matheur{O}$-module structure could thus intuitively be thought of as a way to repeatedly multiply by elements of $\mathfrak{o}$.

Also traditionally, (higher) representation stability usually considers only modules rather than chain complexes thereof. More specifically, let $(c_{\mathfrak{o}}, k_{\mathfrak{o}})$ denote the cohomological and graded degrees where $\mathfrak{o}$ is supported. Then, for each cohomological degree $i$, one usually considers
\[
  \bigoplus_{t=0}^\infty \Ho^{i + tc_\mathfrak{o}}(\alg{2}(X))_{1+tk_\mathfrak{o}} \teq\label{eq:shifted_to_heart}
\]
as a module over $\matheur{O}' = \Sym \Ho^{c_\mathfrak{o}}(\mathfrak{o})$ or $\matheur{O}' = \Alt \Ho^{c_\mathfrak{o}}(\mathfrak{o})$ (the free twisted anti-commutative algebra generated by $\Ho^{c_\mathfrak{o}}(\mathfrak{o})$) depending whether $c_\mathfrak{o}$ is even or odd (see also footnote~\ref{ftn:Sym_contains_both_sym_and_anti-sym}). Moreover, if $\alg{2}(X)$ is free over $\matheur{O}$ then for each $i$, the same is true for~\eqref{eq:shifted_to_heart} as a module over $\matheur{O}'$. Thus, taking derived indecomposables, which a priori needs the whole chain complex, could be computed using just objects of the form~\eqref{eq:shifted_to_heart}. This point of view, however, leads to somewhat complicated notation since $c_\mathfrak{o}$ changes at each step in the iteration. We thus will not adopt this perspective and the notation that arises from it in the paper.

\subsubsection{}
This procedure can be visualized in terms of Figure~\ref{fig:iterative_rep_stability} as follows. Suppose we are interested in studying the action given by multiplying with $\Ho^{c_\mathfrak{o}}(X, \sOmega_X)[c_\mathfrak{o}]$ on $\alg{2}(X)$, then the space constructed in~\eqref{eq:shifted_to_heart} is obtained from taking the direct sum of all the cohomology lying on the line going through $(c, k) = (i, 1)$ that is parallel to the line $c=kc_\mathfrak{o}$, i.e. the one through the origin and the point $(c, k) = (c_\mathfrak{o}, 1)$.

Now, for $c_0 \geq 0$, let $\alg{2}^{(c_0 -1)}(X)$ be obtained from taking derived indecomposables with respect to actions of
\[
  \Sym (\Ho^0(X, \sOmega_X)_1), \Sym (\Ho^1(X, \sOmega_X)[-1]_1), \dots, \Sym (\Ho^{c_0-1}(X, \sOmega_X)[-c_0+1]_1)
\]
from $\alg{2}(X)$. Then, as above, for any $i\geq 0$,
\[
  \bigoplus_{t=0}^\infty \Ho^{i + tc_0}(\alg{2}^{(c_0-1)}(X))_{1+t} \teq\label{eq:mod_i_c_0}
\]
is naturally equipped with the structure of a free module over $\matheur{O}_{c_0} = \Sym \Ho^{c_0}(X, \sOmega_X)_1$ or $\matheur{O}_{c_0} = \Alt \Ho^{c_0}(X, \sOmega_X)_1$ depending on whether $c_0$ is even or odd, and moreover, the generators of this module is given by
\[
  \bigoplus_{t=0}^\infty \Ho^{i + tc_0}(\alg{2}^{(c_0)}(X))_{1+t}. \teq\label{eq:gen_i_c_0}
\]
Directly from Corollary~\ref{cor:factoring_out_first_line_k=1}, we see that when $c_0 < \frac{2d-1}{2}$, the direct sum at~\eqref{eq:gen_i_c_0} is in fact finite, i.e. the $k$-th summand vanishes when $k \gg 0$. In other words, the module written at~\eqref{eq:mod_i_c_0} is finitely generated. Using the bound $b$ in Corollary~\ref{cor:factoring_out_first_line_k=1}, we obtain the following.

\begin{cor} \label{cor:conf_spaces_k=1_i-acyclic}
  Let $X$ be an equidimensional scheme of dimension $d$ such that the multiplication structure on $C^*_c(X, \Lambda)$ is trivial (e.g. when $X$ is $i$-acyclic), and $c_0, i \geq 0$ be integers such that $c_0 < \frac{2d-1}{2}$. Then, $\bigoplus_{t=0}^\infty \Ho^{i + tc_0}(\alg{2}^{(c_0-1)}(X))_{1+t}$ is a finitely generated free module over $\Sym \Ho^{c_0}(X, \sOmega_X)_1$ (resp. $\Alt \Ho^{c_0}(X, \sOmega_X)_1$) when $c_0$ is even (resp. odd), where $\alg{2}^{(c_0-1)}$ is defined in Corollary~\ref{cor:factoring_out_first_line_k=1}. Moreover, the generators live in graded degrees $k$ with
  \begin{myenum}{(\roman*)}
  \item $1\leq k\leq i-c_0$, when $c_0 < d-1$, and
  \item $1\leq k\leq 2(i-d+1)$, when $c_0 = d-1$.
  \end{myenum}
\end{cor}
\begin{proof} The free generators of $\bigoplus_{t=0}^\infty \Ho^{i + tc_0}(\alg{2}^{(c_0-1)}(X))_{1+t}$, given by $\bigoplus_{t=0}^\infty \Ho^{i + tc_0}(\alg{2}^{(c_0)}(X))_{1+t}$, are obtained by taking the direct sum of cohomology along the line $c = i+c_0(k-1)$.

  When $c_0 + 1 \leq \frac{2d-1}{2}$, which is equivalent to saying that $c_0 < d - 1$ since $c_0$ is an integer,
  Corollary~\ref{cor:factoring_out_first_line_k=1} implies that the support of $\alg{2}^{(c_0)}(X)$ lives below the line $c = (c_0+1)k$. The $k$-coordinate of the intersection between these two lines is $k=i-c_0$ which resolves the first part.

  For the second part, also from Corollary~\ref{cor:factoring_out_first_line_k=1}, we know that the support of $\alg{2}^{(c_0)}(X)$ lives below the line $c=\frac{2d-1}{2}k$. Using the fact that $c_0 = d-1$, we see that the $k$-coordinate of the intersection is given by $2(i-d+1)$. The second part is thus also resolved.
\end{proof}

\subsubsection{Generalized configuration spaces}
We conclude this subsection with the case of generalized configuration spaces. As we have seen in Lemma~\ref{lem:homological_est_a_n}, for $n>2$, we know less about $\liealgOr{n}$, and hence also about $\liealg{n}(X)$, than about $\liealgOr{2}$ and $\liealg{2}(X)$. However, $\liealg{2}(X)$ shares two important features with $\liealg{n}(X)$ which could be used to establish analogs of the results above.

First of all, $\liealg{2}(X)[-1]_1$ is the same as $\liealg{n}(X)[-1]_1$. Thus, the iterative process described above still makes sense if we restrict ourselves to the actions induced by the first orange strip in Figure~\ref{fig:iterative_rep_stability}. Secondly, the last part of Lemma~\ref{lem:homological_est_a_n} provides us with the necessary upper bound as in Corollary~\ref{cor:factoring_out_first_line_k=1} and Corollary~\ref{cor:conf_spaces_k=1_i-acyclic}. Thus, arguing as before, we obtain the following results which are analogs of these Corollaries.

\begin{prop}
  \label{prop:first_flavor_gen_conf}
  Let $X$ be an equidimensional scheme of dimension $d$, $c_0\geq 0$, $n\geq 2$ be integers, and
  \[
    b=\min\left(\frac{(2d-1)(n-1)}{n} , c_0+1\right).
  \]
  Let $\alg{n}^{(c_0)}(X) = \coChev^{\un}(\liealg{n}^{(c_0)}(X))$ be the object obtained from iteratively taking derived indecomposables with respect to actions on $\alg{n}(X)$ by
  \[
    \Sym (\Ho^0(X, \sOmega_X)_1), \Sym(\Ho^1(X, \sOmega_X)[-1]_1), \dots, \Sym(\Ho^{c_0}(X, \sOmega_X)[-c_0]_1).
  \]
  Then $\Ho^c(\alg{n}^{(c_0)}(X))_k \simeq 0$ when $c<bk$. In general, $\Ho^c(\alg{n}^{(c_0)}(X))_k$ is finite dimensional.
\end{prop}
\begin{proof}
  Immediate from the discussion above.
\end{proof}

\begin{prop}
  \label{prop:first_flavor_gen_conf_degree_wise}
  Let $X$ be an equidimensional scheme of dimension $d$ such that $C^*_c(X, \Lambda)$ has trivial multiplication structure (e.g. when $X$ is $i$-acyclic), and $c_0, i\geq 0$, $n\geq 2$ be integers such that $c_0<\frac{(2d-1)(n-1)}{n}$. Then, $\bigoplus_{t=0}^\infty \Ho^{i+tc_0}(\alg{n}^{(c_0-1)}(X))_{t+1}$ is a finitely generated free module over $\Sym \Ho^{c_0}(X, \sOmega_X)_1$ (resp. $\Alt \Ho^{c_0}(X, \sOmega_X)_1$) when $c_0$ is even (resp. odd). Moreover, the generators live in graded degrees $k$ with
  \begin{myenum}{(\roman*)}
  \item $1\leq k\leq i - c_0$, when $c_0 + 1 \leq \frac{(2d-1)(n-1)}{n}$, and

  \item $1\leq k \leq \left(\frac{(2d-1)(n-1)}{n}-c_0\right)^{-1}(i-c_0)$, otherwise.
  \end{myenum}
\end{prop}
\begin{proof}
  As in the proof of Corollary~\ref{cor:conf_spaces_k=1_i-acyclic}, the upper bound of $k$ in the first part is given by the $k$-coordinate of the intersection of the lines given by $c=i+c_0(k-1)$ and $c=(c_0+1)k$, which, as above, is $k=i-c_0$, resolving the first part.

  For the second part, we need to find the intersection between the lines given by $c=i+c_0(k-1)$ and $c=\frac{(2d-1)(n-1)}{n}k$. Working out the $k$-coordinate, we obtain the second part.
\end{proof}

\subsubsection{Finite generation for higher representation stability}
In general, it is not known how to relate bounds on derived indecomposables with higher representation stability ranges. The only known case is~\cite{gan_linear_2019}, which was exploited in~\cite{miller_$mathrmfi$-hyperhomology_2019} to yield the best known stability range for (primary) representation stability of ordered configuration spaces of manifolds. We expect that similar results as in~\cite{gan_linear_2019} are possible and might be used alongside with the bounds established above to yield precise ranges for higher representation stability. For the time being, we content ourselves with the following more qualitative result.

\begin{thm}
  \label{thm:finite_generation_higher_rep_stab}
  Let $X$ be an equidimensional scheme of dimension $d$, and $c_0, i\geq 0, n\geq 2$ be integers such that $c_0 < \frac{(2d-1)(n-1)}{n}$. Then $\bigoplus_{t=0}^\infty \Ho^{i+tc_0}(\alg{n}^{(c_0-1)}(X))_{t+1}$ is a finitely generated module over $\Sym \Ho^{c_0}(X, \sOmega_X)_1$ (resp. $\Alt \Ho^{c_0}(X, \sOmega_X)_1$) when $c_0$ is even (resp. odd).

  More concretely, when $t\gg 0$, the map
  \[
    \Ho^{c_0}(X, \sOmega_X)_1 \otimes \Ho^{i+(t-1)c_0}(\alg{n}^{(c_0-1)}(X))_{t} \to \Ho^{i+tc_0}(\alg{n}^{(c_0-1)}(X))_{t+1}
  \]
  is surjective. Here, the tensor on the LHS is the one in $\Vect^{\FB}$.
\end{thm}
\begin{proof}
  From Proposition~\ref{prop:computing_coChev_as_limit}, we know that there exists a spectral sequence of twisted commutative algebras converging to~$\Ho^*(\alg{n}^{(c_0-1)})$ with the first page being $E_1^{p,q} = \Ho^{p+q}(\Sym^p (\liealg{n}^{(c_0-1)}(X)[-1]))$. In other words, the cohomology groups in page $1$ is exactly the situation encountered in Proposition~\ref{prop:first_flavor_gen_conf_degree_wise}, and hence, we obtain finite generation statements for $\Sym \Ho^{c_0}(X, \sOmega_X)_1$-, resp. $\Alt \Ho^{c_0}(X, \sOmega_X)_1$-, modules appearing on the first page.

  The corresponding finite generation statement for page infinity follows from the fact that the (abelian) categories of modules over a free twisted commutative/anti-commutative algebras generated in graded degree $1$ are locally Noetherian. The commutative case (a.k.a. the $\FI_d$ case) follows from~\cite{sam_grobner_2016}*{Corollary 7.1.5}. The anti-commutative case follows from Lemma~\ref{lem:rohit_nagpal} below.

  The second statement is just a rephrase of the first.
\end{proof}

The following lemma establishes local Noetherianity of the abelian category of modules over a free twisted anti-commutative ring. The statement and proof work in the abelian (as opposed to derived) setting. Thus, to avoid clustered notation, for this lemma, we will adopt the convention that all categories that appear are abelian. For example, we use $\Vect^{\FB}$ to denote the abelian category of $\FB$-modules, and for any twisted algebra $A \in \Vect^{\FB}$, $\Mod_A(\Vect^{\FB})$ to denote the abelian category of modules over $A$.

\begin{lem}\label{lem:rohit_nagpal}
  Let $V$ be a finite dimensional vector space over $\Lambda$ and $A = \Alt V_1$, the free twisted anti-commutative algebra generated by $V$ placed in graded degree $1$ in $\Vect^{\FB}$. Then, the (abelian) category $\Mod_A(\Vect^{\FB})$ of modules over $A$ is locally Noetherian.
\end{lem}

We thank R. Nagpal for help with the following proof.

\begin{proof}
  By Schur--Weyl duality, see~\cite{sam_introduction_2012}*{Thm. 5.4.1}, $\Mod_A(\Vect^{\FB})$ is equivalent to $\Mod_B(\Rep^{\poly}(\GL_\infty))$ where $B = \Alt(\Lambda^{\infty} \otimes V)$ and where $\Rep^{\poly}(\GL_\infty)$ is defined in~\cite{sam_introduction_2012}*{\S5.2}. By the transpose duality, see~\cite{sam_introduction_2012}*{\S7.4}, we see that $\Mod_B(\Rep^{\poly}(\GL_\infty))$ is equivalent to $\Mod_{B'}(\Rep^{\poly}(\GL_\infty))$ where $B' = \Sym(\Lambda^{\infty} \otimes V)$. Applying Schur--Weyl duality again, we see that $\Mod_{B'}(\Rep^{\poly}(\GL_\infty))$ is equivalent to $\Mod_{\Sym(V)}(\Vect^{\FB})$. Since the latter is locally Noetherian by~\cite{sam_grobner_2016}*{Corollary 7.1.5}, so is the former and we are done.
\end{proof}

\subsection{Higher representation stability: the $\TopTriv_m$ case}
\label{subsec:higher_rep_stab_TopTrivm}
We will next turn to the blue strip in Figure~\ref{fig:iterative_rep_stability}, assuming that $X$ satisfies $\TopTriv_m$ for some $m\geq 1$ (Figure~\ref{fig:iterative_rep_stability} above illustrates the situation where $X$ satisfies $\TopTriv_3$). This is an analog of the type of secondary representation stability obtained by introducing a pair of orbiting points near the boundary of an open manifold studied in~\cite{miller_higher_2016}.

\subsubsection{}
By Proposition~\ref{prop:splitting_n=2} we have a splitting in $\coLie(\Vect^{\FBplus})$
\[
  \liealg{2}(X) \simeq \liealgOp{2, \leq m}(X) \oplus \liealgGen{2, \leq m}(X) \simeq \bigoplus_{k=1}^m \liealgOp{2, k}(X) \oplus \liealgGen{2, \leq m}(X).
\]
The blue strip illustrates the support of $\liealgOp{2, \leq m}(X)[-1]$ while the support of $\liealgGen{2, \leq m}(X)[-1]$ consists of points marked by $\times$ lying between $L_B$ and $L_1$ and outside of the blue strip. Hence
\[
  \coChev^{\un}(\liealg{2}(X)) \simeq \coChev^{\un}(\liealgOp{2, \leq m}(X)) \otimes \coChev^{\un}(\liealgGen{2, \leq m}(X)),
\]
which means that $\alg{2}(X) \simeq \coChev^{\un}(\liealg{2}(X))$ is a free module over $\coChev^{\un}(\liealgOp{2, \leq m}(X))$ generated by $\coChev^{\un}(\liealgGen{2, \leq m}(X))$. Moreover, as discussed above, the cohomological support of $\coChev^{\un}(\liealgGen{2, \leq m}(X))$ lies in the smallest cone through the origin containing the support of $\liealgGen{2, \leq m}(X)[-1]$.

For the purpose of representation stability, we are mainly interested in the upper bound of the graded degree of $\coChev^{\un}(\liealgGen{2, \leq m}(X))$. Indeed, let $s\geq 0$ be such that $\Ho^k(X, \sOmega_X) = 0$ for all $k \in (0, s)$ and
\[
  b = \min\left(\frac{(2d-1)m}{m+1}, s\right)
\]
then it is easy to see that the support of $\coChev^{\un}(\liealgGen{2, \leq m}(X))$ lies below (and possibly including) the line $c = bk$. In other words, for each cohomological degree $c$, $\Ho^c(\coChev^{\un}(\liealgGen{2, \leq m}(X)))$ lives in graded degrees $\leq c/b$.

\subsubsection{}
The twisted commutative algebra $\coChev^{\un}(\liealgOp{2, \leq m}(X))$ itself can be explicitly described in terms of the cohomology of the ordered configuration spaces of $\mathbb{A}^d$. When $X$ is irreducible, by Remark~\ref{rmk:liealgOp_irred_case} and Corollary~\ref{cor:a_2_vs_coh_P_2}, we have
\[
  \liealgOp{2, \leq m}(X) \simeq \bigoplus_{k=1}^m (\liealgOr{2}[2d](d))_k \simeq \bigoplus_{k=1}^m \Ho^{(2d-1)(k-1)}(P^k_2(\mathbb{A}^d))[-(2d-1)(k-1)+1]
\]
and hence,
\[
  \coChev^{\un}(\liealgOp{2, \leq m}(X)) \simeq \Sym(\liealgOp{2, \leq m}(X)[-1]) \simeq \bigotimes_{k=1}^m \Sym(\Ho^{(2d-1)(k-1)}(P^k_2(\mathbb{A}^d))[-(2d-1)(k-1)]_k).
\]
Note that this is the same as~\cite{miller_higher_2016}*{Definition 3.29}.

The discussion above thus leads to the following theorem.

\begin{thm} \label{thm:stability_Tm}
  Let $X$ be an irreducible scheme of dimension $d$ which satisfies $\TopTriv_m$ for some $m\geq 1$. Then the cohomology of the ordered configuration spaces of $X$, $\alg{2}(X) = \bigoplus_{k=0}^\infty C^*(P^k_2(X), \sOmega_{P^k_2(X)})$, is a free module over the free twisted commutative algebra
  \[
    \algOp{2, \leq m}(X) = \bigotimes_{k=1}^m \Sym(\Ho^{(2d-1)(k-1)}(P^k_2(\mathbb{A}^d))[-(2d-1)(k-1)]_k),
  \]
  generated by $\algGen{2, \leq m}(X) = \coChev^{\un}(\liealgGen{2, \leq m}(X))$, with $\liealgGen{2, \leq m}(X)$ given in Proposition~\ref{prop:splitting_n=2}, i.e.\footnote{As before, the tensor product is taken in $\Vect^{\FB}$.}
  \[
    \alg{2}(X) \simeq \algOp{2, \leq m}(X) \otimes \algGen{2, \leq m}(X).
  \]
  Moreover, if $s\geq 1$ be such that $\Ho^k(X, \sOmega_X) = 0$ for all $k \in (0, s)$ and
  \[
    b = \min\left(\frac{(2d-1)m}{m+1}, s\right),
  \]
  then the support of $\algGen{2, \leq m}(X)$ lies below (and possibly including) the line $c=bk$ where $c$ and $k$ denote cohomological and graded degrees respectively. In other words, $\Ho^c(\algGen{2, \leq m}(X))$ lives in graded degrees $\leq c/b$.
\end{thm}

The corollary below follows immediately from Theorem~\ref{thm:stability_Tm}.

\begin{cor} \label{cor:Tm_image_spanned}
  Let $X$ be an irreducible scheme of dimension $d$ which satisfies $\TopTriv_m$ for some $m \geq 1$. Then, for each $c, k \geq 0$ such that $k > c/b$ (with $b$ being given in Theorem~\ref{thm:stability_Tm}) $\Ho^c(\alg{2}(X))_k$ is spanned by image of actions of $\algOp{2, \leq m}(X)$ on lower graded and cohomological degree pieces of $\Ho^*(\alg{2}(X))$.
\end{cor}

The $m=1$ case of the Corollary above is a generalization of~\cite{miller_higher_2016}*{Theorem 3.27} and~\cite{church_fi-modules_2015} beyond the case where $X$ is a manifold (see also~\S\ref{subsubsec:use_scheme_as_default}). Indeed, the following is a direct consequence of the Corollary above.

\begin{cor} \label{cor:T1_image_spanned}
  Let $X$ be a non-proper irreducible scheme of dimension $d$, i.e. $X$ satisfies $\TopTriv_1$. Then, for each $c\geq 0$, $\Ho^c(\alg{2}(X))$ is, as an $\FI$-module, finitely freely generated and has generators in degrees $\leq 2c$ when $d=1$ and $\leq c$ when $d>1$.
\end{cor}

\subsubsection{}
From the above, we see, in particular, that when $X$ satisfies $\TopTriv_m$, $\alg{2}(X)$ is a free module over
\[
  \algOp{2, k}(X) = \Sym(\Ho^{(2d-1)(k-1)}(P^k_2(\mathbb{A}^d))[-(2d-1)(k-1)]_k)
\]
for each $k$ such that $1 \leq k \leq m$. Usually in (higher) homological/representation, the actions of these algebras are considered iteratively as we mentioned at the beginning of this subsection. This can also be done here: one first considers the action of $\algOp{2, 1}(X)$, which equips $\alg{2}(X)$ with the structure of a free $\FI$-module. $\FI$-homology of $\alg{2}(X)$, which is given by $\Lambda_0 \otimes_{\algOp{2, 1}(X)} \alg{2}(X)$,\footnote{\label{ftn:hyper-FI-homology_vs_FI-homology} Strictly speaking, this computes the $\FI$-hyperhomology. However, since the module involved is free, $\Ho^*(\Lambda_0 \otimes_{\algOp{2, 1}(X)} \alg{2}(X))$ computes precisely the $\FI$-homology of $\Ho^*(\alg{2}(X))$.}
and which measures the failure of representation stability, then inherits the structure of a free module over $\algOp{2, 2}(X)$. One can iterate this process and consider
\[
  \Lambda_0 \otimes_{\algOp{2, m}(X)} \Lambda_0 \otimes_{\algOp{2, m-1}(X)} \cdots \otimes_{\algOp{2, 1}(X)} \alg{2}(X),
\]
which could be thought of as taking higher analogs of $\FI$-homology iteratively.\footnote{Again, a priori, these compute higher analogs of $\FI$-hyperhomology rather than of $\FI$-homology. However, since the modules involved are free, the two coincide. See also footnote~\ref{ftn:hyper-FI-homology_vs_FI-homology}.} But the resulting object is easily seen to be
\[
  \Lambda_0 \otimes_{\algOp{2, \leq m} (X)} \alg{2}(X) \simeq \coChev^{\un}(\liealgGen{2, \leq m}) = \algGen{2, \leq m}(X).
\]
We thus obtain the following, which resolves~\cite{miller_higher_2016}*{Conjecture~3.31} in the case where $X$ satisfies $\TopTriv_m$.

\begin{cor} \label{cor:miller_wilson_conj}
  Let $X$ be an irreducible scheme of dimension $d$ which satisfies $\TopTriv_m$ for some $m \geq 1$. Then, for each $c, k \geq 0$ such that $k > c/b$ ($b$ is given in Theorem~\ref{thm:stability_Tm}),
  \[
    \Ho^c(\Lambda_0 \otimes_{\algOp{2, m}(X)} \Lambda_0 \otimes_{\algOp{2, m-1}(X)} \cdots \otimes_{\algOp{2, 1}(X)} \alg{2}(X))_k = 0.
  \]
\end{cor}

\begin{rmk}\label{rmk:numerology_vs_miller_wilson}
  Corollary~\ref{cor:miller_wilson_conj} might seem different from~\cite{miller_higher_2016}*{Conjecture 3.31}. However, these are mostly cosmetic and are similar to the differences between Corollary~\ref{cor:factoring_out_first_line_k=1} and Corollary~\ref{cor:conf_spaces_k=1_i-acyclic}. We will now comment on these differences in more details.

  \begin{myenum}{(\roman*)}
  \item The first difference is that instead of the chain $\alg{2}(X)$, $\Ho^*(\alg{2}(X))$ is used and the (co)homological degrees are kept tracked of throughout. However, the two yield the same result since the modules involved are free in our case.

  \item The second difference is that the numbers involved in our statement seem simpler than those of~\cite{miller_higher_2016}. This is due to the fact that there is a re-indexing of the objects involved in that paper: note the appearance of $|S|$ in the homological degree in their definition of $\matheur{W}[d]^M_i(S)$. Since their paper considers each cohomological degree separately, this is necessary to make sure that the result of a multiplication (which is used to witness higher representation stability) lands in the correct place.

  \item Finally, because of the re-indexing, the vanishing statement in the conclusion of the conjecture is stronger than ours. In terms of Figure~\ref{fig:iterative_rep_stability} above, our vanishing statement applies to vertical lines (i.e. fixed cohomological degrees) whereas theirs applies to a skew line with slope given by the line through the origin and the support of $\liealgOp{2, m}(X)[-1]$ (note that our $m$ is their $d$). This necessitates the need for extra vanishing assumption on the cohomology of $X$. It is easy to check that once we add this extra vanishing condition to Corollary~\ref{cor:miller_wilson_conj}, we also obtain this stronger vanishing result. See also the discussion in~\S\ref{subsubsec:first_flavor_i-acyclic}.
  \end{myenum}
\end{rmk}

\section*{Acknowledgments}
This work took the current shape when the author visited O. Randal-Williams, who explained to him the twisted commutative algebra point of view for $\FI$-modules, which fits perfectly with the formalism of ``non-commutative'' $\Ran$ space (which is now called twisted commutative $\Ran$ space in this paper) the author was developing. The author would like to thank Randal-Williams for his encouragement and many illuminating conversations, and the University of Cambridge for their hospitality.

The author thanks J. Miller for many helpful conversations around Theorem~\ref{thm:finite_generation_higher_rep_stab} and for many comments which help improve the introduction. The author also thanks R. Nagpal for informing him about the Noetherianity statements in~\cite{sam_grobner_2016} and the proof of Lemma~\ref{lem:rohit_nagpal}.

The author also thanks the anonymous referee for many helpful comments and suggestions which greatly helps improve the paper.

The author gratefully acknowledges the support of the Lise Meitner fellowship, Austrian Science Fund (FWF): M 2751, and the Hong Kong RGC GRF grants 16304923 and 16301324.

%\printbibliography
\bibliography{higher_rep_stab_conf}
\end{document}